\documentclass[12pt]{amsart}
\usepackage{amsmath, amsthm,amssymb,tikz,pgflibraryarrows}
\usepackage{latexsym}		
\usepackage{graphicx}		
\usepackage{enumerate}
\usepackage{fancyhdr}

\theoremstyle{plain}
\newtheorem{theorem}{Theorem}[section]
\newtheorem{corollary}[theorem]{Corollary}
\newtheorem{proposition}[theorem]{Proposition}
\newtheorem{lemma}[theorem]{Lemma}

\theoremstyle{definition}
\newtheorem{definition}[theorem]{Definition}

\theoremstyle{remark}
\newtheorem{remark}[theorem]{Remark}
\newtheorem{remarks}[theorem]{Remarks}
\newtheorem{example}[theorem]{Example}
\newtheorem{examples}[theorem]{Examples}

\newtheorem{notation}[theorem]{Notation}

\begin{document}
\title[Aperiodicity and Primitive ideals]{\boldmath{Aperiodicity and Primitive ideals of row-finite $k$-graphs}}
\author[Sooran Kang]{Sooran Kang}
\address{Sooran Kang, School  of Mathematics and Applied Statistics, University of Wollongong, NSW 2522, Australia}
\email{sooran@uow.edu.au}
\author[David Pask]{David Pask}
\address{David Pask, School  of Mathematics and Applied Statistics, University of Wollongong, NSW 2522, Australia}
\email{dpask@uow.edu.au}

\thanks{This research was supported by the Australian Research Council}

\begin{abstract}
We describe the primitive ideal space of the $C^{\ast}$-algebra of a row-finite $k$-graph with no sources when every ideal is gauge invariant.
We characterize which spectral spaces can occur, and compute the primitive ideal space of two examples. In order to do this we prove some new
results on aperiodicity. Our computations indicate that when every ideal is gauge invariant, the primitive ideal space only depends on the $1$-skeleton of the $k$-graph in question.
\end{abstract}

\date{\today}

\maketitle

\section{Introduction}

A $k$-graph (or higher rank graph) is a higher dimensional analog of a directed graph. The notion of a $k$-graph $\Lambda$ and its associated $C^*$-algebra $C^* ( \Lambda )$ was introduced in \cite{KP1} to provide a graphical
approach to the higher dimensional Cuntz-Krieger algebras introduced by Robertson and Steger in \cite{RS2}. Since then
$k$-graph algebras have been studied by many authors and have provided concrete examples of many classifiable $C^*$-algebras whose fine structure and
invariants may readily be computed (see  \cite{E,PRRS,RSY1,RSY2,RoS1,RoS2,SZ,SP},   amongst others). They also provide a fertile class of
examples for researchers in non-selfadjoint algebras (see \cite{DPY,DY,KP,Po,Ya} amongst others) and crossed products (see \cite{BR,Ex,FPS,PQR}).

One of the reasons for the interest in $k$-graph algebras is that they provide an important testing ground for more complicated mathematical structures such as topological graphs as in \cite{Kat1}, topological quivers as in \cite{MTom}, topological $k$-graphs as in \cite{Y} and Cuntz-Pimsner algebras themselves (see \cite{C,FR,Kat2,P} amongst others).

The purpose of this paper is to describe the primitive ideal space of a row-finite $k$-graph with no sources. This was done for row-finite $1$-graphs with no sources in \cite{BPRS}. Row-finite $k$-graphs with no sources have proved to be the most tractable class of $k$-graphs to study, and we restrict ourselves to this class to avoid the sort of technicalities which occur when analyzing the primitive ideal space of general $1$-graphs (see \cite{B,BHRS,HZ}). It is our hope that our techniques will shed light on similar computations for more complex $C^*$-algebras.

The ideal structure of the $C^*$-algebra of a row-finite $k$-graph $\Lambda$ with no sources is best understood under the hypothesis of aperiodicity on $\Lambda$, introduced in \cite{KP1} but significantly improved later in \cite{RoS1}. Under the aperiodicity hypothesis (see Definition~\ref{aperiodic}) the structure of the gauge invariant ideals of $C^* ( \Lambda )$ is completely understood (see \cite[Theorem 5.2]{RSY1}): The lattice of gauge invariant ideals is isomorphic to the lattice of saturated hereditary sets of vertices of $\Lambda$. To begin our analysis, we first seek a condition under which all ideals are gauge invariant. This was first done in \cite[Theorem 7.2]{S} using a condition called (D), which we prefer to call strong aperiodicity (see Definition~\ref{def:sap}). In Definition~\ref{def:aq} we introduce the notion of an aperiodic quartet at a vertex. We then prove new results about aperiodicity in order to give necessary conditions on a $2$-graph to be strongly aperiodic in Proposition~\ref{prop:sap}.

Following \cite{B,BPRS,BHRS,HZ} the primitive ideals of a $1$-graph algebra are described in terms of a collection of vertices called a maximal tail. In Definition~\ref{Def-MT} we adapt the definition of a maximal tail into the context of row-finite $k$-graphs with no sources. Theorem~\ref{MT}
gives necessary and sufficient conditions on a saturated hereditary set of vertices to give rise to a primitive gauge invariant ideal of $C^* ( \Lambda )$ when $\Lambda$ is strongly aperiodic. Furthermore, in Theorem~\ref{topMT} we describe the topology of the primitive ideal space of $C^* ( \Lambda )$.

We then continue the analysis of \cite{B} and describe in Lemma~\ref{topMT2} an equivalent topology of the primitive ideal space of $C^* ( \Lambda )$.
Moreover, in Theorem~\ref{SS} we give a characterization of those topological spaces which can occur as the primitive ideal space of a row-finite $k$-graph with no sources: a spectral space in which the compact open sets form a countable base. Indeed, in Corollary~\ref{AF} we show that given a row-finite $k$-graph $\Lambda$ with no sources there is an AF algebra with the same primitive ideal space. Unlike in \cite[\S 5]{B} we have been unable to find an algorithm for producing such an AF algebra from the $1$-skeleton of $\Lambda$.

Finally, we give a detailed analysis of the primitive ideal space of two $2$-graphs using the main results in this paper. In Example~\ref{ex1} we
study the first graph $\Lambda$, a skew product graph, which we show, is strongly aperiodic using Proposition~\ref{prop:sap}. Then using Theorems~\ref{MT} and \ref{topMT}, we describe its primitive ideal space. During the example we show that $\Lambda$ cannot be
realized as a cartesian product graph. In Example~\ref{ex2} we consider a second $2$-graph $\Omega \times \Omega$, the cartesian product of an aperiodic $1$-graph $\Omega$ with itself. By Theorem~\ref{prod} (ii), the $2$-graph $\Omega \times \Omega$ is also strongly aperiodic. Since $C^* ( \Omega \times \Omega ) \cong C^* ( \Omega ) \otimes C^* ( \Omega )$,  the primitive ideal space of $C^* ( \Omega \times \Omega )$ is the cartesian product of primitive ideal space of $C^* ( \Omega )$ with itself with the product topology. It turns out that $C^* ( \Omega \times \Omega )$ and $C^* ( \Lambda )$ have homeomorphic primitive ideal spaces, even though $\Lambda$ is not a cartesian product.

\section{Preliminaries}

Let $\mathbb{N}^k$ denote the monoid of $k$-tuples of natural numbers under addition, and denote the canonical generators by $e_1,\dots,e_k$. We write $n_i$ for the $i$th coordinate $n\in\mathbb{N}^k$. For $m,n\in\mathbb{N}^k$, we say that $m\le n$ if $m_i\le n_i$ for each $i$. We write $m\vee n$ for the coordinate-wise maximum of $m$ and $n$ and $m \wedge n$ for the coordinate-wise minimum of $m$ and $n$.

\begin{definition}\cite[Definitions 1.1]{KP1}.
For $k \ge 1$, a $k$-graph $\Lambda$ consists of a countable category $\Lambda$ together with a functor $d:\Lambda\rightarrow \mathbb{N}^k$ satisfying the factorization property : for every $\lambda \in \Lambda$ and $m,n\in\mathbb{N}^k$ with $d(\lambda)=m+n$, there are unique elements $\mu , \nu \in\Lambda$ such that $\lambda=\mu\nu$ and $d(\mu)=m$, $d(\nu)=n$.
\end{definition}

\begin{examples} \label{ex:kex}
\begin{enumerate}[(a)]
\item The path category of a directed graph is a $1$-graph, and vice versa. In particular for $n \ge 1$ we denote by $B_n$ the path category of the directed graph consisting of a single vertex and $n$ edges.
\item Let $\operatorname{Mor} \Omega_k = \{ (m,n) \in \mathbb{N}^k \times \mathbb{N}^k : m \le n \}$, and $\operatorname{Obj} \Omega_k = \mathbb{N}^k$
then, when it is gifted with the structure maps $r (m,n) = m$,  $s (m,n) = n$, composition $(m,n)(n,p)=(m,p)$ and degree map $d (m,n) = n-m$,
one checks that $( \Omega_k , d )$ is a row-finite $k$-graph with no sources. We identify $\operatorname{Obj} \Omega_k$ with $\{ (m,m) : m \in \mathbb{N}^k \} \subset \operatorname{Mor} \Omega_k$.
\item For $n \ge 1$ let $\underline{n} = \{ 1 , \ldots , n \}$. For $m,n \ge 1$ let  $\theta : \underline{m} \times \underline{n} \to \underline{m} \times \underline{n}$ be a bijection. Let $\mathbb{F}^2_\theta$ be the $2$-graph with single vertex $v$ and edges $f_1 , \ldots , f_m , g_1 , \ldots , g_n$, with $d ( f_i ) = e_1$ for $i \in \underline{m} $, $d (g_j ) = e_2$ for $j \in \underline{n}$ and factorization rules $f_i g_j = g_{j'} f_{i'}$ where $\theta ( i , j ) = ( i' , j' )$ for $(i,j) \in \underline{m} \times \underline{n}$ (see \cite{DPY,DY,Po,Ya}).
\end{enumerate}
\end{examples}

\noindent For $E,F \subseteq \Lambda$ and $\lambda,\nu \in\Lambda$, we define $\lambda E:=\{\lambda\mu:\mu\in E, r(\mu)=s(\lambda)\}$ and $F\nu:=\{\mu\nu:\mu\in F, s(\mu)=r(\nu)\}$, $\lambda E \nu = \{ \lambda \mu \nu : \mu \in E , s ( \lambda) = r ( \mu ), s ( \mu ) = r ( \nu) \}$.

For $n \in \mathbb{N}^k$ let $\Lambda^n = d^{-1} (n)$, then by the factorization property we may identify $\Lambda^0  = d^{-1} (0)$ with the objects of $\Lambda$, and for this reason we call elements of $\Lambda^0$ vertices. In particular, for $v\in{\Lambda}^0$ and $n\in\mathbb{N}^k$,
\[
v{\Lambda}^n=\{\lambda\in\Lambda:r(\lambda)=v\;\text{and}\;d(\lambda)=n\} .
\]

\noindent
The $k$-graph $\Lambda$ is \textit{row-finite} if the set $v{\Lambda}^{m}$ is finite for each $m\in\mathbb{N}^k$ and $v\in{\Lambda}^0$. Also, $\Lambda$ has \textit{no sources} if $v{\Lambda}^{e_i}\ne \emptyset$ for all $v\in{\Lambda}^0$ and $i\in\{1,\dots,k\}$.

A morphism between two $k$-graphs $(\Lambda_1,d_1)$ and $(\Lambda_2,d_2)$ is a functor $f:\Lambda_1\rightarrow\Lambda_2$ satisfying $d_2(f(\lambda))=d_1(\lambda)$ for all $\lambda\in\Lambda_1$.

\begin{notation}(see \cite[\S 2]{RSY2})
For $0 \le m \le n \le d ( \lambda )$, by the factorization property we have $\lambda = \lambda ( 0 , m ) \lambda ( m , n ) \lambda ( n , d ( \lambda ) )$ where $d ( \lambda ( 0 , m ) ) = m$, $d ( \lambda (m ,n ) ) = n-m$ and $d ( \lambda ( n , d ( \lambda ) ) ) = d ( \lambda ) - n$.
\end{notation}

\noindent Though originally called a $*$-representation of $\Lambda$ in \cite[Definitions 1.4]{KP1}, nowadays
we call the relations satisfied by the generators of $C^* ( \Lambda )$ the Cuntz-Krieger relations.

\begin{definition}
Let $\Lambda$ be a row-finite $k$-graph with no sources. A \textit{Cuntz-Krieger} $\Lambda$-family in a $C^{\ast}$-algebra $B$ consists of a family of partial isometries $\{t_{\lambda}:\lambda\in\Lambda\}$ satisfying the \textit{Cuntz-Krieger relations}:
\begin{enumerate}[(a)]
\item $\{t_v:v\in{\Lambda}^0\}$ is a family of mutually orthogonal projections,
\item $t_{\lambda\mu}=t_{\lambda}t_{\mu}$ for all $\lambda,\mu\in \Lambda$ with $s(\lambda)=r(\mu)$,
\item $t^{\ast}_{\lambda}t_{\lambda}=t_{s(\lambda)}$,
\item $t_v=\sum_{\lambda\in v{\Lambda}^m} t_{\lambda}t^{\ast}_{\lambda}$ for all $v \in {\Lambda}^0$ and $m\in\mathbb{N}^k$.
\end{enumerate}
\end{definition}

\begin{remark} \label{rem:sep}
As mentioned in \cite{KP1,RSY1}, given a row-finite $k$-graph $\Lambda$ with no sources, there is a $C^{\ast}$-algebra $C^{\ast}(\Lambda)$ generated by a universal Cuntz-Krieger $\Lambda$-family $\{t_\lambda:\lambda\in\Lambda\}$ with $t_\lambda \neq 0$ for all $\lambda \in \Lambda$ on a separable Hilbert space.
\end{remark}

\noindent
By the universal property of $C^{\ast}(\Lambda)$, there is a strongly continuous action of the $k$-torus $\mathbb{T}^k$, called the gauge action, $\gamma:\mathbb{T}^k\rightarrow\text{Aut}\;C^{\ast}(\Lambda)$ defined for $z=(z_1,\dots,z_k)\in\mathbb{T}^k$ and $s_{\lambda}\in C^{\ast}(\Lambda)$ by $\gamma_z(s_{\lambda})=z^{d(\lambda)}s_{\lambda}$, where $z^m=z^{m_1}_1\dots z^{m_k}_k$ for $m=(m_1,\dots,m_k)\in\mathbb{N}^k$.

\begin{definition}\cite[Definitions 2.1]{KP1}. \label{def:aperiodicity}
Let $\Lambda$ be a row-finite $k$-graph with no sources. The set $\Lambda^\infty = \{ x : \Omega_k \to \Lambda : x \text{ is a $k$-graph morphism} \}
$ is called the \textit{infinite path space} of $\Lambda$. For $x \in \Lambda^\infty$ and $v \in \Lambda^0$, we put $v \Lambda^\infty = \{ x \in \Lambda^\infty : x(0,0)=v \}$.
\end{definition}

\begin{remark}
For $\lambda \in \Lambda$, let $Z ( \lambda ) = \{ x \in \Lambda^\infty : x ( 0 , d ( \lambda ) ) = \lambda \}$. Then
$\{ Z ( \lambda ) : \lambda \in \Lambda \}$ forms a basis of compact and open sets for a topology on $\Lambda^\infty$
. For $p \in \mathbb{N}^k$, the shift map $\sigma^p : \Lambda^\infty \to \Lambda^\infty$ defined by
$\sigma^p x(m,n) = x(m+p,n+p)$ for $(m,n) \in \Omega_k$ is a local homeomorphism (for more details see \cite[Remark 2.5, Lemma 2.6]{KP1}).
\end{remark}

\noindent
Since we will be dealing with gauge invariant ideals in $C^* ( \Lambda )$, we must work with $k$-graphs which satisfy certain aperiodicity conditions, such as the one given below:

\begin{definition}\cite[Lemma 3.2 (iv)]{RoS1}.\label{aperiodic}
Let $\Lambda$ be a row-finite $k$-graph with no sources. We say that $\Lambda$  has \emph{no local periodicity} at $v\in{\Lambda}^0$ if for each pair $m\ne n\in\mathbb{N}^k$, there is a path $\lambda\in v\Lambda$ such that $d(\lambda)\ge m\vee n$ and
\begin{equation} \label{eq:lp}
\lambda(m,m+d(\lambda)-(m\vee n))\ne \lambda(n,n+d(\lambda)-(m\vee n)).
\end{equation}

\noindent
The $k$-graph $\Lambda$ is said to be \textit{aperiodic} (or satisfies the \textit{aperiodicity condition}) if every vertex has no local periodicity.
\end{definition}

\noindent Recall that a loop is a path $\mu \in \Lambda$ such that $s ( \mu ) =  r ( \mu )$.  In \cite[\S 3]{KPR2} a directed graph is said to satisfy condition (L) if every loop has an exit, which is the analogue of condition (I) in \cite{CK}. So the analogue of condition (L) for $1$-graphs is that every loop has an entrance: given $\mu \in v \Lambda^p v$ for some $v \in \Lambda^0$ there is $\kappa \in v \Lambda^p$ with $\mu \neq \kappa$. The following result is common knowledge but we have not been able to find a proof.

\begin{lemma} \label{lem:L}
For row-finite $1$-graphs with no sources, the aperiodicity condition is equivalent to condition (L).
\end{lemma}

\begin{proof}
Suppose that a  $\Lambda$ is an aperiodic $1$-graph  with no sources and that $\mu \in v \Lambda^p v$ for some $v \in \Lambda^0$.
Let $m=0$ and $n =p$. Since there is no local periodicity at $v$ and there are no sources, we may find $\lambda \in v \Lambda^{tp}$ for some $t > 1$ such that
\[
\lambda ( 0 , t(p-1) ) = \lambda ( 0 ,  tp - p ) \neq \lambda ( p , p + tp - p ) = \lambda ( p , tp )
\]

\noindent which means that $\mu^t$ cannot be the only path with range $v$ of length $tp$ and hence $\mu$ must have an entrance.

Suppose that every loop in $\Lambda$ has an entrance. Choose $v \in \Lambda^0$ and $m \lneq n \in \mathbb{N}$. If $v$ is not connected
to a loop then \eqref{eq:lp} holds for every $\lambda \in v \Lambda$. Hence without loss of generality we may assume that $v$ is
connected to a loop; moreover also that $v$ lies on the loop, $\mu \in v \Lambda^p v$. Suppose that $p=1$, then since the loop has an entrance $e \neq \mu (0,1)$ the path $\lambda = \mu^{n} e$ satisfies \eqref{eq:lp}, so we may assume $p>1$. We may also, without loss of generality,
assume that the vertices $\mu ( q , q )$ are distinct for all $0 \le q \le p-1$. If $n-m$ is not a multiple of $p$ then $\lambda = \mu^t$
will satisfy \eqref{eq:lp} where $tp > n$ since $\lambda (m,m) \neq \lambda (n,n)$. Finally, suppose that $n = m + tp$ where $t \ge 1$ and let $q$ be such that $qp \le n < (q+1)p$. Since $\mu$ has an exit there must be a path $\kappa \in v \Lambda^p$ such that $\kappa \neq \mu$. Then one
checks that $\lambda = \mu^q \kappa$ satisfies \eqref{eq:lp}.
\end{proof}

\begin{examples} \label{ex:aperiodex}
For $n \ge 2$ the $1$-graphs $B_n$ are aperiodic. Since it has no loops it follows that $\Omega_1$ is aperiodic.
\end{examples}

\begin{remark}
An aperiodicity condition for $k$-graphs was first introduced in \cite[Definition 4.3]{KP1} in terms of infinite paths: A row-finite $k$-graph $\Lambda$ satisfies condition (A) if for every vertex $v\in{\Lambda}^0$, there is an infinite path $x\in v{\Lambda}^{\infty}$ such that ${\sigma}^m(x)\ne {\sigma}^n(x)$ for all $m\ne n\in\mathbb{N}^k$. There are several other definitions for aperiodicity in the literature such as condition (B) of \cite[Theorem 4.3]{RSY1} and the condition described in \cite[Definition 1]{RoS1}. It is shown in \cite[Lemma 3.2]{RoS1} that all these definitions are equivalent. The benefit of Definition \ref{aperiodic} is that it is given in terms of $\Lambda$ directly which makes it a lot easier to work with.
\end{remark}

\begin{definition} \label{def:sh}
Let $\Lambda$ be a row-finite $k$-graph with no sources. Define a relation $\le$ on ${\Lambda}^0$ by $v\le w$ if and only if $v\Lambda w\ne\emptyset$.
\begin{enumerate}[(a)]
\item We say that a subset $H$ of ${\Lambda}^0$ is \textit{hereditary} if $v\in H$ and $v\le w$ imply that $w\in H$.
\item \label{sat} We say that a set $H \subseteq \Lambda^0$ is \textit{saturated} if
for all $v \in \Lambda^0$
\[
r^{-1}(v)\ne \emptyset\;\; \text{and}\;\; \{s(\lambda)\;:\;\lambda\in v{\Lambda}^{e_i}\}\subset H\;\;\text{for some}\;\;i\in\{1,\dots,k\} \;\;\Longrightarrow \;\;v\in H.
\]
\end{enumerate}
The \textit{saturation} of a set $H$ is the smallest saturated subset $\overline{H}$ of ${\Lambda}^0$ containing $H$.
\end{definition}

\begin{remark}
The notion of a saturated hereditary set was first given in \cite{Cu} and adapted for directed graphs in \cite[\S 6]{KPRR}. The definition of a saturated set for a $k$-graph was first introduced in \cite[\S 5]{RSY1}, but the relation $\le$ is defined differently. Our definition is adapted from the one given in \cite[Definition 3.1]{S} which was given in a more general setting.
\end{remark}

\noindent
The following technical lemma is used in the proof of our main Theorem (Theorem~\ref{MT}). We use it to identify when a vertex lies in the saturation of a hereditary set of vertices. Its proof is very similar to the one for a directed graph given in \cite[Lemma 6.2]{BPRS}. We include its proof as we need an additional argument because of new saturation condition (\ref{sat}) of Definition~\ref{def:sh} for a $k$-graph.

\begin{lemma}\label{lem2}
Suppose that $\Lambda$ is a row-finite k-graph with no sources and $v\in {\Lambda}^0$. If $y\in\overline{\{x\in {\Lambda}^0\;:\;v\le x\}}$, then there exists $z\in {\Lambda}^0$ such that $v\le z$ and $y\le z$.
\end{lemma}
\begin{proof}
Let $L_v:=\{x\in {\Lambda}^0:v\le x\}$. First note that $L_v$ is hereditary since if $y\in L_v$ and $y\le z$ then $v\le y$, which implies that $v\le z$. Thus, $z\in L_v$. So its saturation is by definition the smallest saturated set containing $L_v$. Now suppose that $K$ is any saturated set containing $L_v$. Then $K_1:=\{w\in K : w\le x\;\;\text{for some}\;\;x\in L_v\}$ contains $L_v$. To show that $K_1$ is saturated, we suppose that $z\in {\Lambda}^0$ and $\{s(\lambda)\;:\;\lambda\in z{\Lambda}^{e_i}\}\subset K_1$ for some $i\in\{1,\dots,k\}$. Then $K_1\subset K$ by definition, so $z\in K$ since $K$ is saturated. We know that there exists $\lambda$ such that $r(\lambda)=z$ and $s(\lambda)\in K_1$, so $s(\lambda)\le x$ for some $x\in L_v$. Also, $z=r(\lambda)\le s(\lambda)$ implies that $z\le x$ for some $x\in L_v$. Hence, $z\in K_1$ and $K_1$ is saturated. Thus, if $K$ is the smallest saturated set containing $L_v$, then $K=\{w\in K : w\le x\;\;\text{for some}\;\;x\in L_v\}$.

If $y\in L_v$, then $v\le y$. So any $z\in {\Lambda}^0$ satisfying $y\le z$ will give the desired property. If $y\in\overline{L_v}\setminus L_v$, then $\{s(\lambda): \lambda\in y{\Lambda}^{e_i}\}\subset L_v$ for some $i\in\{1,\dots, k\}$. So there exists $z\in L_v$, i.e. $v\le z$ such that $y\le z$, which proves the statement.
\end{proof}

\begin{remark}
Let $C^{\ast}(\Lambda)$ be the $C^{\ast}$-algebra for a row-finite $k$-graph $\Lambda$ with no sources. For an ideal $I$ in $C^{\ast}(\Lambda)$, we let $H_I=\{v\in{\Lambda}^0\;:\;t_v\in I\}$. Also, for each subset $H$ of ${\Lambda}^0$, let $I_H$ denote the ideal in $C^{\ast}(\Lambda)$ generated by $\{t_v:v\in H\}$.
Then, following Lemma 4.3 in \cite{BPRS}, it can be shown that for a saturated hereditary subset $H$ of $\Lambda$,
\[I_H=\overline{\text{span}}\{t_{\alpha}t^{\ast}_{\beta}\;:\;\alpha,\beta\in\Lambda\;\;\text{and}\;\;s(\alpha)=s(\beta)\in H\}.\]
In particular, $I_H$ is gauge invariant in the sense that $\gamma_z(a)=a$ for all $a\in I_H$ and $z\in\mathbb{T}^k$. Moreover, every gauge-invariant ideal is of this form.
\end{remark}

\noindent
The following Theorem gives a complete description of the gauge invariant ideals of $C^* ( \Lambda )$.

\begin{theorem}\cite[Theorem 5.5]{S}, \cite[Theorem 5.2]{RSY1}. \label{ff}
Let $\Lambda$ be a row-finite $k$-graph with no sources. Let $H$ be a subset of ${\Lambda}^0$ and $I$ be an ideal in $C^{\ast}(\Lambda)$.
Let $I_H$ and $H_I$ be as defined above.
\begin{enumerate}[(a)]
\item If I is non-zero gauge invariant ideal of $C^{\ast}(\Lambda)$, then $H_I$ is non-empty saturated hereditary subset and $I=I_{H_I}$.
\item If H is a saturated and hereditary subset of ${\Lambda}^0$ and $I_H$ is the associated ideal, then $H=H_{I_H}$.
\item The map $H\mapsto I_H$ is an isomorphism of the lattice of saturated hereditary subsets of ${\Lambda}^0$ onto the lattice of closed gauge-invariant ideals of $C^{\ast}(\Lambda)$.
\item Suppose $H \neq \Lambda^0$ is saturated and hereditary. Then
\[\Gamma(\Lambda\setminus H):=({\Lambda}^0\setminus H,\{\lambda\in\Lambda:s(\lambda)\notin H\},r,s)\]
is a row-finite $k$-graph with no sources, and $C^{\ast}(\Lambda)/ I_H$ is canonically isomorphic to $C^{\ast}(\Gamma(\Lambda\setminus H))$.
\end{enumerate}
\end{theorem}

\begin{proof}
The proofs for the generalized case of (a) and (b) are given in \cite[Theorem 5.5]{S}. Also, (c) and (d) are shown in \cite[Theorem 5.2]{RSY1} for locally convex row-finite $k$-graphs, so they are certainly true for row-finite $k$-graphs with no sources.
\end{proof}

\begin{remark} \label{rem:trivcase}
Note that $\Lambda^0 , \emptyset$ are always saturated hereditary subsets of $\Lambda^0$, corresponding to the trivial ideals  $C^* ( \Lambda ) $, $\{ 0 \}$  of $C^* ( \Lambda )$ respectively. This observation allows us to complete the analysis of Theorem~\ref{ff} (d) for the case $H = \Lambda^0$.
\end{remark}

\section{The Primitive Ideal Space of $C^* ( \Lambda )$}

In this section, we describe and completely characterize the primitive ideal space of the $C^*$-algebra of a $k$-graph $\Lambda$ all of whose ideals are gauge invariant. Since $C^* ( \Lambda )$ is separable (see Remark~\ref{rem:sep}) it follows from \cite[Proposition 3.13.10 and Proposition 4.3.6]{Pe} that every primitive ideal is prime and vice versa. We shall therefore use the terms primitive and prime interchangeably.

We begin in Proposition~\ref{GI} where we give a condition on $\Lambda$ which ensures that all ideals are gauge invariant. Then, in Proposition~\ref{prop:sap} we show how it may be possible to check this condition.
Next in Theorem~\ref{MT} we give necessary and sufficient conditions on a gauge invariant ideal to be prime when all ideals are gauge invariant. Finally, in Theorem~\ref{topMT} we describe the topology of $\operatorname{Prim} C^* ( \Lambda )$.

\begin{definition} \label{def:sap}
Let $\Lambda$ be a row-finite $k$-graph with no sources. We say that $\Lambda$ is \textit{strongly aperiodic} if $\Gamma(\Lambda\setminus H)$ satisfies aperiodicity condition for all saturated hereditary subsets $H \subsetneq \Lambda^0$.
\end{definition}

\noindent Note that by taking $H= \emptyset$ we see that if a row-finite $k$-graph $\Lambda$ with no sources is strongly aperiodic, then it is automatically aperiodic.

A loop $\mu$ in a $1$-graph $\Lambda$ is said to be \textit{simple} if the vertices $\{ \mu ( i , i ) : 0 \le i \le d( \mu)-1 \}$ are distinct.
In \cite[\S 6]{KPRR}, a directed graph is said to satisfy condition (K) if every vertex is the range of either two distinct simple loops or none; that is for all $v \in \Lambda^0$ either $v \Lambda v = \emptyset$ or there are simple loops $\mu, \nu \in v \Lambda v$ such that $\mu \neq \nu$. When considering $1$-graphs, the analogue of condition (K) is the same. Parts of the proof of the following result can be found in \cite{BHRS} for directed graphs (see also \cite[Remark 6.11]{KPRR}). However, the result has not been stated in this form,  so we give the proof for completeness and to complement the proof of Lemma~\ref{lem:L}.

\begin{lemma} \label{lem:KisSA}
For row-finite $1$-graphs with no sources, strong aperiodicity is equivalent to condition (K).
\end{lemma}

\begin{proof}
Let $\Lambda$ be a strongly aperiodic $1$-graph. Suppose, for contradiction, that $\Lambda$ does not satisfy condition (K). Then there is a vertex $v\in\Lambda^0$ which has only one simple loop based at $v\in\Lambda^0$. Let $\mu\in v\Lambda^p v$ for some $p>0$. Let $X=\{s(\nu) \mid \nu\ne\mu, \nu\in v\Lambda^q\}$. It is clear that $v\notin X$ since $v$ only has one simple loop $\mu$. Let $L_X=\{u\in\Lambda^0 \mid x\le u\;\;\text{for some}\;\;x\in X\}$, then $L_X$ is hereditary. If $x\Lambda v\ne\emptyset$ for some $x\in L_X$, then $v$ is the range of two distinct loops and so $\Lambda$ satisfies condition (K), so $x\Lambda v=\emptyset$ for all $x\in L_X$. Since $v$ does not connect to any vertex in $L_X$, it follows that $v\notin\overline{L_X}$. Moreover, $\Gamma(\Lambda\setminus\overline{L_X})$ contains $v$ and $\mu$. But $v\Gamma(\Lambda\setminus\overline{L_X})^q=\{\mu\}$ by construction. Hence by Lemma \ref{lem:L}, the $1$-graph $\Gamma(\Lambda\setminus\overline{L_X})$ does not satisfy condition (L) and so is not aperiodic. This contradicts $\Lambda$ being strongly aperiodic. Thus strong aperiodicity implies condition (K).

Let $\Lambda$ be a $1$-graph which satisfies condition (K). Suppose for contradiction, that $\Lambda$ is not strongly aperiodic. Then there is a saturated
hereditary subset $H \subsetneq \Lambda^0$ such that $\Gamma ( \Lambda \backslash H  )$ is not aperiodic. By Lemma~\ref{lem:L} there is $v \in \Gamma ( \Lambda \backslash H )^0$ and $\mu \in v \Gamma ( \Lambda \backslash H  )^p v$ such that $| v \Gamma ( \Lambda \backslash H  )^q | = 1$ for all $q \ge 1$. Replacing $\mu$ by a sub-path if necessary we may, without loss of generality, assume that $\mu$ is a simple loop. Since $\Lambda$ itself satisfies condition (K) there is a simple loop $\nu \in v \Lambda^q v$ for some $q \ge 1$ such that $\mu \neq \nu$. By definition of $\Gamma ( \Lambda \backslash H )$ it follows that $\nu \in \Gamma( \Lambda \backslash H )$ since $v = s ( \nu ) \not\in H$. Suppose that $q < p$ then since $| \Gamma ( \Lambda \backslash H )^q | = 1$ it follows that $\mu ( 0 , q ) = \nu ( 0 ,q )$. Since $r ( \nu ) = \mu ( 0 , q ) = v$ it follows that $\mu$ is not simple and a contradiction. We obtain a similar contradiction if $p > q$. If $p=q$ then since  $| v \Gamma ( \Lambda \backslash H  )^q | = 1$ it follows that $\mu = \nu$ which contradicts the assumption that $\mu\neq \nu$. Hence in all cases, the assumption must have been false, and so $\Lambda$ is strongly aperiodic.
\end{proof}

\begin{proposition}\label{GI}
Let $\Lambda$ be a row-finite strongly aperiodic $k$-graph with no sources, then all ideals of $C^* ( \Lambda )$ are gauge invariant.
\end{proposition}

\begin{proof}
If $\Lambda$ is strongly aperiodic, then one can check that $\Lambda$ satisfies condition (D) of \cite[Definition 7.1]{S}. The result follows from \cite[Theorem 7.2]{S}.
\end{proof}

\noindent We now look for conditions on a $2$-graph which ensure that it is strongly aperiodic. First we need a definition.

\begin{definition} \label{def:aq}
Let $\Lambda$ be a $2$-graph and $u \in \Lambda^0$, suppose that there are distinct $\alpha_i \in u \Lambda^{a e_1} u$ and distinct
$\beta_i \in u \Lambda^{b e_2} u$ for $i=1,2$, where $a , b \in \mathbb{Z}^+$  such that
\[
\beta_2 \alpha_1 = \alpha_1 \beta_2 , \ \beta_2 \alpha_2 = \alpha_2 \beta_2 , \ \beta_1 \alpha_1 = \alpha_2 \beta_1 , \ \beta_1 \alpha_2 = \alpha_1 \beta_1 .
\]

\noindent Then $(\alpha_1,\alpha_2,\beta_1,\beta_2)$ is called an $(a,b)$-\emph{aperiodic quartet} at $u$.
\end{definition}

\begin{theorem} \label{thm:aq}
Let $\Lambda$ be a row-finite $2$-graph with no sources and $(\alpha_1,\alpha_2,\beta_1,\beta_2)$ be an $(1,1)$-aperiodic quartet at $u \in \Lambda^0$, then there is no local periodicity at $u$.
\end{theorem}

\begin{proof}
Fix $u \in \Lambda^0$, and let $m \neq n \in \mathbb{N}^2$. Suppose that $m \vee n \neq m$ or $n$.  Let $q_1,q_2 \ge 1$ be such that $q_1 \ge ( m \vee n )_1$
and $q_2 \ge ( m \vee n )_2$ and define $\lambda \in u \Lambda^{q}$ by
$\lambda = \alpha_1^{q_1} \beta_2^{q_2}$. Without loss of generality, suppose that $n_1 < m_1$  and choose $\mu = \lambda \alpha_2 \beta_2$. 
Because of the factorisation property $\beta_2\alpha_1=\alpha_1\beta_2$, we have $\mu=\alpha_1^{q_1}\beta_2^{q_2}\alpha_2\beta_2=\alpha_1^{m_1}\beta_2^{m_2}\alpha_1^{q_1-m_1}\beta_2^{q_2-m_2}\alpha_2\beta_2
=\alpha_1^{n_1}\beta_2^{n_2}\alpha_1^{q_1-n_1}\beta_2^{q_2-n_2}\alpha_2\beta_2$.
 Since $q_1 - m_1 < q_1 - n_1$ we have
\[
\mu ( m , m+(q_1-m_1+1) e_1 )=\alpha_2 \neq \alpha_1=\mu ( n , n + (q_1-m_1+1) e_1 ),
\]

\noindent  and so there is no local periodicity at $u$.

Now suppose that $m \vee n = m$ or $n$. Without loss of generality suppose that
$m \vee n = m$. Let $q_1, q_2 \ge 1$ be such that $q_1 \ge m_1$ and $q_2  \ge m_2$  and define $\lambda \in u \Lambda^q$ by $\lambda = \alpha_1^{q_1} \beta_2^{q_2}$. If $m > n$, then put $\mu = \lambda (\beta_1 \alpha_2)^{m_2-n_2}$, then 
\[
\mu ( m , m + (q_2-m_2+1) e_2)=\beta_1 \neq \beta_2 = \mu ( n , n + (q_2-m_2+1) e_2 )  ,
\]

\noindent and so there is no local periodicity at $u$. If  $m_1 = n_1$ but $m_2 > n_2$ let $\lambda \in u \Lambda^q$ be defined by $\lambda = \alpha_1^{q_1} \beta_2^{q_2}$. Let $\mu = \lambda ( \beta_1 \alpha_2)^{m_2-n_2}$, then 
\[
\mu ( n , n+(q_2-m_2+1)e_2) = \beta_2 \neq \beta_1 = \mu ( m , m +(q_2-m_2+1)e_2 ) ,
\]

\noindent and so there is no local periodicity at $u$.  If $m_2 =n_2$ but $m_1 > n_1$ then similar argument applies using $\mu = \lambda ( \beta_1 \alpha_1)^{m_1-n_1}$, then the proof  is complete.
\end{proof}

\noindent Once we have checked no local aperiodicity at a vertex, we are able to deduce that many other vertices have no local aperiodicity.

\begin{lemma} \label{lem:ishered}
Let $\Lambda$ be a row-finite $k$-graph with no sources. Suppose that there is no local periodicity at $u \in \Lambda^0$ and that $v \le u$, then there is no local periodicity at $v$.
\end{lemma}

\begin{proof}
Let $\kappa \in v \Lambda u$, and $m \neq n \in \mathbb{N}^k$. Let $p=m - m \wedge n$ and $q = n - m \wedge n$, then $p \neq q \in \mathbb{N}^k$.
Since $u$ has no local periodicity, there is $\mu \in u \Lambda^t$, where $t > p \vee q$
such that
\begin{equation} \label{eq:above}
\mu ( p , p + d (\mu) - p \vee q ) \neq \mu ( q , q + d ( \mu) - p \vee q ) .
\end{equation}

\noindent Let $\lambda = \kappa \mu$ then we have
\begin{align*}
\lambda ( d ( \kappa ) + m - m \wedge n , &\; d ( \kappa ) + m - m \wedge n + d (\mu) - p \vee q ) \\
&= \mu ( p , p + d (\mu) - p \vee q ) \\
& \neq \mu ( q , q + d ( \mu) - p \vee q ) \\
&= \lambda ( d ( \kappa ) + n - m \wedge n , d ( \kappa ) + n - m \wedge n + d (\mu) - p \vee q ) .
\end{align*}

\noindent So we have
\[
\lambda ( m , m + d ( \lambda )- m \vee n ) \neq \lambda ( n , n + d ( \lambda ) - m \vee n ),
\]

\noindent  since
\[
0 \le d ( \kappa )  - m \wedge n \le d ( \kappa ) - m \wedge n + d (\mu) - p \vee q ) \le  d ( \lambda )- m \vee n
\]

\noindent and the result follows.
\end{proof}

\noindent Pulling the previous two results together we have the following two Propositions.

\begin{proposition}\label{lem:2graph-ap}
Let $\Lambda$ be a row-finite $2$-graph with no sources. Suppose that every vertex is connected to a vertex with an aperiodic quartet, then $\Lambda$ is
aperiodic.
\end{proposition}

\begin{proof}
Follows by Theorem~\ref{thm:aq} and Lemma~\ref{lem:ishered}.
\end{proof}

\begin{example}
As seen in Example \ref{ex:kex} (c),
let $m,n \ge 1$ and $\theta : \underline{m} \times \underline{n}$ be a bijection
such that there are $i \neq i' \in \underline{m}$ and $j \neq  j' \in \underline{n}$
such that $\theta (i,j) = (i',j)$, $\theta (i',j)=(i,j)$, $\theta (i,j')=(i,j')$ and $\theta (i',j')=(i',j')$ then $f_i , f_{i'} , g_j , g_{j'}$ is a $(1,1)$--aperiodic quartet and so $\mathbb{F}^2_\theta$ is aperiodic by Proposition~\ref{lem:2graph-ap}. More generally, for $a,b \ge 1$ the bijection $\theta$ induces a bijection $\widetilde{\theta} : \underline{m}^a \times \underline{n}^b \to \underline{m}^a \times \underline{n}^b$. If there are
$i \neq i' \in \underline{m}^a$ and $j \neq  j' \in \underline{n}^b$
such that $\theta (i,j) = (i',j)$, $\theta (i',j)=(i,j)$, $\theta (i,j')=(i,j')$ and $\theta (i',j')=(i',j')$ then $f_i, f_{i'} , g_j , g_{j'}$\footnote{where for $i \in \underline{m}^a$ we have $f_i = f_{i_1} \ldots f_{i_a}$ and similarly for $j \in \underline{n}^b$ we have $g_j = g_{j_1}\ldots g_{j_b}$} is an $(a,b)$--aperiodic quartet and so $\mathbb{F}^2_\theta$ is aperiodic by Proposition~\ref{lem:2graph-ap}.
These results are compatible with \cite[Theorem 3.4]{DY}. Indeed, when $\mathbb{F}_\theta^2$ is aperiodic we have $C^* ( \mathbb{F}_\theta^2 ) \cong \mathcal{O}_m \otimes \mathcal{O}_n$ (cf.\ \cite[\S 5]{Ya}).
\end{example}

\begin{proposition} \label{prop:sap}
Let $\Lambda$ be a row-finite $2$-graph with no sources. Suppose that every vertex has an aperiodic quartet, then $\Lambda$ is strongly aperiodic.
\end{proposition}

\begin{proof}
Let $H$ be a hereditary subset of $\Lambda^0$. Since every vertex in $\Lambda^0$ has an aperiodic quartet, there is an aperiodic quartet $(\alpha_1,\alpha_2,\beta_1,\beta_2)$ at any  $v\in\Lambda^0\setminus H$. Since the source and range of $\alpha_i$ and $\beta_i$ for $i=1,2$ are the vertex $v\in\Lambda^0\setminus H$, $\alpha_i,\beta_i\in \Gamma(\Lambda\setminus H)$ for $i=1,2$. Thus, by Proposition \ref{lem:2graph-ap} $\Lambda^0\setminus H$ is aperiodic, which proves that $\Lambda$ is strongly aperiodic.
\end{proof}

\noindent Now we turn our attention to the description of all the prime ideals in
$C^* ( \Lambda )$.  we update the definition of a maximal tail given in \cite[Proposition 6.1]{BPRS} into the context of row-finite $k$-graphs with no sources.

\begin{definition}\label{Def-MT}
Let $\Lambda$ be a row-finite $k$-graph with no sources. A nonempty subset $\gamma$ of ${\Lambda}^0$ is called \textit{maximal tail} if
\begin{enumerate}[(a)]
\item \label{one} for every $v_1,v_2\in \gamma$ there is $w\in \gamma$ such that $v_1\Lambda w\ne\emptyset$ and $v_2\Lambda w\ne\emptyset$,
\item \label{two} for every $v\in \gamma$ and $1\le i\le k$ there is $e\in v{\Lambda}^{e_i}$ such that $s(e)\in \gamma$, and
\item for $w\in \gamma$ and $v\in{\Lambda}^0$ with $v\Lambda w\ne\emptyset$ we have $v\in \gamma$.
\end{enumerate}
\end{definition}

\begin{remark}
Versions of condition (\ref{one}) in Definition~\ref{Def-MT} can be traced back to \cite[Theorem 3.8 (iii)]{B} and \cite[Lemma 3.1 (iii)]{Do}.
The word ``tail'' in the above definition is meant to convey the sense of conditions (a) and (b), and ``maximal'' to convey that of condition (c).
Also the notion of maximal tail for a $k$-graph was introduced in Sims' thesis \cite[Proposition~5.5.3]{Sims_thesis}, but the condition (b) in the above definition is stronger than the condition (MT2) in \cite{Sims_thesis} since not every finite exhaustive set contains edges $e\in v\Lambda^{e_i}$ for all $1\le i\le k$.
\end{remark}

\noindent
Now we state the main theorem of this paper, which is a generalization of \cite[Proposition 6.1]{BPRS}. We give complete proof because
of the new condition (\ref{two}) in the definition of the maximal tail of a $k$-graph.

\begin{theorem}\label{MT}
Let $\Lambda$ be a row-finite strongly aperiodic k-graph with no sources.
 Let $H\subsetneq  {\Lambda}^0$ be a saturated hereditary subset. Then $I_H$ is primitive if and only if $\gamma:={\Lambda}^0\setminus H$ is a maximal tail.
\end{theorem}

\begin{proof}
First suppose that $\gamma\in{\Lambda}^0$ is a maximal tail. Let $H={\Lambda}^0\setminus\gamma$, so $H \neq \Lambda^0$. To see that $H$ is hereditary: let $v\in H$ and suppose $v\le w$. If $w\in\gamma$, then $v\in\gamma$ by $(c)$. Thus, $w\in H={\Lambda}^0\setminus\gamma$. To see that $H$ is saturated: let $v\in{\Lambda}^0$ such that $\{s(\lambda)\;:\;\lambda\in v{\Lambda}^{e_i}\}\subset H$ for some $i\in\{1,\dots,k\}$. If $v\in\gamma$, then there are $\lambda\in v{\Lambda}^{e_i}$ for all $i$ such that $s(\lambda)\in\gamma$ by $(b)$. But this contradicts $\{s(\lambda)\;:\;\lambda\in v{\Lambda}^{e_i}\}\subset H$ for some $i\in\{1,\dots,k\}$. Thus, $v\in H={\Lambda}^0\setminus\gamma$. Now to show that $I_H$ is prime, suppose $I_1$ and $I_2$ are ideals of $C^{\ast}(\Lambda)$ and $I_1\cap I_2\subset I_H$. Since $\Lambda$ is strongly aperiodic, every ideal is gauge invariant by Proposition \ref{GI}. Then by Theorem \ref{ff} $(a)$ and $(c)$, there exist corresponding saturated hereditary subsets $H_1$ and $H_2$ such that $I_1=I_{H_1}$, $I_2=I_{H_2}$ and $I_{H_1\cap H_2}=I_{H_1}\cap I_{H_2}$. Thus, $I_1\cap I_2\subset I_H$ implies that $H_1\cap H_2\subset H$. We claim that $H_1\subset H$ or $H_2\subset H$. Suppose that $H_1\nsubseteq H$ and $H_2\nsubseteq H$, then there exist $v_1\in H_1\setminus H$, $v_2\in H_2\setminus H$. i.e. $v_1\in \gamma$ and $v_2\in\gamma$. By $(a)$, there is $v\in\gamma$ such that $v_1\le v$ and $v_2\le v$. Since $H_1$ and $H_2$ are hereditary, $v\in H_1$ and $v\in H_2$. So $v\in H_1\cap H_2\subset H={\Lambda}^0\setminus\gamma$ which contradicts $v\in\gamma$. Thus, $H_1\subset H$ or $H_2\subset H$ that implies that $I_1=I_{H_1}\subset I_H$ or $I_2=I_{H_2}\subset I_H$. Hence $I_H$ is prime.

Now suppose $H$ is saturated and hereditary and $I_H$ is primitive. Let $\gamma={\Lambda}^0\setminus H$, then $\gamma$ satisfies $(c)$ : suppose not then $v\le w$ and $w\in\gamma$. If $v\notin\gamma$, i.e. $v\in H$, then $w\in H$ since $H$ is hereditary, which contradicts $w\in\gamma$. So $v\in\gamma$. To show that $\gamma$ satisfies $(b)$, let $v\in\gamma$. Since $\Lambda$ has no sources, $v{\Lambda}^{e_i}\ne \emptyset$ for all $i\in\{1,\dots,k\}$. So there are $\lambda\in v{\Lambda}^{e_i}$ for all $i\in\{1,\dots,k\}$. Suppose $\{s(\lambda): v{\Lambda}^{e_i}\}\subset H$ for some $i\in\{1,\dots,k\}$, then $v\in H$ since $H$ is saturated, which contradicts $v\in \gamma$. Thus, there are $\lambda\in v{\Lambda}^{e_i}$ such that $s(\lambda)\in\gamma$ for all $i\in\{1,\dots,k\}$. To prove $(a)$, recall that for a hereditary saturated set $H$, $C(\Lambda)/I_H$ is isomorphic to $C^{\ast}(\Gamma(\Lambda\setminus H))$ by Theorem \ref{ff} (d).
Because $I_H$ is primitive in $C^{\ast}(\Lambda)$, $\{0\}$ is primitive in $C^{\ast}(\Gamma(\Lambda\setminus H))$. Suppose that $v_1, v_2\in{\Lambda}^0\setminus H$. Then $H_i=\{x\in{\Lambda}^0\setminus H: v_i\le x\}$ are non-empty hereditary subsets of ${\Lambda}^0\setminus H={\Gamma(\Lambda\setminus H)}^0$. Since $\{0\}$ is prime in $C^{\ast}(\Gamma(\Lambda\setminus H))$, we must have $I_{\overline{H_1}}\cap I_{\overline{H_2}}\ne\{0\}$. If $I_{\overline{H_1}}\cap I_{\overline{H_2}}=\{0\}$, then $I_{\overline{H_1}}\cap I_{\overline{H_2}}\subseteq\{0\}$ and the fact that $\{0\}$ is prime implies that $I_{\overline{H_1}}\subseteq\{0\}$ or $I_{\overline{H_2}}\subseteq\{0\}$. Hence, $I_{\overline{H_1}}=\{0\}$ or $I_{\overline{H_1}}=\{0\}$, but $I_{H_1}\ne\{0\}$ since $H_1\ne\emptyset$. Thus, $I_{\overline{H_1}}\cap I_{\overline{H_2}}\ne\{0\}$ implies $\overline{H_1}\cap\overline{H_2}\ne\emptyset$ by Theorem \ref{ff} $(c)$. Say $y\in\overline{H_1}\cap\overline{H_2}$, then $y\in\overline{H_1}=\overline{\{x\in {\Lambda}^0\setminus H: v_1\le x\}}$. By Lemma \ref{lem2}, there is $x\in{\Lambda}^0\setminus H$ such that $y\le x$ and $v_1\le x$. Since $y\in\overline{H_2}$ and $\overline{H_2}$ is hereditary, we have $x\in\overline{H_2}=\overline{\{x\in {\Lambda}^0\setminus H: v_2\le x\}}$. Applying Lemma \ref{lem2} again on $x$ and $\overline{H_2}$. we have $z\in{\Lambda}^0\setminus H$ such that $x\le z$ and $v_2\le z$. So we have $v_1\le x\le z$ and $v_2\le z$ in $\Gamma({\Lambda}^0\setminus H)$. Thus $\gamma={\Lambda}^0\setminus H$ satisfies $(a)$.
\end{proof}

\begin{remark}
Observe that the strongly aperiodic condition on $\Lambda$ was not used in the second half of the proof of Theorem~\ref{MT}. Hence, for any row finite $k$-graph with no sources if the ideal $I_H$ is primitive, then $\gamma = \Lambda^0 \setminus H$ is a maximal tail. This is proved for arbitrary finitely aligned $k$-graph in \cite[Proposition~5.5.3]{Sims_thesis}.
\end{remark}

\begin{notation}
Let $\Lambda$ be a row-finite $k$-graph with no sources, then we denote the set of maximal tails of $\Lambda$ by $\chi_{\Lambda}$.
\end{notation}

\noindent
For nonempty subsets $K,L$ of ${\Lambda}^0$, we write $K\le L$ to mean that for each $v\in K$, there exists $w\in L$ such that $v\le w$. Thus condition (c) of Definition \ref{Def-MT} says that ``$\{v\}\le\gamma$ implies $v\in \gamma$''. Also, we can describe the saturated hereditary set $H_{\gamma}$ corresponding to $\gamma\in\chi_{\Lambda}$ as either $H_{\gamma}={\Lambda}^0\setminus\gamma$ or $H_{\gamma}=\{v:\{v\}\nleq\gamma\}$, (as in \cite[Proposition 4.1]{ahr}).

The following description of the Jacobson topology of the primitive ideal space of $C^* (\Lambda )$ is a generalization of \cite[Theorem 6.3]{BPRS} which has an identical proof, so we omit it (note that this topology is $T_0$ (see \cite[4.1.4]{Pe})).

\begin{theorem}\label{topMT}
Let $\Lambda$ be a row-finite strongly aperiodic $k$-graph with no sources. Then there is a topology $\mathcal{T}_0$ on the set $\chi_{\Lambda}$ of maximal tails in $\Lambda$ such that
\[
\overline{S}=\{\delta\in\chi_{\Lambda} : \delta\subseteq\bigcup_{\gamma\in S}\gamma\}
\]
for $S\subset\chi_{\Lambda}$, and then $\gamma\mapsto I_{H_{\gamma}}$ is a homeomorphism of $\chi_{\Lambda}$ onto $\text{Prim}\;C^{\ast}(\Lambda)$.
\end{theorem}

\noindent
Following \cite{B}, we can describe an equivalent topology on $\chi_\Lambda$ which will be useful in the next section.

\begin{lemma}\label{topMT2}
Let $\Lambda$ be a row-finite strongly aperiodic $k$-graph with no sources. For $v\in{\Lambda}^0$, let
\[
 S(v):=\{\chi\in\chi_{\Lambda}:v\in\chi\}.
 \]
 Then, $\{S(v):v\in\Lambda^0\}$ form a countable base of open sets for a topology $\mathcal{T}_1$ on $\chi_{\Lambda}$. Moreover, the topologies $\mathcal{T}_0$ and $\mathcal{T}_1$ on $\chi_{\Lambda}$ are equal.
\end{lemma}

\section{Characterizing $Prim$ $C^{\ast}(\Lambda)$ for a strongly aperiodic $k$-graph $\Lambda$}

\begin{definition}\cite{BE}.
A closed subset $C$ of a topological space is \textit{irreducible} if it cannot be written as the union of two proper closed subsets of itself. A \textit{spectral space} is a $T_0$ space in which every irreducible set is the closure of a point.
\end{definition}

\noindent
The main theorem in this section describes precisely which topological spaces can occur as the primitive ideal space of the $C^*$-algebra of a strongly aperiodic row-finite $k$-graph with no sources, and generalizes \cite[Theorem 4.2]{B}.

\begin{theorem}\label{SS}
Let $X$ be a topological space. Then $X$ is homeomorphic to $Prim\;\;C^{\ast}(\Lambda)$ for a row-finite strongly aperiodic $k$-graph with no sources if and only if $X$ is a spectral space in which the compact open sets form a countable base.
\end{theorem}

\noindent
We first characterize the irreducible subsets of $\chi_\Lambda$. The proof of following result is the same as that of \cite[Lemma 4.3]{B}.

\begin{lemma}
Let $\Lambda$ be a row-finite strongly aperiodic $k$-graph with no sources and let $C\subseteq\chi_{\Lambda}$ be a closed set. Write $\chi_{C}=\cup_{\chi\in C}\chi$. Then the following statements are equivalent:
\begin{enumerate}[(a)]
\item the set $C$ is irreducible.
\item the set $\chi_C$ is a maximal tail.
\end{enumerate}
If these conditions are satisfied, then for all $v\in{\Lambda}^0$, $\chi_C\in S(v)$ if and only if $S(v)\cap C\ne\emptyset$.
\end{lemma}

\noindent The proof of the following Proposition is the same as that of \cite[Proposition 4.4]{B}.

\begin{proposition}\label{irreducible}
Let $\Lambda$ be a row-finite strongly aperiodic $k$-graph with no sources. Then $\chi_{\Lambda}$ has the property that every irreducible set is the closure of a point in $\chi_{\Lambda}$.
\end{proposition}

\noindent
Proposition \ref{irreducible} shows that $\chi_\Lambda$ is a spectral space. We now
turn our attention to finding a base of compact open sets using the topology described in Lemma~\ref{topMT2}.

We prove a technical lemma first. It shows that every vertex in a maximal tail can be reached by a path of strictly positive degree from some
other vertex in the same maximal tail.

\begin{lemma}\label{degree}
Let $\Lambda$ be a row-finite $k$-graph with no sources. Let $\chi \in\chi_{\Lambda}$ be a maximal tail, then for all $v \in \chi$ there is $\mu\in v\Lambda$ such that $d(\mu)>0$ and $s(\mu)\in\chi$. (i.e. $d(\mu)_i>0$ for $i=1, \dots, k$).
\end{lemma}
\begin{proof}
Fix $v\in\chi$. Since $\chi$ is a maximal tail, there is $\lambda\in v\Lambda^{e_i}$ such that $s(\lambda)\in\chi$ for all $i=1,\dots,k$. So let $\lambda_1\in v\Lambda^{e_i}$ such that $s(\lambda_1)\in\chi$. Let $u_1=s(\lambda_1)$. Then, let $\lambda_2\in u_1\Lambda^{e_2}$ such that $s(\lambda_2)\in\chi$. Continuing this process, we find $\lambda_i\in\Lambda^{e_i}$ such that $s(\lambda_i)\in\chi$ and $r(\lambda_{i+1})=s(\lambda_1)$ for all $i=1,\dots,k$. Let $\mu=\lambda_1\dots\lambda_k\in v\Lambda$. Then the factorization property implies $d(\mu)_i>0$ for all $i=1,\dots,k$.
\end{proof}

\noindent The following Proposition is a generalization of \cite[Lemma 4.5]{B}, whose proof is much more intricate due to the complex topology of
$\Lambda^\infty$ and the new condition (\ref{two}) in the definition of the maximal tail of a $k$-graph.

\begin{proposition}\label{basis}
Let $\Lambda$ be a row-finite strongly aperiodic $k$-graph with no sources. Define a map $\beta:{\Lambda}^{\infty}\rightarrow {\Lambda}^0$ by
\[
\beta(x)=\{v\in{\Lambda}^0:v\Lambda x(n,n)\ne\emptyset\;\;\text{for some $n\in\mathbb{N}^k$}\}.
\]

\noindent
Then
\begin{enumerate}[(A)] \item $\beta(x)$ is a maximal tail.
\item $\beta$ is continuous open surjection and $\chi_{\lambda}$ is a quotient space of ${\Lambda}^{\infty}$.
\item The open sets $\{S(v):v\in{\Lambda}^0\}$ are compact subsets of $\chi_{\Lambda}$.
\end{enumerate}
\end{proposition}

\begin{proof}
(A): To check the condition (a) of maximal tail, for $v_1,v_2\in\beta(x)$, we want to find $z\in\beta(x)$ such that $v_1\le z$ and $v_2\le z$. Since $v_1\in\beta(x)$, there is $n_1\in\mathbb{N}^k$ such that $v_1\le x(n_1,n_1)$. Similarly, there is $n_2\in\mathbb{N}^k$ such that $v_2\le x(n_2,n_2)$. Let $N=\text{max}(n_1,n_2)$, and let $z=x(N,N)$. Then, $v_1\le x(n_1,n_1)\le x(N,N)=z$ and $v_2\le x(n_2,n_2)\le x(N,N)=z$. Thus, the condition (a) is satisfied.
Now fix $v\in\beta(x)$, then there is $n\in\mathbb{N}^k$ such that $v\le x(n,n)$. Let $\mu\in v\Lambda x(n,n)$. Since $x\in{\Lambda}^{\infty}$, there is $x(n,n+e_i)\in x(n,n)\Lambda x(n+e_i,n+e_i)$ for all $1\le i\le k$. Then the factorization property implies that there is $\lambda_i\in v{\Lambda}^{e_i}$ such that $\mu x(n,n+e_i)=\lambda_i{\mu}'$. Since $s({\mu}')=s(x(n,n+e_i))=x(n+e_i,n+e_i)$ and $s(\lambda_i)\le s({\mu}')$, $s(\lambda_i)\in\beta(x)$. Thus, the condition (b) is satisfied.
To show that the condition (c) is satisfied, suppose $v\le w$ and $w\in\beta(x)$. Then, there is $n\in\mathbb{N}^k$ such that $w\le x(n,n)$. So $v\le w\le x(n,n)$ implies that $v\in\beta(x)$. Hence, $\beta(x)$ is a maximal tail.

(B): To show that $\beta$ is continuous, we claim
\begin{equation}\label{conti}
\begin{split}
\beta^{-1}(S(v))&=\{x\in{\Lambda}^{\infty}: v\le x(n,n)\;\;\text{for some}\;\;n\in\mathbb{N}^k\}\\
&=\bigcup_{v\le s(\mu)}Z(\mu).
\end{split}
\end{equation}
To show the first equality, let $x\in{\Lambda}^{\infty}$ be such that $v\le x(n,n)$ for some $n\in\mathbb{N}^k$. Then $v\in\beta(x)$. Since $\beta(x)$ is a maximal tail, we have $\beta(x)\in S(v)$. On the other hand, let $x\in{\Lambda}^{\infty}$ such that $\beta(x)\in S(v)$. So $v\in\beta(x)$ and it gives $n\in\mathbb{N}^k$ such that $v\le x(n,n)$. To show the second equality, take $x\in Z(\mu)$ such that $v\le s(\mu)$. Then $x=\mu t$ for $t\in{\Lambda}^{\infty}$. Clearly $s(\mu)=x(n,n)$ for some $n\in\mathbb{N}^k$. Thus, $v\le x(n,n)$. So $x$ belongs to the set. Now take $x\in{\Lambda}^{\infty}$ such that $v\le x(n,n)$ for some $n\in\mathbb{N}^k$. Let $\mu=x(0,n)$. Then $x\in Z(\mu)$ and $v\le x(n,n)=s(\mu)$. Thus, $x\in\bigcup_{v\le s(\mu)}Z(\mu)$, which proves the above equalities.
Since $Z(\mu)$ is open set, $\beta^{-1}(S(v))$ is open. Thus, $\beta$ is continuous.

To show that $\beta$ is surjective, we need to consider two cases : $|\chi|<\infty$ and $|\chi|=\infty$.
If $|\chi|<\infty$, then let $\chi=\{u_1,\dots,u_n\}$. In the case that $n=1$, we apply Lemma \ref{degree} to obtain $\lambda\in u_1\Lambda$ such that $d(\lambda)>0$ and $s(\lambda)\in\chi$. So $s(\lambda)=u_1$. Let $x$ be an infinite path of the form $x=\mu\mu\dots$, then $\beta(x)=\chi$.
In the case that $n>1$, let $p_1=u_1, u_2\in\chi$. By applying the maximal tail condition (a) we have $p_2\in\chi$ such that $u_1\Lambda p_2\ne\emptyset$ and $u_2\Lambda p_2\ne\emptyset$. Let $\mu_1\in u_1\Lambda p_2=p_1\Lambda p_2$. By similar argument, we can find $p_i\in\chi$ and $\mu_i\in p_i\Lambda p_{i+1}$ for $i=1,\dots,n-1$ such that $u_i\le p_i$ for $i=1,\dots,n$. Let $\delta_i\in u_i\Lambda p_i$. Then apply the maximal tail condition $(b)$ on $p_n$, we find $\lambda\in p_n{\Lambda}^{e_i}$ for all $i=1,\dots,k$ such that $s(\lambda)\in\chi$. Since $\chi$ is finite, $s(\lambda)=u_l$ for some $1\le l\le n$. So $\phi=\lambda\delta_{l+1}\mu_{l+1}\dots\mu_n$ is a loop based on $p_n$ since $s(\phi)=s(\mu_n)=p_n$ and $r(\phi)=r(\lambda)=p_n$. Let $x=(\phi)^{\infty}$. Then we certainly have $\beta(x)=\chi$.

If $|\chi|=\infty$, we write $\chi=\{v_i\}^{\infty}_{i=1}$. Let $p_1=v_1$, the the maximal tail condition (a) gives $p_2\in\chi$ such that $v_1\Lambda p_2\ne\emptyset$ and $v_2\Lambda p_2\ne\emptyset$. Let $\mu_1\in v_1\Lambda p_2=p_1\Lambda p_2$. Then by Lemma \ref{degree}, there is $\rho_1\in p_2\Lambda$ such that $d(\rho_1)>0$ and $s(\rho_1)\in\chi$. Let $q_1=s(\rho_1)$ and $\lambda_1=\mu_1\rho_1$. Then certainly $d(\lambda_1)>0$. Apply the similar argument to $v_3$ and $q_1$ to obtain $\lambda_2\in s(\lambda_1)\Lambda$ such that $d(\lambda_2)>0$ and $s(\lambda_2)\in\chi$. Then $\{v_1,v_2,v_3\}\le(\lambda_1\lambda_2)^0$, where $(\lambda_1\lambda_2)^0$ is the set of vertices on $\lambda_1\lambda_2$. So inductively we can find $\lambda_1,\dots,\lambda_n$ such that $\{v_1,\dots,v_{n+1}\}\le(\lambda_1\dots\lambda_n)^0$. Let $x=\lambda_1\lambda_2\dots$, then $\beta(x)=\chi$. Hence, $\beta$ is surjective.\footnote{
The above proof suggests that when $|\chi|=\infty$, for given $v\in\chi$ we can construct an infinite path $x$ such that $\beta(x)=\chi$ and $r(x)=v$.}

To see that $\beta$ is an open map, we claim the following.
\begin{equation}\label{open}
\beta(Z(\mu))=\{\chi:s(\mu)\in\chi\}.
\end{equation}
Let $x=\mu t\in Z(\mu)$. Then, $s(\mu)\in\beta(x)$. Since $\beta(x)$ is a maximal tail, $\beta(x)\in \{\chi:s(\mu)\in\chi\}$. Now take $\delta\in\{\chi:s(\mu)\in\chi\}$. Then by the condition (c) of maximal tail, for $s(\mu)\in\delta$ we have $r(\mu)\in\delta$. As shown in the proof of the surjectivity of $\beta$, for $s(\mu)\in\delta$, there is $x\in s(\mu){\Lambda}^{\infty}$ such that $\beta(x)=\delta$. Then $\mu x\in Z(\mu)$ and $\beta(\mu x)=\delta$, which proves the claim. Thus, we have $\beta(Z(\mu))=\{\chi:s(\mu)\in\chi\}=S(s(\mu))$, which is open in $\chi_{\Lambda}$. Therefore, $\beta$ is an open map.

(C): For $v\in{\Lambda}^0$, the subset $Z(v)\subset{\Lambda}^{\infty}$ is compact by Lemma 2.6 in \cite{KP1}. By \eqref{open} above, we have
\[S(v)=\{\chi\in\chi_{\Lambda}:v\in\chi\}=\beta(Z(v)).\]
Since $\beta$ is continuous, $\beta(Z(v))$ is compact, which implies that $S(v)$ is compact.

\end{proof}

\begin{proof}[Proof of Theorem \ref{SS}]
Let $X$ be a spectral space in which the compact open sets form a countable base. Then by \cite[Theorem 5]{BE}, there is an AF-algebra $\mathcal{A}$ such that Prim $\mathcal{A}\simeq X$. Moreover by \cite[Theorem 1]{D}, $\mathcal{A}$ is Morita equivalent to $C^{\ast}(\mathcal{D}(\mathcal{A}))$ where $\mathcal{D}(\mathcal{A})$ is a Bratteli diagram for $\mathcal{A}$. Hence, by \cite[Corollary 3.33]{RW}, we have Prim $C^{\ast}(\mathcal{D}(\mathcal{A}))\simeq \text{Prim }\mathcal{A}$. Moreover, a Bratteli diagram satisfies condition (K), hence it is strongly aperiodic by Lemma~\ref{lem:KisSA}.

To prove the converse, let $\Lambda$ be a row-finite strongly aperiodic $k$-graph with no sources. Then by Theorem \ref{topMT} we have Prim $C^{\ast}(\Lambda)\simeq\chi_{\Lambda}$. It follows from Proposition \ref{irreducible} that $\chi_{\Lambda}$ is a spectral space, and  by Lemma \ref{basis}, the set $\{S(v):v\in{\Lambda}^0\}$ forms a countable base of compact open sets for the topology on $\chi_{\Lambda}$.
\end{proof}

\noindent As in \cite[Corollary 4.6]{B}, we may apply \cite[Theorem 5]{BE} to Theorem~\ref{SS} to get:

\begin{corollary}\label{AF}
If $\Lambda$ is a row-finite strongly aperiodic $k$-graph with no sources, then there is an $AF$-algebra $\mathcal{A}$ such that $\text{Prim}\;\mathcal{A}\simeq\text{Prim}\;C^{\ast}(\Lambda)$.
\end{corollary}

\begin{remark}
In \cite[\S 5]{B}, Bates gives an algorithm to construct from a directed graph $E$ an auxiliary directed graph $\widetilde{E}$ whose $C^*$-algebra is AF and has the same primitive ideal space $C^* (E)$.  Since we are dealing with $k$-graph $\Lambda$ which has $k$-colored $1$-skeleton, Bates' construction does not work as expected. The main idea of her construction in \cite{B} is that there is one-to-one correspondence between maximal tails in $E$ and maximal tails in the auxiliary graph $\widetilde{E}$. When we apply the same construction of Bates' auxiliary graph to $k$-graph, we can only show that the map from $\Lambda$ to $\widetilde{\Lambda}$ takes the maximal tail of $k$-graph $\Lambda$ to the maximal tail of 1-graph $\widetilde{\Lambda}$ but not the other way. If we start with a maximal tail of 1-graph to construct a maximal tail of $k$-graph, the corresponding set of vertices no longer satisfy the conditions of maximal tail, in particular the condition (b). Thus, we need a different construction of auxiliary graph for $k$-graph that represents $AF$-algebra, and currently we do not know any construction that works.

\end{remark}

\section{Cartesian product and skew product examples}

In this section we describe two strongly aperiodic $2$-graphs with no sources, one a skew product graph which is not a cartesian product graph and the other a cartesian product graph. We compute the primitive ideal spaces of their associated $C^*$-algebras and show that they are homeomorphic even though the $2$-graphs are not isomorphic. The common feature that the examples have is that they have the same $1$-skeleton, which we now describe.

A $k$-graph $\Lambda$ can be visualized by its $1$-skeleton: This is a directed graph $E_\Lambda$ with vertices $\Lambda^0$
and edges $\cup_{i=1}^k \Lambda^{e_i}$ which have range and source in $E_\Lambda$ determined by their range and source in $\Lambda$.
Each edge in $E_\Lambda$ with degree $e_i$ is assigned the same color $c_i$, so $E_\Lambda$ is a colored graph. It is common to
call edges with degree $e_1$ in $\Lambda$ blue edges in $E_\Lambda$ and draw them with solid lines; edges with degree $e_2$
in $\Lambda$ are then called red edges and are drawn as dashed lines. Different $k$-graphs can determine the same $1$-skeleton since
the skeleton does not encode the factorization property of the $k$-graph. In practice, along with the $1$-skeleton we give a
collection of factorization rules which relate the edges of $E_\Lambda$ that occur in the factorization of morphisms of degree
$e_i+e_j$ ($i \neq j$) in $\Lambda$. For more information about $1$-skeletons and their relationship with $k$-graphs we refer
the reader to \cite{RSY1}.

Before we give an example, we introduce the cartesian product graphs and skew-product graphs.

\begin{proposition}\cite[Proposition 1.8]{KP1}.\label{kgraph}
Let $(\Lambda_i, d_i)$  be $k_i$-graphs for $i=1,2$, then the cartesian product graph $(\Lambda_1\times\Lambda_2, d_1\times d_2)$ is a  $(k_1+k_2)$-graph where $\Lambda_1\times\Lambda_2$ is the product category and $d_1\times d_2:\Lambda_1\times\Lambda_2\rightarrow \mathbb{N}^{k_1+k_2}$ is given by $d_1\times d_2(\lambda_1,\lambda_2)=(d_1(\lambda_1),d_2(\lambda_2))\in\mathbb{N}^{k_1}\times\mathbb{N}^{k_2}$ for $\lambda_1\in\Lambda_1$ and $\lambda_2\in\Lambda_2$.
\end{proposition}

When we identify $\mathbb{N}^{k_1+k_2}$ with $\mathbb{N}^{k_1} \times \mathbb{N}^{k_2}$ in the above Proposition we may suppose that $\mathbb{N}^{k_1+k_2}$ has a basis $\{ e_1 , \ldots , e_{k_1} , f_1 , \ldots f_{k_2} \}$ where $\{ e_1 , \ldots , e_{k_1} \}$ and $\{ f_{1} , \ldots , f_{k_2} \}$ are the canonical bases for $\mathbb{N}^{k_1}$ and $\mathbb{N}^{k_2}$ respectively. Then the factorization rule for elements of degree $(e_i + f_{j})$ in $\Lambda_1 \times \Lambda_2$ are of the form
\begin{equation} \label{eq:fpforcart}
(  r(a) , b ) ( a , s(b)  ) = ( a , r(b) ) ( s(a) , b ) \text{ where }  b \in \Lambda_1^{e_i} , a \in \Lambda_2^{f_{j}}.
\end{equation}

\begin{remarks} \label{rem:whencart}
\begin{enumerate}[(a)]
\item Suppose that a $2$-graph $\Lambda$ has edges $b \in \Lambda^{e_1} , a \in \Lambda^{e_2}$ such that $r(a)=s(a)=r(b)=s(b)$.
If $\Lambda$ is a cartesian product of $1$-graphs then by \eqref{eq:fpforcart} we must have $ab=ba$.

\item By (a) it follows that for $m,n \ge 1$ the $2$-graph $\mathbb{F}^2_\theta$ described in Examples~\ref{ex:kex} (3) is the cartesian product $B_m \times B_n$ if and only if $\theta$ is the identity function.
\end{enumerate}
\end{remarks}

\begin{theorem}\label{prod}
Let $\Lambda_i$ be $k_i$-graphs for $i=1,2$. Then
\begin{enumerate}[(a)]
\item $\Lambda_1\times\Lambda_2$ is a row-finite  $(k_1+k_2)$-graph with no sources if and only if $\Lambda_i$ is a row-finite $k_i$-graph with no sources for $i=1,2$; 
\item $\Lambda_1 \times \Lambda_2$ is aperiodic if and only if $\Lambda_1 , \Lambda_2$ are aperiodic; and
\item $\Lambda_1 \times \Lambda_2$ is strongly aperiodic if and only if $\Lambda_1 , \Lambda_2$ are strongly aperiodic.
\end{enumerate}
\end{theorem}

\begin{proof}
(a): Proposition \ref{kgraph} implies that $\Lambda_1\times\Lambda_2$ is a $(k_1+k_2)$-graph. It is straightforward to check that
$\Lambda_1 \times \Lambda_2$ is row-finite with no sources, and conversely.

(b): Suppose that $\Lambda_1 , \Lambda_2$ are aperiodic. Fix $(v,w)\in(\Lambda_1\times\Lambda_2)^0$ and $m\ne n\in\mathbb{N}^{k_1+k_2}$. Let $m'=(m_1,\dots,m_{k_1})$, $m''=(m_{k_1+1},\dots,m_{k_1+k_2})$ and $n'=(n_1,\dots,n_{k_1})$, $n''=(n_{k_1+1},\dots,n_{k_1+k_2})$.
Then $m'\ne n'$, or $m''\ne n''$ (or both). Assume that $m'\ne n'$. Since $\Lambda_1$ is aperiodic, for $v \in \Lambda_1^0$ there exists $\lambda_1\in v\Lambda_1$ such that $d(\lambda_1)>m' \vee n'$ and
\[\lambda_1(m',m'+d(\lambda_1)-m'\vee n')\ne\lambda_1(n',n'+d(\lambda_1)-m'\vee n').\]

\noindent
For $w \in\Lambda_2^0$, let $\lambda_2\in w\Lambda_2$ be any path with $d(\lambda_2)>m''\vee n''$. We claim that \[(\lambda_1,\lambda_2)(m,m+d_1\times d_2(\lambda_1,\lambda_2)-m\vee n)\ne(\lambda_1,\lambda_2)(n,n+d_1\times d_2(\lambda_1,\lambda_2)-m\vee n).
\]

\noindent Since the left hand side of the above equation is equal to 
\[
(\lambda_1(m',m'+d_1(\lambda_1)-m'\vee n'), \lambda_2(m'',m''+d_2(\lambda_2)-m''\vee n''))
\]

\noindent  by the definition of degree functor $d_1\times d_2$ on $\Lambda_1\times \Lambda_2$,  and, similarly the right hand side of the above equation is equal to $(\lambda_1(n',n'+d_1(\lambda_1)-m'\vee n'), \lambda_2(n'',n''+d_2(\lambda_2)-m''\vee n''))$. Since $\lambda_1(m',m'+d(\lambda_1)-m'\vee n')\ne\lambda_1(n',n'+d(\lambda_1)-m'\vee n')$ the claim follows. Therefore, $\Lambda_1\times\Lambda_2$ is aperiodic. If we assume that $m'' \ne n''$ then a similar argument, using the aperiodicity of $\Lambda_2$ shows that $\Lambda_1 \times \Lambda_2$ is aperiodic.

Now suppose that $\Lambda_1 \times \Lambda_2$ is aperiodic. First we identify $\Lambda_1$ with $\Lambda_1 \times \{v_2 \}$ for some $v_2 \in \Lambda_2^0$. Fix $v_1 \in \Lambda_1^0$, then we claim that $\Lambda_1$ has no local periodicity
at $v_1$. To see this, observe that $\Lambda_1 \times \Lambda_2$ has no local periodicity at $(v_1 , v_2 )$. So for $(m,0) \neq (n,0) \in \mathbb{N}^{k_1} \times \mathbb{N}^{k_2}$ there is a path $(\lambda_1 , v_2 ) \in (v_1 , v_2 ) ( \Lambda_1 \times \Lambda_2 )^{(m',0)}$ with $m' > m \vee n$ such that \eqref{eq:lp} in Definition~\ref{aperiodic} holds.  Identifying $( \lambda_1 , v_2 ) \in \Lambda_1 \times \{ v_2 \}$ with $\lambda_1 \in v_1 \Lambda_1$ proves the claim. A similar argument applies to $\Lambda_2$.

(c): Suppose that $\Lambda_1 , \Lambda_2$ are strongly aperiodic. Then it is straightforward to see that every saturated hereditary subset of $( \Lambda_1 \times \Lambda_2 )^0 = \Lambda_1^0 \times \Lambda_2^0$ is of the form $H_1 \times H_2$ where $H_i$ is a saturated hereditary subset of $\Lambda_i^0$ for $i=1,2$. If $H_i \neq \Lambda_i$ for $i=1,2$ then by Theorem~\ref{ff} (4) we have
\begin{align} \label{eq:quot}
\Gamma ( ( \Lambda_1 \times \Lambda_2 ) &\backslash ( H_1 \times H_2 ) ) \nonumber \\
&= ( ( \Lambda_1^0 \times \Lambda_2^0 ) \backslash ( H_1 \times H_2 ) , \{ ( \lambda_1 , \lambda_2 ) \in \Lambda_1 \times \Lambda_2 : s ( \lambda_1 , \lambda_2 ) \not\in H_1 \times H_2 \} ,r , s ) \nonumber \\
&=\Gamma ( \Lambda_1 \backslash H_1 )  \times \Gamma ( \Lambda_2 \backslash H_2 )
\end{align}

\noindent then since $\Lambda_1 ,\Lambda_2$ are strongly aperiodic it follows that
$\Gamma ( \Lambda_i \backslash H_i ) $ is aperiodic for $i=1,2$. Hence by \eqref{eq:quot} it follows that $\Gamma ( ( \Lambda_1 \times \Lambda_2 ) \backslash ( H_1 \times H_2 ) )$ is aperiodic by part (b). Now suppose $H_2 = \Lambda_2^0$ and $H_1 \neq \Lambda_1^0$ then $H_1 \times \Lambda_2^0 \subsetneq \Lambda_1^0 \times \Lambda_2^0$ and
\begin{align} \label{eq:halftriv}
\Gamma ( ( \Lambda_1 \times \Lambda_2 ) \backslash ( H_1 \times \Lambda_2^0 )) &= ( ( \Lambda_1^0 \times \Lambda_2^0 ) \backslash ( H_1 \times \Lambda_2^0 ) , \{ ( \lambda_1 , \lambda_2 ) : s ( \lambda_1 , \lambda_2 ) \not\in H_1 \times \Lambda_2^0 \} , r ,s ) \nonumber \\
&= ( ( \Lambda_1^0 \backslash H_1 ) \times \Lambda_2^0 , \{ ( \lambda_1 , \lambda_2 ) : s ( \lambda_1 ) \not\in H_1 \} , r , s )
\end{align}

\noindent which is isomorphic to $\Gamma ( \Lambda_1 \backslash H_1 ) \times \Lambda_2$.
Since $\Lambda_1$ is strongly aperiodic it follows that $\Gamma ( \Lambda_1 \backslash H_1 )$ and hence $\Gamma ( ( \Lambda_1 \times \Lambda_2 ) \backslash ( H_1 \times \Lambda_2^0 ) )$ is aperiodic by part (b). A similar argument applies if
$H_1 = \Lambda_1^0$ and $H_2 \neq \Lambda_2^0$. Hence $\Lambda_1 \times \Lambda_2$ is strongly aperiodic.

Finally if we assume that $\Lambda_1 \times \Lambda_2$ is strongly aperiodic and that $H_1 \subsetneq \Lambda_1^0$ is a saturated hereditary subset of $\Lambda_1$, then $H_1 \times \Lambda_2^0 \subsetneq \Lambda_1^0 \times \Lambda_2^0$ is a saturated hereditary subset of $\Lambda_1 \times \Lambda_2$. So by \eqref{eq:halftriv} we may identify $\Gamma ( ( \Lambda_1 \times \Lambda_2 ) \backslash ( H_1 \times \Lambda_2^0 ) )$ with $\Gamma( \Lambda_1^0  \backslash H_1 ) \times \Lambda_2$. By part (b) it follows that $\Gamma( \Lambda_1^0  \backslash H_1 )$ is aperiodic and since $H_1$ was arbitrary it follows that $\Lambda_1$ is strongly aperiodic. A similar argument shows that $\Lambda_2$ is strongly aperiodic.
\end{proof}

\begin{example}
Since $\Omega_k \cong \overbrace{\Omega_1 \times \ldots \times \Omega_1}^{k-\text{times}}$
it follows from Theorem~\ref{prod}~(a) and Examples~\ref{ex:aperiodex} that $\Omega_k$ is aperiodic.
\end{example}

\begin{definition}[Definition 5.1 in \cite{KP1}]
Let $G$ be a discrete group, $(\Lambda, d)$ a $k$-graph. Given $c:\Lambda\rightarrow G$ a functor, then define the \textit{skew product} graph $\Lambda\times_c G$ as follows : the objects are identified with ${\Lambda}^0\times G$ and the morphisms are identified with $\Lambda\times G$ with the following structure maps
$s(\lambda, g)=(s(\lambda), gc(\lambda))\;\;\text{and}\;\;r(\lambda, g)=(r(\lambda), g)$.
If $s(\lambda)=r(\mu)$, then $(\lambda,g)$ and $(\mu, gc(\lambda))$ are composable in $\Lambda\times_c G$ and
$(\lambda,g)(\mu, gc(\lambda))=(\lambda\mu, g)$.
The degree map is given by $d(\lambda,g)=d(\lambda)$.
\end{definition}

\begin{remark} It is straightforward to check that if $\Lambda$ is a row-finite $k$-graph with no sources, $\Lambda \times_c G$ is a row-finite $k$-graph with no sources. The factorization rule in $\Lambda \times_c G$ for elements of degree $e_i + e_j$, where $i \neq j$ is induced from the
factorization rule in $\Lambda$ as follows: If $d ( \lambda ) = e_i + e_j$ then $\lambda = \lambda ( 0 , e_i ) \lambda ( e_i ,d ( \lambda ) ) = \lambda ( 0 , e_j ) \lambda ( e_j , d ( \lambda ) )$ in $\Lambda$ and
\begin{equation} \label{eq:fpforspg}
( \lambda ( 0 , e_i ) , g ) ( \lambda ( e_i , d ( \lambda ) ) , g c ( \lambda ( 0 , e_i ) ) ) = ( \lambda ( 0 , e_j ) , g ) ( \lambda ( e_j , d ( \lambda ) ) , g c ( \lambda ( 0 , e_j ) ))
\end{equation}

\noindent in $\Lambda \times_c G$.
\end{remark}

\begin{example}\label{ex1}
Suppose that $\Gamma$ is the following 2-graph with the $1$-skeleton, shown below
\[
\begin{tikzpicture}[scale=0.35]
                    \node[circle,inner sep=0pt] (p11) at (1, 1)
                    {\begin{tikzpicture}[scale=0.4]
                    \node at (0.9, 0.8) [draw, fill=black] {$.$};
                     \end{tikzpicture}}
                            edge[-latex, loop, out=40, in=-40, min distance=80, looseness=1.5] (p11)
                            edge[-latex, loop, out=50, in=-50, min distance=123, looseness=2.2] (p11)
                            edge[-latex, loop, out=60, in=-60, min distance=190, looseness=2.3] (p11)
                            edge[-latex, loop, dashed, out=130, in=230, min distance=75, looseness=1.5] (p11)
                            edge[-latex, loop, dashed, out=125, in=235, min distance=115, looseness=1.8] (p11)
                            edge[-latex, loop, dashed, out=115, in=245, min distance=190, looseness=1.7] (p11)
                            ;
                            \node at (0.9, -0.5) {$v$};
                            \node at (-5, 1) {$g_1,g_2, g_3$};
                            \node at (7,1) {$f_1,f_2,f_3$};
                            \node at (-9, 1) {$\Gamma=$};
\end{tikzpicture}
\]

\noindent
and factorization rules
\begin{equation} \label{eq:fpforgamma}
\begin{array}{ll}
g_i f_1= f_2 g_j,  \quad g_j f_2 = f_1 g_j,  \quad g_j f_3 = f_3 g_j & \text{ for } j=1,3 \text{ and } \\
g_2 f_i = f_i g_2  \text{ where } f_i \in \Lambda^{e_1} \text{ and } g_j \in \Lambda^{e_2} & \text{ for } i,j=1,2,3.
\end{array}
\end{equation}

\noindent In fact $\Gamma = \mathbb{F}_\theta^2$ where $\theta : \underline{3} \times \underline{3} \to \underline{3} \times \underline{3}$
is given by
\[
\begin{array}{ll}
\theta ( 2,j ) = (1,j) , \quad \theta ( 1,j ) = (2,j) , \quad \theta (3,j) = (3,j) & \text{ for } j=1,3 \text{ and } \\
\theta ( i,2 ) = (i,2) \text{ for } i =1,2,3. &
\end{array}
\]


Define a map $c: \Gamma\rightarrow \mathbb{Z}^2$ by $c(g_3)=(0,1)$, $c(f_3)=(1,0)$, $c(f_i)=(0,0)$, $c(g_i)=(0,0)$, for $i=1,2$ then
$c$ preserves the factorization rules in $\Gamma$, thus it extends to a functor $c : \Gamma \to \mathbb{Z}^2$.
Then the $1$-skeleton of $\Lambda=\Gamma\times_c\mathbb{Z}^2$ is given as follows.
\[
\begin{tikzpicture}[scale=0.28]
                    \node[circle,inner sep=0pt] (p11) at (0, 0)
                    {\begin{tikzpicture}[scale=0.4]
                    \node at (0.9, 0.8) [draw, fill=black] {$.$};
                     \end{tikzpicture}}
                            edge[-latex, loop, out=170, in=110, min distance=60, looseness=2] (p11)
                            edge[-latex, loop, out=180, in=100, min distance=90, looseness=3] (p11)
                            edge[-latex, loop, out=280, in=340, min distance=60, dashed, looseness=2] (p11)
                            edge[-latex, loop, out=270, in=350, min distance=90, dashed, looseness=3] (p11)
                            ;
                    \node[circle,inner sep=0pt] (p12) at (0, 4)
                    {\begin{tikzpicture}[scale=0.4]
                    \node at (0.9, 0.8) [draw, fill=black] {$.$};
                     \end{tikzpicture}}
                            edge[-latex, loop, out=170, in=110, min distance=60, looseness=2] (p12)
                            edge[-latex, loop, out=180, in=100, min distance=90, looseness=3] (p12)
                            edge[-latex, loop, out=280, in=340, min distance=60, dashed, looseness=2] (p12)
                            edge[-latex, loop, out=270, in=350, min distance=90, dashed, looseness=3] (p12)
                            ;
                    \node[circle,inner sep=0pt] (p13) at (0, 8)
                    {\begin{tikzpicture}[scale=0.4]
                    \node at (0.9, 0.8) [draw, fill=black] {$.$};
                     \end{tikzpicture}}
                            edge[-latex, loop, out=170, in=110, min distance=60, looseness=2] (p13)
                            edge[-latex, loop, out=180, in=100, min distance=90, looseness=3] (p13)
                            edge[-latex, loop, out=280, in=340, min distance=60, dashed, looseness=2] (p13)
                            edge[-latex, loop, out=270, in=350, min distance=90, dashed, looseness=3] (p13)
                            ;
                    \node[circle,inner sep=0pt] (p14) at (0, 12)
                    {\begin{tikzpicture}[scale=0.4]
                    \node at (0.9, 0.8) [draw, fill=black] {$.$};
                     \end{tikzpicture}}
                            edge[-latex, loop, out=170, in=110, min distance=60, looseness=2] (p14)
                            edge[-latex, loop, out=180, in=100, min distance=90, looseness=3] (p14)
                            edge[-latex, loop, out=280, in=340, min distance=60, dashed, looseness=2] (p14)
                            edge[-latex, loop, out=270, in=350, min distance=90, dashed, looseness=3] (p14)
                            ;
                     \draw[style=semithick, dashed, -latex] (p12.south)--(p11.north);
                     \draw[style=semithick, dashed, -latex] (p13.south)--(p12.north);
                     \draw[style=semithick, dashed, -latex] (p14.south)--(p13.north);
                    \node[circle,inner sep=0pt] (p21) at (4, 0)
                    {\begin{tikzpicture}[scale=0.4]
                    \node at (0.9, 0.8) [draw, fill=black] {$.$};
                     \end{tikzpicture}}
                            edge[-latex, loop, out=170, in=110, min distance=60, looseness=2] (p21)
                            edge[-latex, loop, out=180, in=100, min distance=90, looseness=3] (p21)
                            edge[-latex, loop, out=280, in=340, min distance=60, dashed, looseness=2] (p21)
                            edge[-latex, loop, out=270, in=350, min distance=90, dashed, looseness=3] (p21)
                            ;
                    \node[circle,inner sep=0pt] (p22) at (4, 4)
                    {\begin{tikzpicture}[scale=0.4]
                    \node at (0.9, 0.8) [draw, fill=black] {$.$};
                     \end{tikzpicture}}
                            edge[-latex, loop, out=170, in=110, min distance=60, looseness=2] (p22)
                            edge[-latex, loop, out=180, in=100, min distance=90, looseness=3] (p22)
                            edge[-latex, loop, out=280, in=340, min distance=60, dashed, looseness=2] (p22)
                            edge[-latex, loop, out=270, in=350, min distance=90, dashed, looseness=3] (p22)
                            ;
                    \node[circle,inner sep=0pt] (p23) at (4, 8)
                    {\begin{tikzpicture}[scale=0.4]
                    \node at (0.9, 0.8) [draw, fill=black] {$.$};
                     \end{tikzpicture}}
                            edge[-latex, loop, out=170, in=110, min distance=60, looseness=2] (p23)
                            edge[-latex, loop, out=180, in=100, min distance=90, looseness=3] (p23)
                            edge[-latex, loop, out=280, in=340, min distance=60, dashed, looseness=2] (p23)
                            edge[-latex, loop, out=270, in=350, min distance=90, dashed, looseness=3] (p23)
                            ;
                    \node[circle,inner sep=0pt] (p24) at (4, 12)
                    {\begin{tikzpicture}[scale=0.4]
                    \node at (0.9, 0.8) [draw, fill=black] {$.$};
                     \end{tikzpicture}}
                            edge[-latex, loop, out=170, in=110, min distance=60, looseness=2] (p24)
                            edge[-latex, loop, out=180, in=100, min distance=90, looseness=3] (p24)
                            edge[-latex, loop, out=280, in=340, min distance=60, dashed, looseness=2] (p24)
                            edge[-latex, loop, out=270, in=350, min distance=90, dashed, looseness=3] (p24)
                            ;
                     \draw[style=semithick, dashed, -latex] (p22.south)--(p21.north);
                     \draw[style=semithick, dashed, -latex] (p23.south)--(p22.north);
                     \draw[style=semithick, dashed, -latex] (p24.south)--(p23.north);
                     \draw[style=semithick, -latex] (p21.west)--(p11.east);
                    \draw[style=semithick, -latex] (p22.west)--(p12.east);
                    \draw[style=semithick, -latex] (p23.west)--(p13.east);
                    \draw[style=semithick, -latex] (p24.west)--(p14.east);
                    \node[circle,inner sep=0pt] (p31) at (8, 0)
                    {\begin{tikzpicture}[scale=0.4]
                    \node at (0.9, 0.8) [draw, fill=black] {$.$};
                     \end{tikzpicture}}
                            edge[-latex, loop, out=170, in=110, min distance=60, looseness=2] (p31)
                            edge[-latex, loop, out=180, in=100, min distance=90, looseness=3] (p31)
                            edge[-latex, loop, out=280, in=340, min distance=60, dashed, looseness=2] (p31)
                            edge[-latex, loop, out=270, in=350, min distance=90, dashed, looseness=3] (p31)
                            ;
                    \node[circle,inner sep=0pt] (p32) at (8, 4)
                    {\begin{tikzpicture}[scale=0.4]
                    \node at (0.9, 0.8) [draw, fill=black] {$.$};
                     \end{tikzpicture}}
                            edge[-latex, loop, out=170, in=110, min distance=60, looseness=2] (p32)
                            edge[-latex, loop, out=180, in=100, min distance=90, looseness=3] (p32)
                            edge[-latex, loop, out=280, in=340, min distance=60, dashed, looseness=2] (p32)
                            edge[-latex, loop, out=270, in=350, min distance=90, dashed, looseness=3] (p32)
                            ;
                    \node[circle,inner sep=0pt] (p33) at (8, 8)
                    {\begin{tikzpicture}[scale=0.4]
                    \node at (0.9, 0.8) [draw, fill=black] {$.$};
                     \end{tikzpicture}}
                            edge[-latex, loop, out=170, in=110, min distance=60, looseness=2] (p33)
                            edge[-latex, loop, out=180, in=100, min distance=90, looseness=3] (p33)
                            edge[-latex, loop, out=280, in=340, min distance=60, dashed, looseness=2] (p33)
                            edge[-latex, loop, out=270, in=350, min distance=90, dashed, looseness=3] (p33)
                            ;
                    \node[circle,inner sep=0pt] (p34) at (8, 12)
                    {\begin{tikzpicture}[scale=0.4]
                    \node at (0.9, 0.8) [draw, fill=black] {$.$};
                     \end{tikzpicture}}
                            edge[-latex, loop, out=170, in=110, min distance=60, looseness=2] (p34)
                            edge[-latex, loop, out=180, in=100, min distance=90, looseness=3] (p34)
                            edge[-latex, loop, out=280, in=340, min distance=60, dashed, looseness=2] (p34)
                            edge[-latex, loop, out=270, in=350, min distance=90, dashed, looseness=3] (p34)
                            ;
                     \draw[style=semithick, dashed, -latex] (p32.south)--(p31.north);
                     \draw[style=semithick, dashed, -latex] (p33.south)--(p32.north);
                     \draw[style=semithick, dashed, -latex] (p34.south)--(p33.north);
                     \draw[style=semithick, -latex] (p31.west)--(p21.east);
                    \draw[style=semithick, -latex] (p32.west)--(p22.east);
                    \draw[style=semithick, -latex] (p33.west)--(p23.east);
                    \draw[style=semithick, -latex] (p34.west)--(p24.east);
                    \node[circle,inner sep=0pt] (p41) at (12, 0)
                    {\begin{tikzpicture}[scale=0.4]
                    \node at (0.9, 0.8) [draw, fill=black] {$.$};
                     \end{tikzpicture}}
                            edge[-latex, loop, out=170, in=110, min distance=60, looseness=2] (p41)
                            edge[-latex, loop, out=180, in=100, min distance=90, looseness=3] (p41)
                            edge[-latex, loop, out=280, in=340, min distance=60, dashed, looseness=2] (p41)
                            edge[-latex, loop, out=270, in=350, min distance=90, dashed, looseness=3] (p41)
                            ;
                    \node[circle,inner sep=0pt] (p42) at (12, 4)
                    {\begin{tikzpicture}[scale=0.4]
                    \node at (0.9, 0.8) [draw, fill=black] {$.$};
                     \end{tikzpicture}}
                            edge[-latex, loop, out=170, in=110, min distance=60, looseness=2] (p42)
                            edge[-latex, loop, out=180, in=100, min distance=90, looseness=3] (p42)
                            edge[-latex, loop, out=280, in=340, min distance=60, dashed, looseness=2] (p42)
                            edge[-latex, loop, out=270, in=350, min distance=90, dashed, looseness=3] (p42)
                            ;
                    \node[circle,inner sep=0pt] (p43) at (12, 8)
                    {\begin{tikzpicture}[scale=0.4]
                    \node at (0.9, 0.8) [draw, fill=black] {$.$};
                     \end{tikzpicture}}
                            edge[-latex, loop, out=170, in=110, min distance=60, looseness=2] (p43)
                            edge[-latex, loop, out=180, in=100, min distance=90, looseness=3] (p43)
                            edge[-latex, loop, out=280, in=340, min distance=60, dashed, looseness=2] (p43)
                            edge[-latex, loop, out=270, in=350, min distance=90, dashed, looseness=3] (p43)
                            ;
                    \node[circle,inner sep=0pt] (p44) at (12, 12)
                    {\begin{tikzpicture}[scale=0.4]
                    \node at (0.9, 0.8) [draw, fill=black] {$.$};
                     \end{tikzpicture}}
                            edge[-latex, loop, out=170, in=110, min distance=60, looseness=2] (p44)
                            edge[-latex, loop, out=180, in=100, min distance=90, looseness=3] (p44)
                            edge[-latex, loop, out=280, in=340, min distance=60, dashed, looseness=2] (p44)
                            edge[-latex, loop, out=270, in=350, min distance=90, dashed, looseness=3] (p44)
                            ;
                     \draw[style=semithick, dashed, -latex] (p42.south)--(p41.north);
                     \draw[style=semithick, dashed, -latex] (p43.south)--(p42.north);
                     \draw[style=semithick, dashed, -latex] (p44.south)--(p43.north);
                     \draw[style=semithick, -latex] (p41.west)--(p31.east);
                    \draw[style=semithick, -latex] (p42.west)--(p32.east);
                    \draw[style=semithick, -latex] (p43.west)--(p33.east);
                    \draw[style=semithick, -latex] (p44.west)--(p34.east);
                    \node[circle,inner sep=0pt] (p51) at (16, 0)
                    {\begin{tikzpicture}[scale=0.4]
                    \node at (0.9, 0.8) [draw, fill=black] {$.$};
                     \end{tikzpicture}}
                            edge[-latex, loop, out=170, in=110, min distance=60, looseness=2] (p51)
                            edge[-latex, loop, out=180, in=100, min distance=90, looseness=3] (p51)
                            edge[-latex, loop, out=280, in=340, min distance=60, dashed, looseness=2] (p51)
                            edge[-latex, loop, out=270, in=350, min distance=90, dashed, looseness=3] (p51)
                            ;
                    \node[circle,inner sep=0pt] (p52) at (16, 4)
                    {\begin{tikzpicture}[scale=0.4]
                    \node at (0.9, 0.8) [draw, fill=black] {$.$};
                     \end{tikzpicture}}
                            edge[-latex, loop, out=170, in=110, min distance=60, looseness=2] (p52)
                            edge[-latex, loop, out=180, in=100, min distance=90, looseness=3] (p52)
                            edge[-latex, loop, out=280, in=340, min distance=60, dashed, looseness=2] (p52)
                            edge[-latex, loop, out=270, in=350, min distance=90, dashed, looseness=3] (p52)
                            ;
                    \node[circle,inner sep=0pt] (p53) at (16, 8)
                    {\begin{tikzpicture}[scale=0.4]
                    \node at (0.9, 0.8) [draw, fill=black] {$.$};
                     \end{tikzpicture}}
                            edge[-latex, loop, out=170, in=110, min distance=60, looseness=2] (p53)
                            edge[-latex, loop, out=180, in=100, min distance=90, looseness=3] (p53)
                            edge[-latex, loop, out=280, in=340, min distance=60, dashed, looseness=2] (p53)
                            edge[-latex, loop, out=270, in=350, min distance=90, dashed, looseness=3] (p53)
                            ;
                    \node[circle,inner sep=0pt] (p54) at (16, 12)
                    {\begin{tikzpicture}[scale=0.4]
                    \node at (0.9, 0.8) [draw, fill=black] {$.$};
                     \end{tikzpicture}}
                            edge[-latex, loop, out=170, in=110, min distance=60, looseness=2] (p54)
                            edge[-latex, loop, out=180, in=100, min distance=90, looseness=3] (p54)
                            edge[-latex, loop, out=280, in=340, min distance=60, dashed, looseness=2] (p54)
                            edge[-latex, loop, out=270, in=350, min distance=90, dashed, looseness=3] (p54)
                            ;
                     \draw[style=semithick, dashed, -latex] (p52.south)--(p51.north);
                     \draw[style=semithick, dashed, -latex] (p53.south)--(p52.north);
                     \draw[style=semithick, dashed, -latex] (p54.south)--(p53.north);
                     \draw[style=semithick, -latex] (p51.west)--(p41.east);
                    \draw[style=semithick, -latex] (p52.west)--(p42.east);
                    \draw[style=semithick, -latex] (p53.west)--(p43.east);
                    \draw[style=semithick, -latex] (p54.west)--(p44.east);
                    \node[circle,inner sep=0pt] (p61) at (20, 0)
                    {\begin{tikzpicture}[scale=0.4]
                    \node at (0.9, 0.8) [draw, fill=black] {$.$};
                     \end{tikzpicture}}
                            edge[-latex, loop, out=170, in=110, min distance=60, looseness=2] (p61)
                            edge[-latex, loop, out=180, in=100, min distance=90, looseness=3] (p61)
                            edge[-latex, loop, out=280, in=340, min distance=60, dashed, looseness=2] (p61)
                            edge[-latex, loop, out=270, in=350, min distance=90, dashed, looseness=3] (p61)
                            ;
                    \node[circle,inner sep=0pt] (p62) at (20, 4)
                    {\begin{tikzpicture}[scale=0.4]
                    \node at (0.9, 0.8) [draw, fill=black] {$.$};
                     \end{tikzpicture}}
                            edge[-latex, loop, out=170, in=110, min distance=60, looseness=2] (p62)
                            edge[-latex, loop, out=180, in=100, min distance=90, looseness=3] (p62)
                            edge[-latex, loop, out=280, in=340, min distance=60, dashed, looseness=2] (p62)
                            edge[-latex, loop, out=270, in=350, min distance=90, dashed, looseness=3] (p62)
                            ;
                    \node[circle,inner sep=0pt] (p63) at (20, 8)
                    {\begin{tikzpicture}[scale=0.4]
                    \node at (0.9, 0.8) [draw, fill=black] {$.$};
                     \end{tikzpicture}}
                            edge[-latex, loop, out=170, in=110, min distance=60, looseness=2] (p63)
                            edge[-latex, loop, out=180, in=100, min distance=90, looseness=3] (p63)
                            edge[-latex, loop, out=280, in=340, min distance=60, dashed, looseness=2] (p63)
                            edge[-latex, loop, out=270, in=350, min distance=90, dashed, looseness=3] (p63)
                            ;
                    \node[circle,inner sep=0pt] (p64) at (20, 12)
                    {\begin{tikzpicture}[scale=0.4]
                    \node at (0.9, 0.8) [draw, fill=black] {$.$};
                     \end{tikzpicture}}
                            edge[-latex, loop, out=170, in=110, min distance=60, looseness=2] (p64)
                            edge[-latex, loop, out=180, in=100, min distance=90, looseness=3] (p64)
                            edge[-latex, loop, out=280, in=340, min distance=60, dashed, looseness=2] (p64)
                            edge[-latex, loop, out=270, in=350, min distance=90, dashed, looseness=3] (p64)
                            ;
                     \draw[style=semithick, dashed, -latex] (p62.south)--(p61.north);
                     \draw[style=semithick, dashed, -latex] (p63.south)--(p62.north);
                     \draw[style=semithick, dashed, -latex] (p64.south)--(p63.north);
                     \draw[style=semithick, -latex] (p61.west)--(p51.east);
                    \draw[style=semithick, -latex] (p62.west)--(p52.east);
                    \draw[style=semithick, -latex] (p63.west)--(p53.east);
                    \draw[style=semithick, -latex] (p64.west)--(p54.east);
                    \node at (-1.5,0) {.}; \node at (-1,0) {.}; \node at (-0.5,0) {.};
                     \node at (-1.5,4) {.}; \node at (-1,4) {.}; \node at (-0.5,4) {.};
                     \node at (-1.5,8) {.}; \node at (-1,8) {.}; \node at (-0.5,8) {.};
                     \node at (-1.5,12) {.}; \node at (-1,12) {.}; \node at (-0.5,12) {.};
                     \node at (21,0) {.}; \node at (21.5,0) {.}; \node at (22,0) {.};
                    \node at (21,4) {.}; \node at (21.5,4) {.}; \node at (22,4) {.};
                    \node at (21,8) {.}; \node at (21.5,8) {.}; \node at (22,8) {.};
                    \node at (21,12) {.}; \node at (21.5,12) {.}; \node at (22,12) {.};
                     \node at (0,14) {.}; \node at (0,14.5) {.}; \node at (0,15) {.};
                     \node at (4,14) {.}; \node at (4,14.5) {.}; \node at (4,15) {.};
                     \node at (8,14) {.}; \node at (8,14.5) {.}; \node at (8,15) {.};
                     \node at (12,14) {.}; \node at (12,14.5) {.}; \node at (12,15) {.};
                    \node at (16,14) {.}; \node at (16,14.5) {.}; \node at (16,15) {.};
                    \node at (20,14) {.}; \node at (20,14.5) {.}; \node at (20,15) {.};
                    \node at (0,-0.5) {.}; \node at (0,-1) {.}; \node at (0,-1.5) {.};
                    \node at (4,-0.5) {.}; \node at (4,-1) {.}; \node at (4,-1.5) {.};
                    \node at (8,-0.5) {.}; \node at (8,-1) {.}; \node at (8,-1.5) {.};
                    \node at (12,-0.5) {.}; \node at (12,-1) {.}; \node at (12,-1.5) {.};
                    \node at (16,-0.5) {.}; \node at (16,-1) {.}; \node at (16,-1.5) {.};
                    \node at (20,-0.5) {.}; \node at (20,-1) {.}; \node at (20,-1.5) {.};
\end{tikzpicture}
\]

\noindent and factorization rules induced from those in $\Gamma$ given in \eqref{eq:fpforgamma} to $\Lambda \times_c \mathbb{Z}^2$ using \eqref{eq:fpforspg}.

For each $(v,m) \in \Lambda^0$ one checks that $((f_1,m),(f_2,m),(g_1,m),(g_2,m))$ is a $(1,1)$-aperiodic quartet, so Proposition~\ref{prop:sap} shows that $\Lambda$ is strongly aperiodic. Moreover, the graph $\Lambda$ is not
a cartesian product by Remark~\ref{rem:whencart} since the edges $(f_i,m) \in \Lambda^{e_1}$ and $(g_i,m) \in \Lambda^{e_2}$, $i=1,2$ have source $(v,m)$, but have the factorization rule $(g_1,m)(f_2,m)=(f_1,m)(g_1,m)$ by \eqref{eq:fpforgamma} and \eqref{eq:fpforspg} since $c(g_i)=c(f_i)=0$ for $i=1,2$.

Now we describe the ideal structure of $C^{\ast}(\Lambda)$. In the following figure, the vertices in the gray shaded area shown form a hereditary collection of vertices in $\Lambda$. The region is also saturated since vertices $w_1=(v,m_1),w_2=(v,m_2)\in\Lambda^0$ receive both solid and dashed edges from vertices not in the shaded area. However, the complement of the shaded area is not a maximal tail since it fails to satisfy the condition (a) of Definition~\ref{Def-MT}. i.e. there is no common ancestor for $w_1$ and $w_2$.
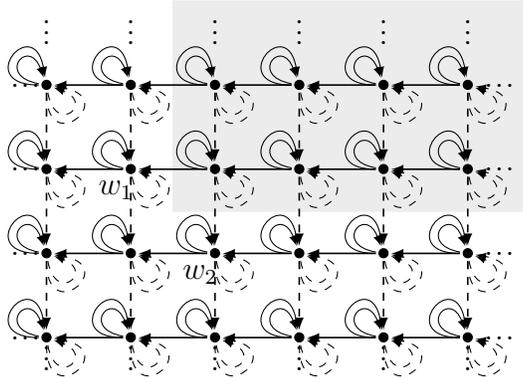
\[
\begin{tikzpicture}[scale=0.28]
                    \filldraw[color=gray!15!white] (6,6) rectangle (23,16);
                    \node at (3.3, 7) {$w_1$};
                    \node at (7.3, 3) {$w_2$};
                    \node[circle,inner sep=0pt] (p11) at (0, 0)
                    {\begin{tikzpicture}[scale=0.4]
                    \node at (0.9, 0.8) [draw, fill=black] {$.$};
                     \end{tikzpicture}}
                            edge[-latex, loop, out=170, in=110, min distance=60, looseness=2] (p11)
                            edge[-latex, loop, out=180, in=100, min distance=90, looseness=3] (p11)
                            edge[-latex, loop, out=280, in=340, min distance=60, dashed, looseness=2] (p11)
                            edge[-latex, loop, out=270, in=350, min distance=90, dashed, looseness=3] (p11)
                            ;
                    \node[circle,inner sep=0pt] (p12) at (0, 4)
                    {\begin{tikzpicture}[scale=0.4]
                    \node at (0.9, 0.8) [draw, fill=black] {$.$};
                     \end{tikzpicture}}
                            edge[-latex, loop, out=170, in=110, min distance=60, looseness=2] (p12)
                            edge[-latex, loop, out=180, in=100, min distance=90, looseness=3] (p12)
                            edge[-latex, loop, out=280, in=340, min distance=60, dashed, looseness=2] (p12)
                            edge[-latex, loop, out=270, in=350, min distance=90, dashed, looseness=3] (p12)
                            ;
                    \node[circle,inner sep=0pt] (p13) at (0, 8)
                    {\begin{tikzpicture}[scale=0.4]
                    \node at (0.9, 0.8) [draw, fill=black] {$.$};
                     \end{tikzpicture}}
                            edge[-latex, loop, out=170, in=110, min distance=60, looseness=2] (p13)
                            edge[-latex, loop, out=180, in=100, min distance=90, looseness=3] (p13)
                            edge[-latex, loop, out=280, in=340, min distance=60, dashed, looseness=2] (p13)
                            edge[-latex, loop, out=270, in=350, min distance=90, dashed, looseness=3] (p13)
                            ;
                    \node[circle,inner sep=0pt] (p14) at (0, 12)
                    {\begin{tikzpicture}[scale=0.4]
                    \node at (0.9, 0.8) [draw, fill=black] {$.$};
                     \end{tikzpicture}}
                            edge[-latex, loop, out=170, in=110, min distance=60, looseness=2] (p14)
                            edge[-latex, loop, out=180, in=100, min distance=90, looseness=3] (p14)
                            edge[-latex, loop, out=280, in=340, min distance=60, dashed, looseness=2] (p14)
                            edge[-latex, loop, out=270, in=350, min distance=90, dashed, looseness=3] (p14)
                            ;
                     \draw[style=semithick, dashed, -latex] (p12.south)--(p11.north);
                     \draw[style=semithick, dashed, -latex] (p13.south)--(p12.north);
                     \draw[style=semithick, dashed, -latex] (p14.south)--(p13.north);
                    \node[circle,inner sep=0pt] (p21) at (4, 0)
                    {\begin{tikzpicture}[scale=0.4]
                    \node at (0.9, 0.8) [draw, fill=black] {$.$};
                     \end{tikzpicture}}
                            edge[-latex, loop, out=170, in=110, min distance=60, looseness=2] (p21)
                            edge[-latex, loop, out=180, in=100, min distance=90, looseness=3] (p21)
                            edge[-latex, loop, out=280, in=340, min distance=60, dashed, looseness=2] (p21)
                            edge[-latex, loop, out=270, in=350, min distance=90, dashed, looseness=3] (p21)
                            ;
                    \node[circle,inner sep=0pt] (p22) at (4, 4)
                    {\begin{tikzpicture}[scale=0.4]
                    \node at (0.9, 0.8) [draw, fill=black] {$.$};
                     \end{tikzpicture}}
                            edge[-latex, loop, out=170, in=110, min distance=60, looseness=2] (p22)
                            edge[-latex, loop, out=180, in=100, min distance=90, looseness=3] (p22)
                            edge[-latex, loop, out=280, in=340, min distance=60, dashed, looseness=2] (p22)
                            edge[-latex, loop, out=270, in=350, min distance=90, dashed, looseness=3] (p22)
                            ;
                    \node[circle,inner sep=0pt] (p23) at (4, 8)
                    {\begin{tikzpicture}[scale=0.4]
                    \node at (0.9, 0.8) [draw, fill=black] {$.$};
                     \end{tikzpicture}}
                            edge[-latex, loop, out=170, in=110, min distance=60, looseness=2] (p23)
                            edge[-latex, loop, out=180, in=100, min distance=90, looseness=3] (p23)
                            edge[-latex, loop, out=280, in=340, min distance=60, dashed, looseness=2] (p23)
                            edge[-latex, loop, out=270, in=350, min distance=90, dashed, looseness=3] (p23)
                            ;
                    \node[circle,inner sep=0pt] (p24) at (4, 12)
                    {\begin{tikzpicture}[scale=0.4]
                    \node at (0.9, 0.8) [draw, fill=black] {$.$};
                     \end{tikzpicture}}
                            edge[-latex, loop, out=170, in=110, min distance=60, looseness=2] (p24)
                            edge[-latex, loop, out=180, in=100, min distance=90, looseness=3] (p24)
                            edge[-latex, loop, out=280, in=340, min distance=60, dashed, looseness=2] (p24)
                            edge[-latex, loop, out=270, in=350, min distance=90, dashed, looseness=3] (p24)
                            ;
                     \draw[style=semithick, dashed, -latex] (p22.south)--(p21.north);
                     \draw[style=semithick, dashed, -latex] (p23.south)--(p22.north);
                     \draw[style=semithick, dashed, -latex] (p24.south)--(p23.north);
                     \draw[style=semithick, -latex] (p21.west)--(p11.east);
                    \draw[style=semithick, -latex] (p22.west)--(p12.east);
                    \draw[style=semithick, -latex] (p23.west)--(p13.east);
                    \draw[style=semithick, -latex] (p24.west)--(p14.east);
                    \node[circle,inner sep=0pt] (p31) at (8, 0)
                    {\begin{tikzpicture}[scale=0.4]
                    \node at (0.9, 0.8) [draw, fill=black] {$.$};
                     \end{tikzpicture}}
                            edge[-latex, loop, out=170, in=110, min distance=60, looseness=2] (p31)
                            edge[-latex, loop, out=180, in=100, min distance=90, looseness=3] (p31)
                            edge[-latex, loop, out=280, in=340, min distance=60, dashed, looseness=2] (p31)
                            edge[-latex, loop, out=270, in=350, min distance=90, dashed, looseness=3] (p31)
                            ;
                    \node[circle,inner sep=0pt] (p32) at (8, 4)
                    {\begin{tikzpicture}[scale=0.4]
                    \node at (0.9, 0.8) [draw, fill=black] {$.$};
                     \end{tikzpicture}}
                            edge[-latex, loop, out=170, in=110, min distance=60, looseness=2] (p32)
                            edge[-latex, loop, out=180, in=100, min distance=90, looseness=3] (p32)
                            edge[-latex, loop, out=280, in=340, min distance=60, dashed, looseness=2] (p32)
                            edge[-latex, loop, out=270, in=350, min distance=90, dashed, looseness=3] (p32)
                            ;
                    \node[circle,inner sep=0pt] (p33) at (8, 8)
                    {\begin{tikzpicture}[scale=0.4]
                    \node at (0.9, 0.8) [draw, fill=black] {$.$};
                     \end{tikzpicture}}
                            edge[-latex, loop, out=170, in=110, min distance=60, looseness=2] (p33)
                            edge[-latex, loop, out=180, in=100, min distance=90, looseness=3] (p33)
                            edge[-latex, loop, out=280, in=340, min distance=60, dashed, looseness=2] (p33)
                            edge[-latex, loop, out=270, in=350, min distance=90, dashed, looseness=3] (p33)
                            ;
                    \node[circle,inner sep=0pt] (p34) at (8, 12)
                    {\begin{tikzpicture}[scale=0.4]
                    \node at (0.9, 0.8) [draw, fill=black] {$.$};
                     \end{tikzpicture}}
                            edge[-latex, loop, out=170, in=110, min distance=60, looseness=2] (p34)
                            edge[-latex, loop, out=180, in=100, min distance=90, looseness=3] (p34)
                            edge[-latex, loop, out=280, in=340, min distance=60, dashed, looseness=2] (p34)
                            edge[-latex, loop, out=270, in=350, min distance=90, dashed, looseness=3] (p34)
                            ;
                     \draw[style=semithick, dashed, -latex] (p32.south)--(p31.north);
                     \draw[style=semithick, dashed, -latex] (p33.south)--(p32.north);
                     \draw[style=semithick, dashed, -latex] (p34.south)--(p33.north);
                     \draw[style=semithick, -latex] (p31.west)--(p21.east);
                    \draw[style=semithick, -latex] (p32.west)--(p22.east);
                    \draw[style=semithick, -latex] (p33.west)--(p23.east);
                    \draw[style=semithick, -latex] (p34.west)--(p24.east);
                    \node[circle,inner sep=0pt] (p41) at (12, 0)
                    {\begin{tikzpicture}[scale=0.4]
                    \node at (0.9, 0.8) [draw, fill=black] {$.$};
                     \end{tikzpicture}}
                            edge[-latex, loop, out=170, in=110, min distance=60, looseness=2] (p41)
                            edge[-latex, loop, out=180, in=100, min distance=90, looseness=3] (p41)
                            edge[-latex, loop, out=280, in=340, min distance=60, dashed, looseness=2] (p41)
                            edge[-latex, loop, out=270, in=350, min distance=90, dashed, looseness=3] (p41)
                            ;
                    \node[circle,inner sep=0pt] (p42) at (12, 4)
                    {\begin{tikzpicture}[scale=0.4]
                    \node at (0.9, 0.8) [draw, fill=black] {$.$};
                     \end{tikzpicture}}
                            edge[-latex, loop, out=170, in=110, min distance=60, looseness=2] (p42)
                            edge[-latex, loop, out=180, in=100, min distance=90, looseness=3] (p42)
                            edge[-latex, loop, out=280, in=340, min distance=60, dashed, looseness=2] (p42)
                            edge[-latex, loop, out=270, in=350, min distance=90, dashed, looseness=3] (p42)
                            ;
                    \node[circle,inner sep=0pt] (p43) at (12, 8)
                    {\begin{tikzpicture}[scale=0.4]
                    \node at (0.9, 0.8) [draw, fill=black] {$.$};
                     \end{tikzpicture}}
                            edge[-latex, loop, out=170, in=110, min distance=60, looseness=2] (p43)
                            edge[-latex, loop, out=180, in=100, min distance=90, looseness=3] (p43)
                            edge[-latex, loop, out=280, in=340, min distance=60, dashed, looseness=2] (p43)
                            edge[-latex, loop, out=270, in=350, min distance=90, dashed, looseness=3] (p43)
                            ;
                    \node[circle,inner sep=0pt] (p44) at (12, 12)
                    {\begin{tikzpicture}[scale=0.4]
                    \node at (0.9, 0.8) [draw, fill=black] {$.$};
                     \end{tikzpicture}}
                            edge[-latex, loop, out=170, in=110, min distance=60, looseness=2] (p44)
                            edge[-latex, loop, out=180, in=100, min distance=90, looseness=3] (p44)
                            edge[-latex, loop, out=280, in=340, min distance=60, dashed, looseness=2] (p44)
                            edge[-latex, loop, out=270, in=350, min distance=90, dashed, looseness=3] (p44)
                            ;
                     \draw[style=semithick, dashed, -latex] (p42.south)--(p41.north);
                     \draw[style=semithick, dashed, -latex] (p43.south)--(p42.north);
                     \draw[style=semithick, dashed, -latex] (p44.south)--(p43.north);
                     \draw[style=semithick, -latex] (p41.west)--(p31.east);
                    \draw[style=semithick, -latex] (p42.west)--(p32.east);
                    \draw[style=semithick, -latex] (p43.west)--(p33.east);
                    \draw[style=semithick, -latex] (p44.west)--(p34.east);
                    \node[circle,inner sep=0pt] (p51) at (16, 0)
                    {\begin{tikzpicture}[scale=0.4]
                    \node at (0.9, 0.8) [draw, fill=black] {$.$};
                     \end{tikzpicture}}
                            edge[-latex, loop, out=170, in=110, min distance=60, looseness=2] (p51)
                            edge[-latex, loop, out=180, in=100, min distance=90, looseness=3] (p51)
                            edge[-latex, loop, out=280, in=340, min distance=60, dashed, looseness=2] (p51)
                            edge[-latex, loop, out=270, in=350, min distance=90, dashed, looseness=3] (p51)
                            ;
                    \node[circle,inner sep=0pt] (p52) at (16, 4)
                    {\begin{tikzpicture}[scale=0.4]
                    \node at (0.9, 0.8) [draw, fill=black] {$.$};
                     \end{tikzpicture}}
                            edge[-latex, loop, out=170, in=110, min distance=60, looseness=2] (p52)
                            edge[-latex, loop, out=180, in=100, min distance=90, looseness=3] (p52)
                            edge[-latex, loop, out=280, in=340, min distance=60, dashed, looseness=2] (p52)
                            edge[-latex, loop, out=270, in=350, min distance=90, dashed, looseness=3] (p52)
                            ;
                    \node[circle,inner sep=0pt] (p53) at (16, 8)
                    {\begin{tikzpicture}[scale=0.4]
                    \node at (0.9, 0.8) [draw, fill=black] {$.$};
                     \end{tikzpicture}}
                            edge[-latex, loop, out=170, in=110, min distance=60, looseness=2] (p53)
                            edge[-latex, loop, out=180, in=100, min distance=90, looseness=3] (p53)
                            edge[-latex, loop, out=280, in=340, min distance=60, dashed, looseness=2] (p53)
                            edge[-latex, loop, out=270, in=350, min distance=90, dashed, looseness=3] (p53)
                            ;
                    \node[circle,inner sep=0pt] (p54) at (16, 12)
                    {\begin{tikzpicture}[scale=0.4]
                    \node at (0.9, 0.8) [draw, fill=black] {$.$};
                     \end{tikzpicture}}
                            edge[-latex, loop, out=170, in=110, min distance=60, looseness=2] (p54)
                            edge[-latex, loop, out=180, in=100, min distance=90, looseness=3] (p54)
                            edge[-latex, loop, out=280, in=340, min distance=60, dashed, looseness=2] (p54)
                            edge[-latex, loop, out=270, in=350, min distance=90, dashed, looseness=3] (p54)
                            ;
                     \draw[style=semithick, dashed, -latex] (p52.south)--(p51.north);
                     \draw[style=semithick, dashed, -latex] (p53.south)--(p52.north);
                     \draw[style=semithick, dashed, -latex] (p54.south)--(p53.north);
                     \draw[style=semithick, -latex] (p51.west)--(p41.east);
                    \draw[style=semithick, -latex] (p52.west)--(p42.east);
                    \draw[style=semithick, -latex] (p53.west)--(p43.east);
                    \draw[style=semithick, -latex] (p54.west)--(p44.east);
                    \node[circle,inner sep=0pt] (p61) at (20, 0)
                    {\begin{tikzpicture}[scale=0.4]
                    \node at (0.9, 0.8) [draw, fill=black] {$.$};
                     \end{tikzpicture}}
                            edge[-latex, loop, out=170, in=110, min distance=60, looseness=2] (p61)
                            edge[-latex, loop, out=180, in=100, min distance=90, looseness=3] (p61)
                            edge[-latex, loop, out=280, in=340, min distance=60, dashed, looseness=2] (p61)
                            edge[-latex, loop, out=270, in=350, min distance=90, dashed, looseness=3] (p61)
                            ;
                    \node[circle,inner sep=0pt] (p62) at (20, 4)
                    {\begin{tikzpicture}[scale=0.4]
                    \node at (0.9, 0.8) [draw, fill=black] {$.$};
                     \end{tikzpicture}}
                            edge[-latex, loop, out=170, in=110, min distance=60, looseness=2] (p62)
                            edge[-latex, loop, out=180, in=100, min distance=90, looseness=3] (p62)
                            edge[-latex, loop, out=280, in=340, min distance=60, dashed, looseness=2] (p62)
                            edge[-latex, loop, out=270, in=350, min distance=90, dashed, looseness=3] (p62)
                            ;
                    \node[circle,inner sep=0pt] (p63) at (20, 8)
                    {\begin{tikzpicture}[scale=0.4]
                    \node at (0.9, 0.8) [draw, fill=black] {$.$};
                     \end{tikzpicture}}
                            edge[-latex, loop, out=170, in=110, min distance=60, looseness=2] (p63)
                            edge[-latex, loop, out=180, in=100, min distance=90, looseness=3] (p63)
                            edge[-latex, loop, out=280, in=340, min distance=60, dashed, looseness=2] (p63)
                            edge[-latex, loop, out=270, in=350, min distance=90, dashed, looseness=3] (p63)
                            ;
                    \node[circle,inner sep=0pt] (p64) at (20, 12)
                    {\begin{tikzpicture}[scale=0.4]
                    \node at (0.9, 0.8) [draw, fill=black] {$.$};
                     \end{tikzpicture}}
                            edge[-latex, loop, out=170, in=110, min distance=60, looseness=2] (p64)
                            edge[-latex, loop, out=180, in=100, min distance=90, looseness=3] (p64)
                            edge[-latex, loop, out=280, in=340, min distance=60, dashed, looseness=2] (p64)
                            edge[-latex, loop, out=270, in=350, min distance=90, dashed, looseness=3] (p64)
                            ;
                     \draw[style=semithick, dashed, -latex] (p62.south)--(p61.north);
                     \draw[style=semithick, dashed, -latex] (p63.south)--(p62.north);
                     \draw[style=semithick, dashed, -latex] (p64.south)--(p63.north);
                     \draw[style=semithick, -latex] (p61.west)--(p51.east);
                    \draw[style=semithick, -latex] (p62.west)--(p52.east);
                    \draw[style=semithick, -latex] (p63.west)--(p53.east);
                    \draw[style=semithick, -latex] (p64.west)--(p54.east);
                    \node at (-1.5,0) {.}; \node at (-1,0) {.}; \node at (-0.5,0) {.};
                     \node at (-1.5,4) {.}; \node at (-1,4) {.}; \node at (-0.5,4) {.};
                     \node at (-1.5,8) {.}; \node at (-1,8) {.}; \node at (-0.5,8) {.};
                     \node at (-1.5,12) {.}; \node at (-1,12) {.}; \node at (-0.5,12) {.};
                     \node at (21,0) {.}; \node at (21.5,0) {.}; \node at (22,0) {.};
                    \node at (21,4) {.}; \node at (21.5,4) {.}; \node at (22,4) {.};
                    \node at (21,8) {.}; \node at (21.5,8) {.}; \node at (22,8) {.};
                    \node at (21,12) {.}; \node at (21.5,12) {.}; \node at (22,12) {.};
                     \node at (0,14) {.}; \node at (0,14.5) {.}; \node at (0,15) {.};
                     \node at (4,14) {.}; \node at (4,14.5) {.}; \node at (4,15) {.};
                     \node at (8,14) {.}; \node at (8,14.5) {.}; \node at (8,15) {.};
                     \node at (12,14) {.}; \node at (12,14.5) {.}; \node at (12,15) {.};
                    \node at (16,14) {.}; \node at (16,14.5) {.}; \node at (16,15) {.};
                    \node at (20,14) {.}; \node at (20,14.5) {.}; \node at (20,15) {.};
                    \node at (0,-0.5) {.}; \node at (0,-1) {.}; \node at (0,-1.5) {.};
                    \node at (4,-0.5) {.}; \node at (4,-1) {.}; \node at (4,-1.5) {.};
                    \node at (8,-0.5) {.}; \node at (8,-1) {.}; \node at (8,-1.5) {.};
                    \node at (12,-0.5) {.}; \node at (12,-1) {.}; \node at (12,-1.5) {.};
                    \node at (16,-0.5) {.}; \node at (16,-1) {.}; \node at (16,-1.5) {.};
                    \node at (20,-0.5) {.}; \node at (20,-1) {.}; \node at (20,-1.5) {.};
\end{tikzpicture}
\]

\noindent
In the following figure, the vertices in the shaded areas form a hereditary subset of $\Lambda$. This region is saturated since vertex $w\in\Lambda^0$ receives both dashed and solid edges from vertices not in the shaded area. Also, the complement of the shaded region satisfies conditions (a), (b) and (c) of  Definition~\ref{Def-MT}.
\[
\begin{tikzpicture}[scale=0.28]
                    \filldraw[color=gray!15!white] (-2,6) rectangle (23,16);
                    \node at (7.3, 3) {$w$};
                    \node at (10,-3) {Figure 1};
                    \node[circle,inner sep=0pt] (p11) at (0, 0)
                    {\begin{tikzpicture}[scale=0.4]
                    \node at (0.9, 0.8) [draw, fill=black] {$.$};
                     \end{tikzpicture}}
                            edge[-latex, loop, out=170, in=110, min distance=60, looseness=2] (p11)
                            edge[-latex, loop, out=180, in=100, min distance=90, looseness=3] (p11)
                            edge[-latex, loop, out=280, in=340, min distance=60, dashed, looseness=2] (p11)
                            edge[-latex, loop, out=270, in=350, min distance=90, dashed, looseness=3] (p11)
                            ;
                    \node[circle,inner sep=0pt] (p12) at (0, 4)
                    {\begin{tikzpicture}[scale=0.4]
                    \node at (0.9, 0.8) [draw, fill=black] {$.$};
                     \end{tikzpicture}}
                            edge[-latex, loop, out=170, in=110, min distance=60, looseness=2] (p12)
                            edge[-latex, loop, out=180, in=100, min distance=90, looseness=3] (p12)
                            edge[-latex, loop, out=280, in=340, min distance=60, dashed, looseness=2] (p12)
                            edge[-latex, loop, out=270, in=350, min distance=90, dashed, looseness=3] (p12)
                            ;
                    \node[circle,inner sep=0pt] (p13) at (0, 8)
                    {\begin{tikzpicture}[scale=0.4]
                    \node at (0.9, 0.8) [draw, fill=black] {$.$};
                     \end{tikzpicture}}
                            edge[-latex, loop, out=170, in=110, min distance=60, looseness=2] (p13)
                            edge[-latex, loop, out=180, in=100, min distance=90, looseness=3] (p13)
                            edge[-latex, loop, out=280, in=340, min distance=60, dashed, looseness=2] (p13)
                            edge[-latex, loop, out=270, in=350, min distance=90, dashed, looseness=3] (p13)
                            ;
                    \node[circle,inner sep=0pt] (p14) at (0, 12)
                    {\begin{tikzpicture}[scale=0.4]
                    \node at (0.9, 0.8) [draw, fill=black] {$.$};
                     \end{tikzpicture}}
                            edge[-latex, loop, out=170, in=110, min distance=60, looseness=2] (p14)
                            edge[-latex, loop, out=180, in=100, min distance=90, looseness=3] (p14)
                            edge[-latex, loop, out=280, in=340, min distance=60, dashed, looseness=2] (p14)
                            edge[-latex, loop, out=270, in=350, min distance=90, dashed, looseness=3] (p14)
                            ;
                     \draw[style=semithick, dashed, -latex] (p12.south)--(p11.north);
                     \draw[style=semithick, dashed, -latex] (p13.south)--(p12.north);
                     \draw[style=semithick, dashed, -latex] (p14.south)--(p13.north);
                    \node[circle,inner sep=0pt] (p21) at (4, 0)
                    {\begin{tikzpicture}[scale=0.4]
                    \node at (0.9, 0.8) [draw, fill=black] {$.$};
                     \end{tikzpicture}}
                            edge[-latex, loop, out=170, in=110, min distance=60, looseness=2] (p21)
                            edge[-latex, loop, out=180, in=100, min distance=90, looseness=3] (p21)
                            edge[-latex, loop, out=280, in=340, min distance=60, dashed, looseness=2] (p21)
                            edge[-latex, loop, out=270, in=350, min distance=90, dashed, looseness=3] (p21)
                            ;
                    \node[circle,inner sep=0pt] (p22) at (4, 4)
                    {\begin{tikzpicture}[scale=0.4]
                    \node at (0.9, 0.8) [draw, fill=black] {$.$};
                     \end{tikzpicture}}
                            edge[-latex, loop, out=170, in=110, min distance=60, looseness=2] (p22)
                            edge[-latex, loop, out=180, in=100, min distance=90, looseness=3] (p22)
                            edge[-latex, loop, out=280, in=340, min distance=60, dashed, looseness=2] (p22)
                            edge[-latex, loop, out=270, in=350, min distance=90, dashed, looseness=3] (p22)
                            ;
                    \node[circle,inner sep=0pt] (p23) at (4, 8)
                    {\begin{tikzpicture}[scale=0.4]
                    \node at (0.9, 0.8) [draw, fill=black] {$.$};
                     \end{tikzpicture}}
                            edge[-latex, loop, out=170, in=110, min distance=60, looseness=2] (p23)
                            edge[-latex, loop, out=180, in=100, min distance=90, looseness=3] (p23)
                            edge[-latex, loop, out=280, in=340, min distance=60, dashed, looseness=2] (p23)
                            edge[-latex, loop, out=270, in=350, min distance=90, dashed, looseness=3] (p23)
                            ;
                    \node[circle,inner sep=0pt] (p24) at (4, 12)
                    {\begin{tikzpicture}[scale=0.4]
                    \node at (0.9, 0.8) [draw, fill=black] {$.$};
                     \end{tikzpicture}}
                            edge[-latex, loop, out=170, in=110, min distance=60, looseness=2] (p24)
                            edge[-latex, loop, out=180, in=100, min distance=90, looseness=3] (p24)
                            edge[-latex, loop, out=280, in=340, min distance=60, dashed, looseness=2] (p24)
                            edge[-latex, loop, out=270, in=350, min distance=90, dashed, looseness=3] (p24)
                            ;
                     \draw[style=semithick, dashed, -latex] (p22.south)--(p21.north);
                     \draw[style=semithick, dashed, -latex] (p23.south)--(p22.north);
                     \draw[style=semithick, dashed, -latex] (p24.south)--(p23.north);
                     \draw[style=semithick, -latex] (p21.west)--(p11.east);
                    \draw[style=semithick, -latex] (p22.west)--(p12.east);
                    \draw[style=semithick, -latex] (p23.west)--(p13.east);
                    \draw[style=semithick, -latex] (p24.west)--(p14.east);
                    \node[circle,inner sep=0pt] (p31) at (8, 0)
                    {\begin{tikzpicture}[scale=0.4]
                    \node at (0.9, 0.8) [draw, fill=black] {$.$};
                     \end{tikzpicture}}
                            edge[-latex, loop, out=170, in=110, min distance=60, looseness=2] (p31)
                            edge[-latex, loop, out=180, in=100, min distance=90, looseness=3] (p31)
                            edge[-latex, loop, out=280, in=340, min distance=60, dashed, looseness=2] (p31)
                            edge[-latex, loop, out=270, in=350, min distance=90, dashed, looseness=3] (p31)
                            ;
                    \node[circle,inner sep=0pt] (p32) at (8, 4)
                    {\begin{tikzpicture}[scale=0.4]
                    \node at (0.9, 0.8) [draw, fill=black] {$.$};
                     \end{tikzpicture}}
                            edge[-latex, loop, out=170, in=110, min distance=60, looseness=2] (p32)
                            edge[-latex, loop, out=180, in=100, min distance=90, looseness=3] (p32)
                            edge[-latex, loop, out=280, in=340, min distance=60, dashed, looseness=2] (p32)
                            edge[-latex, loop, out=270, in=350, min distance=90, dashed, looseness=3] (p32)
                            ;
                    \node[circle,inner sep=0pt] (p33) at (8, 8)
                    {\begin{tikzpicture}[scale=0.4]
                    \node at (0.9, 0.8) [draw, fill=black] {$.$};
                     \end{tikzpicture}}
                            edge[-latex, loop, out=170, in=110, min distance=60, looseness=2] (p33)
                            edge[-latex, loop, out=180, in=100, min distance=90, looseness=3] (p33)
                            edge[-latex, loop, out=280, in=340, min distance=60, dashed, looseness=2] (p33)
                            edge[-latex, loop, out=270, in=350, min distance=90, dashed, looseness=3] (p33)
                            ;
                    \node[circle,inner sep=0pt] (p34) at (8, 12)
                    {\begin{tikzpicture}[scale=0.4]
                    \node at (0.9, 0.8) [draw, fill=black] {$.$};
                     \end{tikzpicture}}
                            edge[-latex, loop, out=170, in=110, min distance=60, looseness=2] (p34)
                            edge[-latex, loop, out=180, in=100, min distance=90, looseness=3] (p34)
                            edge[-latex, loop, out=280, in=340, min distance=60, dashed, looseness=2] (p34)
                            edge[-latex, loop, out=270, in=350, min distance=90, dashed, looseness=3] (p34)
                            ;
                     \draw[style=semithick, dashed, -latex] (p32.south)--(p31.north);
                     \draw[style=semithick, dashed, -latex] (p33.south)--(p32.north);
                     \draw[style=semithick, dashed, -latex] (p34.south)--(p33.north);
                     \draw[style=semithick, -latex] (p31.west)--(p21.east);
                    \draw[style=semithick, -latex] (p32.west)--(p22.east);
                    \draw[style=semithick, -latex] (p33.west)--(p23.east);
                    \draw[style=semithick, -latex] (p34.west)--(p24.east);
                    \node[circle,inner sep=0pt] (p41) at (12, 0)
                    {\begin{tikzpicture}[scale=0.4]
                    \node at (0.9, 0.8) [draw, fill=black] {$.$};
                     \end{tikzpicture}}
                            edge[-latex, loop, out=170, in=110, min distance=60, looseness=2] (p41)
                            edge[-latex, loop, out=180, in=100, min distance=90, looseness=3] (p41)
                            edge[-latex, loop, out=280, in=340, min distance=60, dashed, looseness=2] (p41)
                            edge[-latex, loop, out=270, in=350, min distance=90, dashed, looseness=3] (p41)
                            ;
                    \node[circle,inner sep=0pt] (p42) at (12, 4)
                    {\begin{tikzpicture}[scale=0.4]
                    \node at (0.9, 0.8) [draw, fill=black] {$.$};
                     \end{tikzpicture}}
                            edge[-latex, loop, out=170, in=110, min distance=60, looseness=2] (p42)
                            edge[-latex, loop, out=180, in=100, min distance=90, looseness=3] (p42)
                            edge[-latex, loop, out=280, in=340, min distance=60, dashed, looseness=2] (p42)
                            edge[-latex, loop, out=270, in=350, min distance=90, dashed, looseness=3] (p42)
                            ;
                    \node[circle,inner sep=0pt] (p43) at (12, 8)
                    {\begin{tikzpicture}[scale=0.4]
                    \node at (0.9, 0.8) [draw, fill=black] {$.$};
                     \end{tikzpicture}}
                            edge[-latex, loop, out=170, in=110, min distance=60, looseness=2] (p43)
                            edge[-latex, loop, out=180, in=100, min distance=90, looseness=3] (p43)
                            edge[-latex, loop, out=280, in=340, min distance=60, dashed, looseness=2] (p43)
                            edge[-latex, loop, out=270, in=350, min distance=90, dashed, looseness=3] (p43)
                            ;
                    \node[circle,inner sep=0pt] (p44) at (12, 12)
                    {\begin{tikzpicture}[scale=0.4]
                    \node at (0.9, 0.8) [draw, fill=black] {$.$};
                     \end{tikzpicture}}
                            edge[-latex, loop, out=170, in=110, min distance=60, looseness=2] (p44)
                            edge[-latex, loop, out=180, in=100, min distance=90, looseness=3] (p44)
                            edge[-latex, loop, out=280, in=340, min distance=60, dashed, looseness=2] (p44)
                            edge[-latex, loop, out=270, in=350, min distance=90, dashed, looseness=3] (p44)
                            ;
                     \draw[style=semithick, dashed, -latex] (p42.south)--(p41.north);
                     \draw[style=semithick, dashed, -latex] (p43.south)--(p42.north);
                     \draw[style=semithick, dashed, -latex] (p44.south)--(p43.north);
                     \draw[style=semithick, -latex] (p41.west)--(p31.east);
                    \draw[style=semithick, -latex] (p42.west)--(p32.east);
                    \draw[style=semithick, -latex] (p43.west)--(p33.east);
                    \draw[style=semithick, -latex] (p44.west)--(p34.east);
                    \node[circle,inner sep=0pt] (p51) at (16, 0)
                    {\begin{tikzpicture}[scale=0.4]
                    \node at (0.9, 0.8) [draw, fill=black] {$.$};
                     \end{tikzpicture}}
                            edge[-latex, loop, out=170, in=110, min distance=60, looseness=2] (p51)
                            edge[-latex, loop, out=180, in=100, min distance=90, looseness=3] (p51)
                            edge[-latex, loop, out=280, in=340, min distance=60, dashed, looseness=2] (p51)
                            edge[-latex, loop, out=270, in=350, min distance=90, dashed, looseness=3] (p51)
                            ;
                    \node[circle,inner sep=0pt] (p52) at (16, 4)
                    {\begin{tikzpicture}[scale=0.4]
                    \node at (0.9, 0.8) [draw, fill=black] {$.$};
                     \end{tikzpicture}}
                            edge[-latex, loop, out=170, in=110, min distance=60, looseness=2] (p52)
                            edge[-latex, loop, out=180, in=100, min distance=90, looseness=3] (p52)
                            edge[-latex, loop, out=280, in=340, min distance=60, dashed, looseness=2] (p52)
                            edge[-latex, loop, out=270, in=350, min distance=90, dashed, looseness=3] (p52)
                            ;
                    \node[circle,inner sep=0pt] (p53) at (16, 8)
                    {\begin{tikzpicture}[scale=0.4]
                    \node at (0.9, 0.8) [draw, fill=black] {$.$};
                     \end{tikzpicture}}
                            edge[-latex, loop, out=170, in=110, min distance=60, looseness=2] (p53)
                            edge[-latex, loop, out=180, in=100, min distance=90, looseness=3] (p53)
                            edge[-latex, loop, out=280, in=340, min distance=60, dashed, looseness=2] (p53)
                            edge[-latex, loop, out=270, in=350, min distance=90, dashed, looseness=3] (p53)
                            ;
                    \node[circle,inner sep=0pt] (p54) at (16, 12)
                    {\begin{tikzpicture}[scale=0.4]
                    \node at (0.9, 0.8) [draw, fill=black] {$.$};
                     \end{tikzpicture}}
                            edge[-latex, loop, out=170, in=110, min distance=60, looseness=2] (p54)
                            edge[-latex, loop, out=180, in=100, min distance=90, looseness=3] (p54)
                            edge[-latex, loop, out=280, in=340, min distance=60, dashed, looseness=2] (p54)
                            edge[-latex, loop, out=270, in=350, min distance=90, dashed, looseness=3] (p54)
                            ;
                     \draw[style=semithick, dashed, -latex] (p52.south)--(p51.north);
                     \draw[style=semithick, dashed, -latex] (p53.south)--(p52.north);
                     \draw[style=semithick, dashed, -latex] (p54.south)--(p53.north);
                     \draw[style=semithick, -latex] (p51.west)--(p41.east);
                    \draw[style=semithick, -latex] (p52.west)--(p42.east);
                    \draw[style=semithick, -latex] (p53.west)--(p43.east);
                    \draw[style=semithick, -latex] (p54.west)--(p44.east);
                    \node[circle,inner sep=0pt] (p61) at (20, 0)
                    {\begin{tikzpicture}[scale=0.4]
                    \node at (0.9, 0.8) [draw, fill=black] {$.$};
                     \end{tikzpicture}}
                            edge[-latex, loop, out=170, in=110, min distance=60, looseness=2] (p61)
                            edge[-latex, loop, out=180, in=100, min distance=90, looseness=3] (p61)
                            edge[-latex, loop, out=280, in=340, min distance=60, dashed, looseness=2] (p61)
                            edge[-latex, loop, out=270, in=350, min distance=90, dashed, looseness=3] (p61)
                            ;
                    \node[circle,inner sep=0pt] (p62) at (20, 4)
                    {\begin{tikzpicture}[scale=0.4]
                    \node at (0.9, 0.8) [draw, fill=black] {$.$};
                     \end{tikzpicture}}
                            edge[-latex, loop, out=170, in=110, min distance=60, looseness=2] (p62)
                            edge[-latex, loop, out=180, in=100, min distance=90, looseness=3] (p62)
                            edge[-latex, loop, out=280, in=340, min distance=60, dashed, looseness=2] (p62)
                            edge[-latex, loop, out=270, in=350, min distance=90, dashed, looseness=3] (p62)
                            ;
                    \node[circle,inner sep=0pt] (p63) at (20, 8)
                    {\begin{tikzpicture}[scale=0.4]
                    \node at (0.9, 0.8) [draw, fill=black] {$.$};
                     \end{tikzpicture}}
                            edge[-latex, loop, out=170, in=110, min distance=60, looseness=2] (p63)
                            edge[-latex, loop, out=180, in=100, min distance=90, looseness=3] (p63)
                            edge[-latex, loop, out=280, in=340, min distance=60, dashed, looseness=2] (p63)
                            edge[-latex, loop, out=270, in=350, min distance=90, dashed, looseness=3] (p63)
                            ;
                    \node[circle,inner sep=0pt] (p64) at (20, 12)
                    {\begin{tikzpicture}[scale=0.4]
                    \node at (0.9, 0.8) [draw, fill=black] {$.$};
                     \end{tikzpicture}}
                            edge[-latex, loop, out=170, in=110, min distance=60, looseness=2] (p64)
                            edge[-latex, loop, out=180, in=100, min distance=90, looseness=3] (p64)
                            edge[-latex, loop, out=280, in=340, min distance=60, dashed, looseness=2] (p64)
                            edge[-latex, loop, out=270, in=350, min distance=90, dashed, looseness=3] (p64)
                            ;
                     \draw[style=semithick, dashed, -latex] (p62.south)--(p61.north);
                     \draw[style=semithick, dashed, -latex] (p63.south)--(p62.north);
                     \draw[style=semithick, dashed, -latex] (p64.south)--(p63.north);
                     \draw[style=semithick, -latex] (p61.west)--(p51.east);
                    \draw[style=semithick, -latex] (p62.west)--(p52.east);
                    \draw[style=semithick, -latex] (p63.west)--(p53.east);
                    \draw[style=semithick, -latex] (p64.west)--(p54.east);
                    \node at (-1.5,0) {.}; \node at (-1,0) {.}; \node at (-0.5,0) {.};
                     \node at (-1.5,4) {.}; \node at (-1,4) {.}; \node at (-0.5,4) {.};
                     \node at (-1.5,8) {.}; \node at (-1,8) {.}; \node at (-0.5,8) {.};
                     \node at (-1.5,12) {.}; \node at (-1,12) {.}; \node at (-0.5,12) {.};
                     \node at (21,0) {.}; \node at (21.5,0) {.}; \node at (22,0) {.};
                    \node at (21,4) {.}; \node at (21.5,4) {.}; \node at (22,4) {.};
                    \node at (21,8) {.}; \node at (21.5,8) {.}; \node at (22,8) {.};
                    \node at (21,12) {.}; \node at (21.5,12) {.}; \node at (22,12) {.};
                     \node at (0,14) {.}; \node at (0,14.5) {.}; \node at (0,15) {.};
                     \node at (4,14) {.}; \node at (4,14.5) {.}; \node at (4,15) {.};
                     \node at (8,14) {.}; \node at (8,14.5) {.}; \node at (8,15) {.};
                     \node at (12,14) {.}; \node at (12,14.5) {.}; \node at (12,15) {.};
                    \node at (16,14) {.}; \node at (16,14.5) {.}; \node at (16,15) {.};
                    \node at (20,14) {.}; \node at (20,14.5) {.}; \node at (20,15) {.};
                    \node at (0,-0.5) {.}; \node at (0,-1) {.}; \node at (0,-1.5) {.};
                    \node at (4,-0.5) {.}; \node at (4,-1) {.}; \node at (4,-1.5) {.};
                    \node at (8,-0.5) {.}; \node at (8,-1) {.}; \node at (8,-1.5) {.};
                    \node at (12,-0.5) {.}; \node at (12,-1) {.}; \node at (12,-1.5) {.};
                    \node at (16,-0.5) {.}; \node at (16,-1) {.}; \node at (16,-1.5) {.};
                    \node at (20,-0.5) {.}; \node at (20,-1) {.}; \node at (20,-1.5) {.};
\end{tikzpicture}
\]

\noindent
Similarly, the vertices in the shaded area of the following figure form a hereditary and saturated subset in $\Lambda$, and the complement is a maximal tail.
\[
\begin{tikzpicture}[scale=0.28]
                    \filldraw[color=gray!15!white] (10,-2) rectangle (23,16);
                    \node at (7.3, 3) {$w$};
                    \node at (10,-3) {Figure 2};
                    \node[circle,inner sep=0pt] (p11) at (0, 0)
                    {\begin{tikzpicture}[scale=0.4]
                    \node at (0.9, 0.8) [draw, fill=black] {$.$};
                     \end{tikzpicture}}
                            edge[-latex, loop, out=170, in=110, min distance=60, looseness=2] (p11)
                            edge[-latex, loop, out=180, in=100, min distance=90, looseness=3] (p11)
                            edge[-latex, loop, out=280, in=340, min distance=60, dashed, looseness=2] (p11)
                            edge[-latex, loop, out=270, in=350, min distance=90, dashed, looseness=3] (p11)
                            ;
                    \node[circle,inner sep=0pt] (p12) at (0, 4)
                    {\begin{tikzpicture}[scale=0.4]
                    \node at (0.9, 0.8) [draw, fill=black] {$.$};
                     \end{tikzpicture}}
                            edge[-latex, loop, out=170, in=110, min distance=60, looseness=2] (p12)
                            edge[-latex, loop, out=180, in=100, min distance=90, looseness=3] (p12)
                            edge[-latex, loop, out=280, in=340, min distance=60, dashed, looseness=2] (p12)
                            edge[-latex, loop, out=270, in=350, min distance=90, dashed, looseness=3] (p12)
                            ;
                    \node[circle,inner sep=0pt] (p13) at (0, 8)
                    {\begin{tikzpicture}[scale=0.4]
                    \node at (0.9, 0.8) [draw, fill=black] {$.$};
                     \end{tikzpicture}}
                            edge[-latex, loop, out=170, in=110, min distance=60, looseness=2] (p13)
                            edge[-latex, loop, out=180, in=100, min distance=90, looseness=3] (p13)
                            edge[-latex, loop, out=280, in=340, min distance=60, dashed, looseness=2] (p13)
                            edge[-latex, loop, out=270, in=350, min distance=90, dashed, looseness=3] (p13)
                            ;
                    \node[circle,inner sep=0pt] (p14) at (0, 12)
                    {\begin{tikzpicture}[scale=0.4]
                    \node at (0.9, 0.8) [draw, fill=black] {$.$};
                     \end{tikzpicture}}
                            edge[-latex, loop, out=170, in=110, min distance=60, looseness=2] (p14)
                            edge[-latex, loop, out=180, in=100, min distance=90, looseness=3] (p14)
                            edge[-latex, loop, out=280, in=340, min distance=60, dashed, looseness=2] (p14)
                            edge[-latex, loop, out=270, in=350, min distance=90, dashed, looseness=3] (p14)
                            ;
                     \draw[style=semithick, dashed, -latex] (p12.south)--(p11.north);
                     \draw[style=semithick, dashed, -latex] (p13.south)--(p12.north);
                     \draw[style=semithick, dashed, -latex] (p14.south)--(p13.north);
                    \node[circle,inner sep=0pt] (p21) at (4, 0)
                    {\begin{tikzpicture}[scale=0.4]
                    \node at (0.9, 0.8) [draw, fill=black] {$.$};
                     \end{tikzpicture}}
                            edge[-latex, loop, out=170, in=110, min distance=60, looseness=2] (p21)
                            edge[-latex, loop, out=180, in=100, min distance=90, looseness=3] (p21)
                            edge[-latex, loop, out=280, in=340, min distance=60, dashed, looseness=2] (p21)
                            edge[-latex, loop, out=270, in=350, min distance=90, dashed, looseness=3] (p21)
                            ;
                    \node[circle,inner sep=0pt] (p22) at (4, 4)
                    {\begin{tikzpicture}[scale=0.4]
                    \node at (0.9, 0.8) [draw, fill=black] {$.$};
                     \end{tikzpicture}}
                            edge[-latex, loop, out=170, in=110, min distance=60, looseness=2] (p22)
                            edge[-latex, loop, out=180, in=100, min distance=90, looseness=3] (p22)
                            edge[-latex, loop, out=280, in=340, min distance=60, dashed, looseness=2] (p22)
                            edge[-latex, loop, out=270, in=350, min distance=90, dashed, looseness=3] (p22)
                            ;
                    \node[circle,inner sep=0pt] (p23) at (4, 8)
                    {\begin{tikzpicture}[scale=0.4]
                    \node at (0.9, 0.8) [draw, fill=black] {$.$};
                     \end{tikzpicture}}
                            edge[-latex, loop, out=170, in=110, min distance=60, looseness=2] (p23)
                            edge[-latex, loop, out=180, in=100, min distance=90, looseness=3] (p23)
                            edge[-latex, loop, out=280, in=340, min distance=60, dashed, looseness=2] (p23)
                            edge[-latex, loop, out=270, in=350, min distance=90, dashed, looseness=3] (p23)
                            ;
                    \node[circle,inner sep=0pt] (p24) at (4, 12)
                    {\begin{tikzpicture}[scale=0.4]
                    \node at (0.9, 0.8) [draw, fill=black] {$.$};
                     \end{tikzpicture}}
                            edge[-latex, loop, out=170, in=110, min distance=60, looseness=2] (p24)
                            edge[-latex, loop, out=180, in=100, min distance=90, looseness=3] (p24)
                            edge[-latex, loop, out=280, in=340, min distance=60, dashed, looseness=2] (p24)
                            edge[-latex, loop, out=270, in=350, min distance=90, dashed, looseness=3] (p24)
                            ;
                     \draw[style=semithick, dashed, -latex] (p22.south)--(p21.north);
                     \draw[style=semithick, dashed, -latex] (p23.south)--(p22.north);
                     \draw[style=semithick, dashed, -latex] (p24.south)--(p23.north);
                     \draw[style=semithick, -latex] (p21.west)--(p11.east);
                    \draw[style=semithick, -latex] (p22.west)--(p12.east);
                    \draw[style=semithick, -latex] (p23.west)--(p13.east);
                    \draw[style=semithick, -latex] (p24.west)--(p14.east);
                    \node[circle,inner sep=0pt] (p31) at (8, 0)
                    {\begin{tikzpicture}[scale=0.4]
                    \node at (0.9, 0.8) [draw, fill=black] {$.$};
                     \end{tikzpicture}}
                            edge[-latex, loop, out=170, in=110, min distance=60, looseness=2] (p31)
                            edge[-latex, loop, out=180, in=100, min distance=90, looseness=3] (p31)
                            edge[-latex, loop, out=280, in=340, min distance=60, dashed, looseness=2] (p31)
                            edge[-latex, loop, out=270, in=350, min distance=90, dashed, looseness=3] (p31)
                            ;
                    \node[circle,inner sep=0pt] (p32) at (8, 4)
                    {\begin{tikzpicture}[scale=0.4]
                    \node at (0.9, 0.8) [draw, fill=black] {$.$};
                     \end{tikzpicture}}
                            edge[-latex, loop, out=170, in=110, min distance=60, looseness=2] (p32)
                            edge[-latex, loop, out=180, in=100, min distance=90, looseness=3] (p32)
                            edge[-latex, loop, out=280, in=340, min distance=60, dashed, looseness=2] (p32)
                            edge[-latex, loop, out=270, in=350, min distance=90, dashed, looseness=3] (p32)
                            ;
                    \node[circle,inner sep=0pt] (p33) at (8, 8)
                    {\begin{tikzpicture}[scale=0.4]
                    \node at (0.9, 0.8) [draw, fill=black] {$.$};
                     \end{tikzpicture}}
                            edge[-latex, loop, out=170, in=110, min distance=60, looseness=2] (p33)
                            edge[-latex, loop, out=180, in=100, min distance=90, looseness=3] (p33)
                            edge[-latex, loop, out=280, in=340, min distance=60, dashed, looseness=2] (p33)
                            edge[-latex, loop, out=270, in=350, min distance=90, dashed, looseness=3] (p33)
                            ;
                    \node[circle,inner sep=0pt] (p34) at (8, 12)
                    {\begin{tikzpicture}[scale=0.4]
                    \node at (0.9, 0.8) [draw, fill=black] {$.$};
                     \end{tikzpicture}}
                            edge[-latex, loop, out=170, in=110, min distance=60, looseness=2] (p34)
                            edge[-latex, loop, out=180, in=100, min distance=90, looseness=3] (p34)
                            edge[-latex, loop, out=280, in=340, min distance=60, dashed, looseness=2] (p34)
                            edge[-latex, loop, out=270, in=350, min distance=90, dashed, looseness=3] (p34)
                            ;
                     \draw[style=semithick, dashed, -latex] (p32.south)--(p31.north);
                     \draw[style=semithick, dashed, -latex] (p33.south)--(p32.north);
                     \draw[style=semithick, dashed, -latex] (p34.south)--(p33.north);
                     \draw[style=semithick, -latex] (p31.west)--(p21.east);
                    \draw[style=semithick, -latex] (p32.west)--(p22.east);
                    \draw[style=semithick, -latex] (p33.west)--(p23.east);
                    \draw[style=semithick, -latex] (p34.west)--(p24.east);
                    \node[circle,inner sep=0pt] (p41) at (12, 0)
                    {\begin{tikzpicture}[scale=0.4]
                    \node at (0.9, 0.8) [draw, fill=black] {$.$};
                     \end{tikzpicture}}
                            edge[-latex, loop, out=170, in=110, min distance=60, looseness=2] (p41)
                            edge[-latex, loop, out=180, in=100, min distance=90, looseness=3] (p41)
                            edge[-latex, loop, out=280, in=340, min distance=60, dashed, looseness=2] (p41)
                            edge[-latex, loop, out=270, in=350, min distance=90, dashed, looseness=3] (p41)
                            ;
                    \node[circle,inner sep=0pt] (p42) at (12, 4)
                    {\begin{tikzpicture}[scale=0.4]
                    \node at (0.9, 0.8) [draw, fill=black] {$.$};
                     \end{tikzpicture}}
                            edge[-latex, loop, out=170, in=110, min distance=60, looseness=2] (p42)
                            edge[-latex, loop, out=180, in=100, min distance=90, looseness=3] (p42)
                            edge[-latex, loop, out=280, in=340, min distance=60, dashed, looseness=2] (p42)
                            edge[-latex, loop, out=270, in=350, min distance=90, dashed, looseness=3] (p42)
                            ;
                    \node[circle,inner sep=0pt] (p43) at (12, 8)
                    {\begin{tikzpicture}[scale=0.4]
                    \node at (0.9, 0.8) [draw, fill=black] {$.$};
                     \end{tikzpicture}}
                            edge[-latex, loop, out=170, in=110, min distance=60, looseness=2] (p43)
                            edge[-latex, loop, out=180, in=100, min distance=90, looseness=3] (p43)
                            edge[-latex, loop, out=280, in=340, min distance=60, dashed, looseness=2] (p43)
                            edge[-latex, loop, out=270, in=350, min distance=90, dashed, looseness=3] (p43)
                            ;
                    \node[circle,inner sep=0pt] (p44) at (12, 12)
                    {\begin{tikzpicture}[scale=0.4]
                    \node at (0.9, 0.8) [draw, fill=black] {$.$};
                     \end{tikzpicture}}
                            edge[-latex, loop, out=170, in=110, min distance=60, looseness=2] (p44)
                            edge[-latex, loop, out=180, in=100, min distance=90, looseness=3] (p44)
                            edge[-latex, loop, out=280, in=340, min distance=60, dashed, looseness=2] (p44)
                            edge[-latex, loop, out=270, in=350, min distance=90, dashed, looseness=3] (p44)
                            ;
                     \draw[style=semithick, dashed, -latex] (p42.south)--(p41.north);
                     \draw[style=semithick, dashed, -latex] (p43.south)--(p42.north);
                     \draw[style=semithick, dashed, -latex] (p44.south)--(p43.north);
                     \draw[style=semithick, -latex] (p41.west)--(p31.east);
                    \draw[style=semithick, -latex] (p42.west)--(p32.east);
                    \draw[style=semithick, -latex] (p43.west)--(p33.east);
                    \draw[style=semithick, -latex] (p44.west)--(p34.east);
                    \node[circle,inner sep=0pt] (p51) at (16, 0)
                    {\begin{tikzpicture}[scale=0.4]
                    \node at (0.9, 0.8) [draw, fill=black] {$.$};
                     \end{tikzpicture}}
                            edge[-latex, loop, out=170, in=110, min distance=60, looseness=2] (p51)
                            edge[-latex, loop, out=180, in=100, min distance=90, looseness=3] (p51)
                            edge[-latex, loop, out=280, in=340, min distance=60, dashed, looseness=2] (p51)
                            edge[-latex, loop, out=270, in=350, min distance=90, dashed, looseness=3] (p51)
                            ;
                    \node[circle,inner sep=0pt] (p52) at (16, 4)
                    {\begin{tikzpicture}[scale=0.4]
                    \node at (0.9, 0.8) [draw, fill=black] {$.$};
                     \end{tikzpicture}}
                            edge[-latex, loop, out=170, in=110, min distance=60, looseness=2] (p52)
                            edge[-latex, loop, out=180, in=100, min distance=90, looseness=3] (p52)
                            edge[-latex, loop, out=280, in=340, min distance=60, dashed, looseness=2] (p52)
                            edge[-latex, loop, out=270, in=350, min distance=90, dashed, looseness=3] (p52)
                            ;
                    \node[circle,inner sep=0pt] (p53) at (16, 8)
                    {\begin{tikzpicture}[scale=0.4]
                    \node at (0.9, 0.8) [draw, fill=black] {$.$};
                     \end{tikzpicture}}
                            edge[-latex, loop, out=170, in=110, min distance=60, looseness=2] (p53)
                            edge[-latex, loop, out=180, in=100, min distance=90, looseness=3] (p53)
                            edge[-latex, loop, out=280, in=340, min distance=60, dashed, looseness=2] (p53)
                            edge[-latex, loop, out=270, in=350, min distance=90, dashed, looseness=3] (p53)
                            ;
                    \node[circle,inner sep=0pt] (p54) at (16, 12)
                    {\begin{tikzpicture}[scale=0.4]
                    \node at (0.9, 0.8) [draw, fill=black] {$.$};
                     \end{tikzpicture}}
                            edge[-latex, loop, out=170, in=110, min distance=60, looseness=2] (p54)
                            edge[-latex, loop, out=180, in=100, min distance=90, looseness=3] (p54)
                            edge[-latex, loop, out=280, in=340, min distance=60, dashed, looseness=2] (p54)
                            edge[-latex, loop, out=270, in=350, min distance=90, dashed, looseness=3] (p54)
                            ;
                     \draw[style=semithick, dashed, -latex] (p52.south)--(p51.north);
                     \draw[style=semithick, dashed, -latex] (p53.south)--(p52.north);
                     \draw[style=semithick, dashed, -latex] (p54.south)--(p53.north);
                     \draw[style=semithick, -latex] (p51.west)--(p41.east);
                    \draw[style=semithick, -latex] (p52.west)--(p42.east);
                    \draw[style=semithick, -latex] (p53.west)--(p43.east);
                    \draw[style=semithick, -latex] (p54.west)--(p44.east);
                    \node[circle,inner sep=0pt] (p61) at (20, 0)
                    {\begin{tikzpicture}[scale=0.4]
                    \node at (0.9, 0.8) [draw, fill=black] {$.$};
                     \end{tikzpicture}}
                            edge[-latex, loop, out=170, in=110, min distance=60, looseness=2] (p61)
                            edge[-latex, loop, out=180, in=100, min distance=90, looseness=3] (p61)
                            edge[-latex, loop, out=280, in=340, min distance=60, dashed, looseness=2] (p61)
                            edge[-latex, loop, out=270, in=350, min distance=90, dashed, looseness=3] (p61)
                            ;
                    \node[circle,inner sep=0pt] (p62) at (20, 4)
                    {\begin{tikzpicture}[scale=0.4]
                    \node at (0.9, 0.8) [draw, fill=black] {$.$};
                     \end{tikzpicture}}
                            edge[-latex, loop, out=170, in=110, min distance=60, looseness=2] (p62)
                            edge[-latex, loop, out=180, in=100, min distance=90, looseness=3] (p62)
                            edge[-latex, loop, out=280, in=340, min distance=60, dashed, looseness=2] (p62)
                            edge[-latex, loop, out=270, in=350, min distance=90, dashed, looseness=3] (p62)
                            ;
                    \node[circle,inner sep=0pt] (p63) at (20, 8)
                    {\begin{tikzpicture}[scale=0.4]
                    \node at (0.9, 0.8) [draw, fill=black] {$.$};
                     \end{tikzpicture}}
                            edge[-latex, loop, out=170, in=110, min distance=60, looseness=2] (p63)
                            edge[-latex, loop, out=180, in=100, min distance=90, looseness=3] (p63)
                            edge[-latex, loop, out=280, in=340, min distance=60, dashed, looseness=2] (p63)
                            edge[-latex, loop, out=270, in=350, min distance=90, dashed, looseness=3] (p63)
                            ;
                    \node[circle,inner sep=0pt] (p64) at (20, 12)
                    {\begin{tikzpicture}[scale=0.4]
                    \node at (0.9, 0.8) [draw, fill=black] {$.$};
                     \end{tikzpicture}}
                            edge[-latex, loop, out=170, in=110, min distance=60, looseness=2] (p64)
                            edge[-latex, loop, out=180, in=100, min distance=90, looseness=3] (p64)
                            edge[-latex, loop, out=280, in=340, min distance=60, dashed, looseness=2] (p64)
                            edge[-latex, loop, out=270, in=350, min distance=90, dashed, looseness=3] (p64)
                            ;
                     \draw[style=semithick, dashed, -latex] (p62.south)--(p61.north);
                     \draw[style=semithick, dashed, -latex] (p63.south)--(p62.north);
                     \draw[style=semithick, dashed, -latex] (p64.south)--(p63.north);
                     \draw[style=semithick, -latex] (p61.west)--(p51.east);
                    \draw[style=semithick, -latex] (p62.west)--(p52.east);
                    \draw[style=semithick, -latex] (p63.west)--(p53.east);
                    \draw[style=semithick, -latex] (p64.west)--(p54.east);
                    \node at (-1.5,0) {.}; \node at (-1,0) {.}; \node at (-0.5,0) {.};
                     \node at (-1.5,4) {.}; \node at (-1,4) {.}; \node at (-0.5,4) {.};
                     \node at (-1.5,8) {.}; \node at (-1,8) {.}; \node at (-0.5,8) {.};
                     \node at (-1.5,12) {.}; \node at (-1,12) {.}; \node at (-0.5,12) {.};
                     \node at (21,0) {.}; \node at (21.5,0) {.}; \node at (22,0) {.};
                    \node at (21,4) {.}; \node at (21.5,4) {.}; \node at (22,4) {.};
                    \node at (21,8) {.}; \node at (21.5,8) {.}; \node at (22,8) {.};
                    \node at (21,12) {.}; \node at (21.5,12) {.}; \node at (22,12) {.};
                     \node at (0,14) {.}; \node at (0,14.5) {.}; \node at (0,15) {.};
                     \node at (4,14) {.}; \node at (4,14.5) {.}; \node at (4,15) {.};
                     \node at (8,14) {.}; \node at (8,14.5) {.}; \node at (8,15) {.};
                     \node at (12,14) {.}; \node at (12,14.5) {.}; \node at (12,15) {.};
                    \node at (16,14) {.}; \node at (16,14.5) {.}; \node at (16,15) {.};
                    \node at (20,14) {.}; \node at (20,14.5) {.}; \node at (20,15) {.};
                    \node at (0,-0.5) {.}; \node at (0,-1) {.}; \node at (0,-1.5) {.};
                    \node at (4,-0.5) {.}; \node at (4,-1) {.}; \node at (4,-1.5) {.};
                    \node at (8,-0.5) {.}; \node at (8,-1) {.}; \node at (8,-1.5) {.};
                    \node at (12,-0.5) {.}; \node at (12,-1) {.}; \node at (12,-1.5) {.};
                    \node at (16,-0.5) {.}; \node at (16,-1) {.}; \node at (16,-1.5) {.};
                    \node at (20,-0.5) {.}; \node at (20,-1) {.}; \node at (20,-1.5) {.};
\end{tikzpicture}
\]

\noindent
The following figure shows the last type of the hereditary and saturated subset, the shaded area, of $\Lambda$, of which the complement is a maximal tail.
\[
\begin{tikzpicture}[scale=0.28]
                    \filldraw[color=gray!15!white] (10,-2) rectangle (23,16);
                    \filldraw[color=gray!15!white] (-2,6) rectangle (23,16);
                    \node at (7.3, 3) {$w$};
                    \node at (10,-3) {Figure 3};
                    \node[circle,inner sep=0pt] (p11) at (0, 0)
                    {\begin{tikzpicture}[scale=0.4]
                    \node at (0.9, 0.8) [draw, fill=black] {$.$};
                     \end{tikzpicture}}
                            edge[-latex, loop, out=170, in=110, min distance=60, looseness=2] (p11)
                            edge[-latex, loop, out=180, in=100, min distance=90, looseness=3] (p11)
                            edge[-latex, loop, out=280, in=340, min distance=60, dashed, looseness=2] (p11)
                            edge[-latex, loop, out=270, in=350, min distance=90, dashed, looseness=3] (p11)
                            ;
                    \node[circle,inner sep=0pt] (p12) at (0, 4)
                    {\begin{tikzpicture}[scale=0.4]
                    \node at (0.9, 0.8) [draw, fill=black] {$.$};
                     \end{tikzpicture}}
                            edge[-latex, loop, out=170, in=110, min distance=60, looseness=2] (p12)
                            edge[-latex, loop, out=180, in=100, min distance=90, looseness=3] (p12)
                            edge[-latex, loop, out=280, in=340, min distance=60, dashed, looseness=2] (p12)
                            edge[-latex, loop, out=270, in=350, min distance=90, dashed, looseness=3] (p12)
                            ;
                    \node[circle,inner sep=0pt] (p13) at (0, 8)
                    {\begin{tikzpicture}[scale=0.4]
                    \node at (0.9, 0.8) [draw, fill=black] {$.$};
                     \end{tikzpicture}}
                            edge[-latex, loop, out=170, in=110, min distance=60, looseness=2] (p13)
                            edge[-latex, loop, out=180, in=100, min distance=90, looseness=3] (p13)
                            edge[-latex, loop, out=280, in=340, min distance=60, dashed, looseness=2] (p13)
                            edge[-latex, loop, out=270, in=350, min distance=90, dashed, looseness=3] (p13)
                            ;
                    \node[circle,inner sep=0pt] (p14) at (0, 12)
                    {\begin{tikzpicture}[scale=0.4]
                    \node at (0.9, 0.8) [draw, fill=black] {$.$};
                     \end{tikzpicture}}
                            edge[-latex, loop, out=170, in=110, min distance=60, looseness=2] (p14)
                            edge[-latex, loop, out=180, in=100, min distance=90, looseness=3] (p14)
                            edge[-latex, loop, out=280, in=340, min distance=60, dashed, looseness=2] (p14)
                            edge[-latex, loop, out=270, in=350, min distance=90, dashed, looseness=3] (p14)
                            ;
                     \draw[style=semithick, dashed, -latex] (p12.south)--(p11.north);
                     \draw[style=semithick, dashed, -latex] (p13.south)--(p12.north);
                     \draw[style=semithick, dashed, -latex] (p14.south)--(p13.north);
                    \node[circle,inner sep=0pt] (p21) at (4, 0)
                    {\begin{tikzpicture}[scale=0.4]
                    \node at (0.9, 0.8) [draw, fill=black] {$.$};
                     \end{tikzpicture}}
                            edge[-latex, loop, out=170, in=110, min distance=60, looseness=2] (p21)
                            edge[-latex, loop, out=180, in=100, min distance=90, looseness=3] (p21)
                            edge[-latex, loop, out=280, in=340, min distance=60, dashed, looseness=2] (p21)
                            edge[-latex, loop, out=270, in=350, min distance=90, dashed, looseness=3] (p21)
                            ;
                    \node[circle,inner sep=0pt] (p22) at (4, 4)
                    {\begin{tikzpicture}[scale=0.4]
                    \node at (0.9, 0.8) [draw, fill=black] {$.$};
                     \end{tikzpicture}}
                            edge[-latex, loop, out=170, in=110, min distance=60, looseness=2] (p22)
                            edge[-latex, loop, out=180, in=100, min distance=90, looseness=3] (p22)
                            edge[-latex, loop, out=280, in=340, min distance=60, dashed, looseness=2] (p22)
                            edge[-latex, loop, out=270, in=350, min distance=90, dashed, looseness=3] (p22)
                            ;
                    \node[circle,inner sep=0pt] (p23) at (4, 8)
                    {\begin{tikzpicture}[scale=0.4]
                    \node at (0.9, 0.8) [draw, fill=black] {$.$};
                     \end{tikzpicture}}
                            edge[-latex, loop, out=170, in=110, min distance=60, looseness=2] (p23)
                            edge[-latex, loop, out=180, in=100, min distance=90, looseness=3] (p23)
                            edge[-latex, loop, out=280, in=340, min distance=60, dashed, looseness=2] (p23)
                            edge[-latex, loop, out=270, in=350, min distance=90, dashed, looseness=3] (p23)
                            ;
                    \node[circle,inner sep=0pt] (p24) at (4, 12)
                    {\begin{tikzpicture}[scale=0.4]
                    \node at (0.9, 0.8) [draw, fill=black] {$.$};
                     \end{tikzpicture}}
                            edge[-latex, loop, out=170, in=110, min distance=60, looseness=2] (p24)
                            edge[-latex, loop, out=180, in=100, min distance=90, looseness=3] (p24)
                            edge[-latex, loop, out=280, in=340, min distance=60, dashed, looseness=2] (p24)
                            edge[-latex, loop, out=270, in=350, min distance=90, dashed, looseness=3] (p24)
                            ;
                     \draw[style=semithick, dashed, -latex] (p22.south)--(p21.north);
                     \draw[style=semithick, dashed, -latex] (p23.south)--(p22.north);
                     \draw[style=semithick, dashed, -latex] (p24.south)--(p23.north);
                     \draw[style=semithick, -latex] (p21.west)--(p11.east);
                    \draw[style=semithick, -latex] (p22.west)--(p12.east);
                    \draw[style=semithick, -latex] (p23.west)--(p13.east);
                    \draw[style=semithick, -latex] (p24.west)--(p14.east);
                    \node[circle,inner sep=0pt] (p31) at (8, 0)
                    {\begin{tikzpicture}[scale=0.4]
                    \node at (0.9, 0.8) [draw, fill=black] {$.$};
                     \end{tikzpicture}}
                            edge[-latex, loop, out=170, in=110, min distance=60, looseness=2] (p31)
                            edge[-latex, loop, out=180, in=100, min distance=90, looseness=3] (p31)
                            edge[-latex, loop, out=280, in=340, min distance=60, dashed, looseness=2] (p31)
                            edge[-latex, loop, out=270, in=350, min distance=90, dashed, looseness=3] (p31)
                            ;
                    \node[circle,inner sep=0pt] (p32) at (8, 4)
                    {\begin{tikzpicture}[scale=0.4]
                    \node at (0.9, 0.8) [draw, fill=black] {$.$};
                     \end{tikzpicture}}
                            edge[-latex, loop, out=170, in=110, min distance=60, looseness=2] (p32)
                            edge[-latex, loop, out=180, in=100, min distance=90, looseness=3] (p32)
                            edge[-latex, loop, out=280, in=340, min distance=60, dashed, looseness=2] (p32)
                            edge[-latex, loop, out=270, in=350, min distance=90, dashed, looseness=3] (p32)
                            ;
                    \node[circle,inner sep=0pt] (p33) at (8, 8)
                    {\begin{tikzpicture}[scale=0.4]
                    \node at (0.9, 0.8) [draw, fill=black] {$.$};
                     \end{tikzpicture}}
                            edge[-latex, loop, out=170, in=110, min distance=60, looseness=2] (p33)
                            edge[-latex, loop, out=180, in=100, min distance=90, looseness=3] (p33)
                            edge[-latex, loop, out=280, in=340, min distance=60, dashed, looseness=2] (p33)
                            edge[-latex, loop, out=270, in=350, min distance=90, dashed, looseness=3] (p33)
                            ;
                    \node[circle,inner sep=0pt] (p34) at (8, 12)
                    {\begin{tikzpicture}[scale=0.4]
                    \node at (0.9, 0.8) [draw, fill=black] {$.$};
                     \end{tikzpicture}}
                            edge[-latex, loop, out=170, in=110, min distance=60, looseness=2] (p34)
                            edge[-latex, loop, out=180, in=100, min distance=90, looseness=3] (p34)
                            edge[-latex, loop, out=280, in=340, min distance=60, dashed, looseness=2] (p34)
                            edge[-latex, loop, out=270, in=350, min distance=90, dashed, looseness=3] (p34)
                            ;
                     \draw[style=semithick, dashed, -latex] (p32.south)--(p31.north);
                     \draw[style=semithick, dashed, -latex] (p33.south)--(p32.north);
                     \draw[style=semithick, dashed, -latex] (p34.south)--(p33.north);
                     \draw[style=semithick, -latex] (p31.west)--(p21.east);
                    \draw[style=semithick, -latex] (p32.west)--(p22.east);
                    \draw[style=semithick, -latex] (p33.west)--(p23.east);
                    \draw[style=semithick, -latex] (p34.west)--(p24.east);
                    \node[circle,inner sep=0pt] (p41) at (12, 0)
                    {\begin{tikzpicture}[scale=0.4]
                    \node at (0.9, 0.8) [draw, fill=black] {$.$};
                     \end{tikzpicture}}
                            edge[-latex, loop, out=170, in=110, min distance=60, looseness=2] (p41)
                            edge[-latex, loop, out=180, in=100, min distance=90, looseness=3] (p41)
                            edge[-latex, loop, out=280, in=340, min distance=60, dashed, looseness=2] (p41)
                            edge[-latex, loop, out=270, in=350, min distance=90, dashed, looseness=3] (p41)
                            ;
                    \node[circle,inner sep=0pt] (p42) at (12, 4)
                    {\begin{tikzpicture}[scale=0.4]
                    \node at (0.9, 0.8) [draw, fill=black] {$.$};
                     \end{tikzpicture}}
                            edge[-latex, loop, out=170, in=110, min distance=60, looseness=2] (p42)
                            edge[-latex, loop, out=180, in=100, min distance=90, looseness=3] (p42)
                            edge[-latex, loop, out=280, in=340, min distance=60, dashed, looseness=2] (p42)
                            edge[-latex, loop, out=270, in=350, min distance=90, dashed, looseness=3] (p42)
                            ;
                    \node[circle,inner sep=0pt] (p43) at (12, 8)
                    {\begin{tikzpicture}[scale=0.4]
                    \node at (0.9, 0.8) [draw, fill=black] {$.$};
                     \end{tikzpicture}}
                            edge[-latex, loop, out=170, in=110, min distance=60, looseness=2] (p43)
                            edge[-latex, loop, out=180, in=100, min distance=90, looseness=3] (p43)
                            edge[-latex, loop, out=280, in=340, min distance=60, dashed, looseness=2] (p43)
                            edge[-latex, loop, out=270, in=350, min distance=90, dashed, looseness=3] (p43)
                            ;
                    \node[circle,inner sep=0pt] (p44) at (12, 12)
                    {\begin{tikzpicture}[scale=0.4]
                    \node at (0.9, 0.8) [draw, fill=black] {$.$};
                     \end{tikzpicture}}
                            edge[-latex, loop, out=170, in=110, min distance=60, looseness=2] (p44)
                            edge[-latex, loop, out=180, in=100, min distance=90, looseness=3] (p44)
                            edge[-latex, loop, out=280, in=340, min distance=60, dashed, looseness=2] (p44)
                            edge[-latex, loop, out=270, in=350, min distance=90, dashed, looseness=3] (p44)
                            ;
                     \draw[style=semithick, dashed, -latex] (p42.south)--(p41.north);
                     \draw[style=semithick, dashed, -latex] (p43.south)--(p42.north);
                     \draw[style=semithick, dashed, -latex] (p44.south)--(p43.north);
                     \draw[style=semithick, -latex] (p41.west)--(p31.east);
                    \draw[style=semithick, -latex] (p42.west)--(p32.east);
                    \draw[style=semithick, -latex] (p43.west)--(p33.east);
                    \draw[style=semithick, -latex] (p44.west)--(p34.east);
                    \node[circle,inner sep=0pt] (p51) at (16, 0)
                    {\begin{tikzpicture}[scale=0.4]
                    \node at (0.9, 0.8) [draw, fill=black] {$.$};
                     \end{tikzpicture}}
                            edge[-latex, loop, out=170, in=110, min distance=60, looseness=2] (p51)
                            edge[-latex, loop, out=180, in=100, min distance=90, looseness=3] (p51)
                            edge[-latex, loop, out=280, in=340, min distance=60, dashed, looseness=2] (p51)
                            edge[-latex, loop, out=270, in=350, min distance=90, dashed, looseness=3] (p51)
                            ;
                    \node[circle,inner sep=0pt] (p52) at (16, 4)
                    {\begin{tikzpicture}[scale=0.4]
                    \node at (0.9, 0.8) [draw, fill=black] {$.$};
                     \end{tikzpicture}}
                            edge[-latex, loop, out=170, in=110, min distance=60, looseness=2] (p52)
                            edge[-latex, loop, out=180, in=100, min distance=90, looseness=3] (p52)
                            edge[-latex, loop, out=280, in=340, min distance=60, dashed, looseness=2] (p52)
                            edge[-latex, loop, out=270, in=350, min distance=90, dashed, looseness=3] (p52)
                            ;
                    \node[circle,inner sep=0pt] (p53) at (16, 8)
                    {\begin{tikzpicture}[scale=0.4]
                    \node at (0.9, 0.8) [draw, fill=black] {$.$};
                     \end{tikzpicture}}
                            edge[-latex, loop, out=170, in=110, min distance=60, looseness=2] (p53)
                            edge[-latex, loop, out=180, in=100, min distance=90, looseness=3] (p53)
                            edge[-latex, loop, out=280, in=340, min distance=60, dashed, looseness=2] (p53)
                            edge[-latex, loop, out=270, in=350, min distance=90, dashed, looseness=3] (p53)
                            ;
                    \node[circle,inner sep=0pt] (p54) at (16, 12)
                    {\begin{tikzpicture}[scale=0.4]
                    \node at (0.9, 0.8) [draw, fill=black] {$.$};
                     \end{tikzpicture}}
                            edge[-latex, loop, out=170, in=110, min distance=60, looseness=2] (p54)
                            edge[-latex, loop, out=180, in=100, min distance=90, looseness=3] (p54)
                            edge[-latex, loop, out=280, in=340, min distance=60, dashed, looseness=2] (p54)
                            edge[-latex, loop, out=270, in=350, min distance=90, dashed, looseness=3] (p54)
                            ;
                     \draw[style=semithick, dashed, -latex] (p52.south)--(p51.north);
                     \draw[style=semithick, dashed, -latex] (p53.south)--(p52.north);
                     \draw[style=semithick, dashed, -latex] (p54.south)--(p53.north);
                     \draw[style=semithick, -latex] (p51.west)--(p41.east);
                    \draw[style=semithick, -latex] (p52.west)--(p42.east);
                    \draw[style=semithick, -latex] (p53.west)--(p43.east);
                    \draw[style=semithick, -latex] (p54.west)--(p44.east);
                    \node[circle,inner sep=0pt] (p61) at (20, 0)
                    {\begin{tikzpicture}[scale=0.4]
                    \node at (0.9, 0.8) [draw, fill=black] {$.$};
                     \end{tikzpicture}}
                            edge[-latex, loop, out=170, in=110, min distance=60, looseness=2] (p61)
                            edge[-latex, loop, out=180, in=100, min distance=90, looseness=3] (p61)
                            edge[-latex, loop, out=280, in=340, min distance=60, dashed, looseness=2] (p61)
                            edge[-latex, loop, out=270, in=350, min distance=90, dashed, looseness=3] (p61)
                            ;
                    \node[circle,inner sep=0pt] (p62) at (20, 4)
                    {\begin{tikzpicture}[scale=0.4]
                    \node at (0.9, 0.8) [draw, fill=black] {$.$};
                     \end{tikzpicture}}
                            edge[-latex, loop, out=170, in=110, min distance=60, looseness=2] (p62)
                            edge[-latex, loop, out=180, in=100, min distance=90, looseness=3] (p62)
                            edge[-latex, loop, out=280, in=340, min distance=60, dashed, looseness=2] (p62)
                            edge[-latex, loop, out=270, in=350, min distance=90, dashed, looseness=3] (p62)
                            ;
                    \node[circle,inner sep=0pt] (p63) at (20, 8)
                    {\begin{tikzpicture}[scale=0.4]
                    \node at (0.9, 0.8) [draw, fill=black] {$.$};
                     \end{tikzpicture}}
                            edge[-latex, loop, out=170, in=110, min distance=60, looseness=2] (p63)
                            edge[-latex, loop, out=180, in=100, min distance=90, looseness=3] (p63)
                            edge[-latex, loop, out=280, in=340, min distance=60, dashed, looseness=2] (p63)
                            edge[-latex, loop, out=270, in=350, min distance=90, dashed, looseness=3] (p63)
                            ;
                    \node[circle,inner sep=0pt] (p64) at (20, 12)
                    {\begin{tikzpicture}[scale=0.4]
                    \node at (0.9, 0.8) [draw, fill=black] {$.$};
                     \end{tikzpicture}}
                            edge[-latex, loop, out=170, in=110, min distance=60, looseness=2] (p64)
                            edge[-latex, loop, out=180, in=100, min distance=90, looseness=3] (p64)
                            edge[-latex, loop, out=280, in=340, min distance=60, dashed, looseness=2] (p64)
                            edge[-latex, loop, out=270, in=350, min distance=90, dashed, looseness=3] (p64)
                            ;
                     \draw[style=semithick, dashed, -latex] (p62.south)--(p61.north);
                     \draw[style=semithick, dashed, -latex] (p63.south)--(p62.north);
                     \draw[style=semithick, dashed, -latex] (p64.south)--(p63.north);
                     \draw[style=semithick, -latex] (p61.west)--(p51.east);
                    \draw[style=semithick, -latex] (p62.west)--(p52.east);
                    \draw[style=semithick, -latex] (p63.west)--(p53.east);
                    \draw[style=semithick, -latex] (p64.west)--(p54.east);
                    \node at (-1.5,0) {.}; \node at (-1,0) {.}; \node at (-0.5,0) {.};
                     \node at (-1.5,4) {.}; \node at (-1,4) {.}; \node at (-0.5,4) {.};
                     \node at (-1.5,8) {.}; \node at (-1,8) {.}; \node at (-0.5,8) {.};
                     \node at (-1.5,12) {.}; \node at (-1,12) {.}; \node at (-0.5,12) {.};
                     \node at (21,0) {.}; \node at (21.5,0) {.}; \node at (22,0) {.};
                    \node at (21,4) {.}; \node at (21.5,4) {.}; \node at (22,4) {.};
                    \node at (21,8) {.}; \node at (21.5,8) {.}; \node at (22,8) {.};
                    \node at (21,12) {.}; \node at (21.5,12) {.}; \node at (22,12) {.};
                     \node at (0,14) {.}; \node at (0,14.5) {.}; \node at (0,15) {.};
                     \node at (4,14) {.}; \node at (4,14.5) {.}; \node at (4,15) {.};
                     \node at (8,14) {.}; \node at (8,14.5) {.}; \node at (8,15) {.};
                     \node at (12,14) {.}; \node at (12,14.5) {.}; \node at (12,15) {.};
                    \node at (16,14) {.}; \node at (16,14.5) {.}; \node at (16,15) {.};
                    \node at (20,14) {.}; \node at (20,14.5) {.}; \node at (20,15) {.};
                    \node at (0,-0.5) {.}; \node at (0,-1) {.}; \node at (0,-1.5) {.};
                    \node at (4,-0.5) {.}; \node at (4,-1) {.}; \node at (4,-1.5) {.};
                    \node at (8,-0.5) {.}; \node at (8,-1) {.}; \node at (8,-1.5) {.};
                    \node at (12,-0.5) {.}; \node at (12,-1) {.}; \node at (12,-1.5) {.};
                    \node at (16,-0.5) {.}; \node at (16,-1) {.}; \node at (16,-1.5) {.};
                    \node at (20,-0.5) {.}; \node at (20,-1) {.}; \node at (20,-1.5) {.};
\end{tikzpicture}
\]

\noindent
To describe the topology of $\text{Prim}\;C^{\ast}(\Lambda)$ we first identify the vertices of $\Lambda$ with $\mathbb{Z}^2$. It is tedious, but not difficult to show
that every maximal tail in $\chi_\Lambda$ is of the form depicted in Figures 1,2,3, plus the tail $\Lambda^0$. In Figure 1, we have
a tail of the form
\[
H_{b} = \{ (x,y) \in \mathbb{Z}^2 : y \le b \} \text{ for each } b\in\mathbb{Z} .
\]

\noindent In Figure 2, we we have a tail of the form
\[
V_{a} = \{(x,y) \in\mathbb{Z}^2 : x \le a \} \text{ for each }
a \in \mathbb{Z} .
\]

\noindent In Figure 3, we have a tail of the form
\[
LQ_{(a,b)} = \{ (x,y) \in \mathbb{Z}^2 : x \le a , y \le b \} \text{ for each }
a,b \in \mathbb{Z} .
\]

\noindent So
\[
\chi_\Lambda = \{ V_{a} : a \in \mathbb{Z} \}
\cup \{ H_{b} : b \in \mathbb{Z} \} \cup \{ LQ_{(a,b)} : (a,b) \in \mathbb{Z}^2 \} \cup \{ \Lambda^0 \}.
\]

\noindent
Recall that the closed set $\overline{S}$ of $\chi_{\Lambda}$ given in Theorem \ref{topMT} is $\overline{S}=\{\delta\in\chi_{\Lambda}:\delta \subseteq \bigcup_{\gamma\in S}\gamma\}$. Hence
\begin{align*}
\overline{\{ H_b \}} = \{ H_{d} : d \le b \} , \qquad &  \overline{\{LQ_{(a,b)}\}} = \{ LQ_{(c,d)} : (c,d) \le (a,b) \} , \\
\overline{\{V_c \}} = \{  V_a : c \le a \}, \qquad & \overline{\{ \Lambda^0 \}} = \chi_\Lambda ,
\end{align*}

\noindent so $\Lambda^0$ is a dense point in $\chi_\Lambda$. Since 
\[
\overline{ \{ LQ_{(a,b)} \} }^c = \overline{ \{ V_a \} }^c \cup \overline{ \{ H_b \} }^c = \{ V_c : c > a \} \cup \{ H_{d} : d > b \}  \cup \{ \Lambda^0 \} , 
\]

\noindent a basis for the topology on $\chi_\Lambda$ is
\[
\left\{ \{ H_{d} : d > b \} \cup \{ \Lambda^0 \} : b \in \mathbb{Z}\right \} , \left\{  \{ V_c : c > a \} \cup \{ \Lambda^0 \} : c \in \mathbb{Z} \right\} .
\]

\noindent
Let $X = \mathbb{Z}\cup\{\infty\}$ and extend the order on $\mathbb{Z}$ to $X$ by declaring $x\le \infty$ for all $x\in\mathbb{Z}$. Then $X$ is totally ordered and has right-order topology with basis $(a,\infty]=\{x\in\mathbb{Z}: x> a\}$. Let $Y= X \times X$ with the product topology. Define a map $\phi : \chi_\Lambda \to Y$ by
\[
\phi ( H_b ) = ( \infty , b ) , \  \phi ( V_a ) =(a , \infty), \ \phi ( LQ_{(a,b)} ) = (a,b)  , \
\phi ( \Lambda^0 ) = ( \infty , \infty ) ,
\]

\noindent then $\phi$ is a bijection such that
\[
\phi \left(  \{ H_{d} : d > b \} \cup \{ \Lambda^0 \} \right) = X \times (b,\infty)
\text{ and } \phi \left( \{ V_c : c > a \} \cup \{ \Lambda^0 \} \right) = (a,\infty) \times X ,
\]

\noindent and so it is a homeomorphism. Therefore, the topology of $\text{Prim}\;C^{\ast}(\Lambda)$ corresponds to the right-order topology of $(\mathbb{Z}\cup\{\infty\})\times(\mathbb{Z}\cup\{\infty\})$.
\end{example}

\begin{example} \label{ex2}
Let $\Omega$ be the $1$ graph with the following $1$-skeleton.
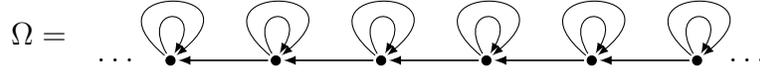
\[\begin{tikzpicture}[scale=0.35]
                    \node at (-5,1) {$\Omega=$};
                    \node at (-2,0) {$\dots$};
                    \node at (22,0) {$\dots$};
                    \node[circle,inner sep=0pt] (p11) at (0, 0)
                    {\begin{tikzpicture}[scale=0.4]
                    \node at (0.9, 0.8) [draw, fill=black] {$.$};
                     \end{tikzpicture}}
                            edge[-latex, loop, out=120, in=60, min distance=60, looseness=2] (p11)
                            edge[-latex, loop, out=140, in=40, min distance=120, looseness=2.2] (p11);
                    \node[circle,inner sep=0pt] (p12) at (4, 0)
                    {\begin{tikzpicture}[scale=0.4]
                    \node at (0.9, 0.8) [draw, fill=black] {$.$};
                     \end{tikzpicture}}
                            edge[-latex, loop, out=120, in=60, min distance=60, looseness=2] (p12)
                            edge[-latex, loop, out=140, in=40, min distance=120, looseness=2.2] (p12);
                     \node[circle,inner sep=0pt] (p13) at (8, 0)
                    {\begin{tikzpicture}[scale=0.4]
                    \node at (0.9, 0.8) [draw, fill=black] {$.$};
                     \end{tikzpicture}}
                            edge[-latex, loop, out=120, in=60, min distance=60, looseness=2] (p13)
                            edge[-latex, loop, out=140, in=40, min distance=120, looseness=2.2] (p13);
                    \node[circle,inner sep=0pt] (p14) at (12, 0)
                    {\begin{tikzpicture}[scale=0.4]
                    \node at (0.9, 0.8) [draw, fill=black] {$.$};
                     \end{tikzpicture}}
                            edge[-latex, loop, out=120, in=60, min distance=60, looseness=2] (p14)
                            edge[-latex, loop, out=140, in=40, min distance=120, looseness=2.2] (p14);
                     \node[circle,inner sep=0pt] (p15) at (16, 0)
                    {\begin{tikzpicture}[scale=0.4]
                    \node at (0.9, 0.8) [draw, fill=black] {$.$};
                     \end{tikzpicture}}
                            edge[-latex, loop, out=120, in=60, min distance=60, looseness=2] (p15)
                            edge[-latex, loop, out=140, in=40, min distance=120, looseness=2.2] (p15);
                    \node[circle,inner sep=0pt] (p16) at (20, 0)
                    {\begin{tikzpicture}[scale=0.4]
                    \node at (0.9, 0.8) [draw, fill=black] {$.$};
                     \end{tikzpicture}}
                            edge[-latex, loop, out=120, in=60, min distance=60, looseness=2] (p16)
                            edge[-latex, loop, out=140, in=40, min distance=120, looseness=2.2] (p16);
                     \draw[style=semithick, -latex] (p12.west)--(p11.east);
                     \draw[style=semithick, -latex] (p13.west)--(p12.east);
                     \draw[style=semithick, -latex] (p14.west)--(p13.east);
                     \draw[style=semithick, -latex] (p15.west)--(p14.east);
                     \draw[style=semithick, -latex] (p16.west)--(p15.east);
\end{tikzpicture}\]

\noindent  Then the 1-skeleton of $\Omega\times\Omega$ is the same as the 1-skeleton of $\Lambda$ in the above example. Since $\Omega$ satisfies condition (K), it follows that $\Omega$ is strongly aperiodic by Lemma~\ref{lem:KisSA}. Then by Theorem~\ref{prod} (b), $\Omega\times\Omega$ is strongly aperiodic. If we identify the vertices of $\Omega$ with $\mathbb{Z}$, it is straightforward to see that every saturated hereditary subset of $\Omega^0$
is of the form $[n+1,+\infty)$ for some $n \in \mathbb{Z}$, or the empty set $\emptyset$, or $\Omega^0$. Similarly every maximal tail is
of the form $ \chi_n := ( - \infty , n ]$  for some $n \in \mathbb{Z}$, or $\Omega^0$ (recall that a maximal tail is nonempty). So
$\chi_\Omega = \{ \chi_n : n \in \mathbb{Z} \} \cup \{ \Omega^0 \}$ which we may identify (as a set) with $\mathbb{Z} \cup \{ \infty \}$.
The closure of $\{ \chi_n \}$ is $\{ \chi_j : j \le n \}$ and the closure of $\Omega^0$ is $\chi_\Omega$. Thus, the nontrivial open sets in
$\chi_\Omega$ are of the form $\{ \chi_j : j > n \} \cup \{ \Omega^0 \}$. When we identify $\chi_\Omega$ with
$\mathbb{Z} \cup \{ \infty \}$, this topology coincides with the right-order topology on $\mathbb{Z} \cup \{ \infty \}$.
By \cite[Corollary 3.5 (iv)]{KP1} we have $C^* ( \Omega \times \Omega ) \cong C^* ( \Omega ) \otimes C^* ( \Omega )$, so by \cite{W}
the primitive ideal space of $C^* ( \Omega \times \Omega )$ is the cartesian product of the primitive ideal space of $C^* ( \Omega )$
with itself with the product topology.
\end{example}

\begin{remark}
The primitive ideal space of $C^* ( \Omega \times \Omega)$ is precisely the same topology as the primitive ideal space
of  $C^* ( \Lambda )$ described in Example~\ref{ex1}. Even though $\Lambda$ and $\Omega$ are not isomorphic,
it is unclear whether $C^* ( \Lambda )$ and $C^* ( \Omega \times \Omega )$ are isomorphic.
\end{remark}


\begin{thebibliography}{00}

\bibitem{ahr} A. an Huef and I. Raeburn, \textit{The ideal structure of {C}untz-{K}rieger algebras}, Ergod. Th. and Dyn. Sys., {\bf 17} (1997),
611--624.

\bibitem{B}{T. Bates}, \textit{On the primitive ideal spaces of the $C^{\ast}$-algebras of graphs}, Math. Proc. Camb. Phil. Soc. {\bf 139} (2005), no. 3, 427--439.

\bibitem{BPRS}{T. Bates, D. Pask, I. Raeburn, and W. Szyma$\acute{\text{n}}$ski}, \textit{The $C^{\ast}$-algebras of row-finite graphs}, New York J. Math. {\bf 6} (2000), 307--324.

\bibitem{BHRS}{T. Bates, J. Hong, I. Raeburn and W. Szyma$\acute{\text{n}}$ski}, \textit{The ideal structure of the $C^{\ast}$-algberas of infinite graphs}, Illinois J. Math. {\bf 46} (2002), no. 4, 1159--1176.

\bibitem{BE}{O. Bratteli, and G. Elliott}, \textit{Structure spaces of approximately finite-dimensional $C^{\ast}$-algebras II}, J. Funct. Anal. {\bf 30}  (1978), 74--82.

\bibitem{BR} N. Brownlowe, \textit{Realising the $C^*$-algebra of a higher rank graph as an Exel crossed product}, to appear J. Operator Theory.

\bibitem{C} T.M. Carlsen, \emph{Cuntz-{P}imsner {$C\sp *$}-algebras associated with subshifts}, Internat.
    J. Math. \textbf{19} (2008),  47--70.

\bibitem{Cu} J. Cuntz, {\em A class of $C^*$-algebras and topological Markov chains II: reducible chains and the Ext-functor for $C^*$-algebras},
   Invent. Math. {\bf 63} (1981), 25--40.

\bibitem{CK} J. Cuntz and W. Krieger,{\em A class of $C^*$-algebras and topological  Markov chains}, Invent. Math. {\bf 56} (1980), 251--268.

\bibitem{DPY} K. R. Davidson, S. C. Power and D. Yang, \textit{Atomic representations of rank 2 graph algebras}, J. Funct. Anal. \textbf{255} (2008), 819--853.

\bibitem{DY} K.R.\ Davidson and D.\ Yang, \textit{Periodicity in rank 2 graph algebras}, Canad. J. Math. {\bf 61} (2009), no. 6, 1239�1261.

\bibitem{Do} A.H.\ Dooley, \textit{The spectral theory of posets and its applications to $C^*$-algebras}, Trans.\ Amer.\ Math.\ Soc.\ {\bf 224} (1976) 143--155.

\bibitem{D}{D. Drinen}, \textit{Viewing AF-algebras as graph algebras}, Proc. Amer. Math. Soc. {\bf 128} (2000), 1991--2000.

\bibitem{E} D. G. Evans, \textit{On the $K$-theory of higher rank graph $C^*$-algebras}, New York J. Math. \textbf{14} (2008), 1--31.

\bibitem{Ex} R.\ Exel. \textit{Inverse semigroups and combinatorial C*-algebras},	Bull. Braz. Math. Soc. (N.S.), \textbf{39} (2008), 191--313.

\bibitem{FPS} C. Farthing, D. Pask, and A. Sims,  \emph{Crossed products of $k$-graph $C^*$-algebras by $\mathbb{Z}^l$}, Houston
    J. Math,

\bibitem{FR} N.J.\ Fowler and I.\ Raeburn, \textit{The Toeplitz algebra of a Hilbert bimodule}, Indiana Univ. Math. J. {\bf 48} (1999),  155--181.

\bibitem{HZ}{J. Hong and W. Szyma$\acute{\text{n}}$ski}, \textit{The primitive ideal space of the $C^{\ast}$-algebras of infinite graphs}, J. Math. Soc. Japan {\bf 56} (2004), no. 1, 45--64.

\bibitem{Kat1} T.\ Katsura. \textit{A class of $C^*$-algebras generalising graph algebras and homeomorphism $C^*$-algebras I, fundamental results},
Trans.\ Math.\ Soc.\ Amer.\ {\bf 356} (2004), 4287-4322.

\bibitem{Kat2} T. Katsura, \textit{On {$C^*$}-algebras associated with {$C\sp *$}-correspondences}, J. Funct. Anal. \textbf{217} (2004),  366--401.

\bibitem{KP} D.W. Kribs and  S.C. Power,  \textit{Analytic algebras of higher rank graphs}, Mathematical Proceedings of the Royal Irish Academy, {\bf 106} (2006), 199--218

\bibitem{KP1}{A. Kumjian and D. Pask}, \textit{Higher rank graph $C^{\ast}$-algebras}, New York J. Math. {\bf 6} (2000), 1--20.

\bibitem{KP2}{A. Kumjian and D. Pask}, \textit{$C^{\ast}$-algebras of directed graphs and group actions}, Ergod. Th. \& Dynam. Sys. {\bf 19} (1999), 1503--1519.

\bibitem{KPRR}{A. Kumjian, D. Pask, Iain Raeburn and Jean Renault}, \textit{Graphs, groupoids, and Cuntz-Krieger algebras}, J. Funct. Anal. {\bf 144} (1997), 505--541

\bibitem{KPR2}{A. Kumjian, D. Pask and Iain Raeburn}, \textit{Cuntz-Krieger algebras of directed graphs}, Pacific J. Math. {\bf 184} no. 1 (1998), 161--174

\bibitem{LS}{P. Lewin and A. Sims}, \textit{Aperiodicity and cofinality for finitely aligned higher-rank graphs}, Math. Proc. Cambridge Philos. Soc. 149 (2010), no. 2, 333--350,

\bibitem{MTom} P.S.\ Muhly, and M.\ Tomforde. \textit{Topological quivers.} Internat. J. Math. {\bf 16} (2005), 693--755.

\bibitem{PQR}{D. Pask, J. Quigg, and I. Raeburn}, \textit{Covering of $k$-graphs}, J. Algebra {\bf 289} (2005), 161--191.

\bibitem{PRRS}D. Pask, I. Raeburn, M. R{\o}rdam and A. Sims, \textit{Rank-two graphs whose $C^{\ast}$-algebras are direct limits of circle
algebras}, J. Funct. Anal. \textbf{239} (2006), 137--178.

\bibitem{Pe} G.~Pedersen, C*-algebras and their automorphism groups, L.M.S.~Monographs {\bf 14}, Academic Press (1979).

\bibitem{P} M.V.\ Pimsner, \textit{A class of C*-algebras generalizing both Cuntz-Krieger algebras and crossed products by $\mathbb{Z}$}. Free probability theory (Waterloo, ON, 1995), 189--212, Fields Inst. Commun., {\bf 12}, Amer. Math. Soc., Providence, RI, 1997.

\bibitem{Po}  S.\ Power, \textit{Classifying higher rank analytic Toeplitz algebras}, New York J. Math. {\bf 13} (2007), 271--298.

\bibitem{RSY1}{I. Raeburn, A. Sims, and T. Yeend}, \textit{Higher-rank graphs and their $C^{\ast}$-algebras}, Proc. Edinb. Math. Soc. {\bf 46} (2003), 99--115.

\bibitem{RSY2}{I. Raeburn, A. Sims, and T. Yeend}, \textit{The $C^{\ast}$-algebras of finitely aligned higher-rank graphs}, J. Funct. Anal. {\bf 213} (2004), 206--240.

\bibitem{RoS1}{D. I. Robertson and A. Sims}, \textit{Simplicity of $C^{\ast}$-algebras associated to higher-rank graphs}, Bull. London Math. Soc. {\bf 39} (2007), 337--344.

\bibitem{RoS2}{D. I. Robertson and A. Sims}, \textit{Simplicity of $C^{\ast}$-algebras associated to row-finite locally convex higher-rank graphs}, Israel J. Math. {\bf 172} (2009), 171--192.

\bibitem{RS2}G. Robertson and T. Steger, \textit{Affine buildings, tiling systems and higher rank Cuntz-Krieger algebras}, J. reine
angew. Math.  \textbf{513} (1999), 115--144.

\bibitem{RW}{I. Raeburn and D. Williams}, Morita Equivalence and Continuous-Trace $C^{\ast}$-Algebras, A.M.S. Providence, 1998.

\bibitem{Sims_thesis}{A. Sims,} \textit{$C^*$-algebras associated to higher-rank graphs}, PhD.\ Thesis, University of Newcastle, 2003.

\bibitem{S}{A. Sims}, \textit{Gauge-invariant ideals in the $C^{\ast}$-algebras of finitely aligned higher-rank graphs}, Canad. J. Math. {\bf 58} (2006), no. 6, 1268--1290.

\bibitem{SZ}A. Skalski and J. Zacharias, \textit{Entropy of shifts on higher-rank graph $C^*$-algebras}, Houston J. Math. \textbf{34} (2008). 269--282.

\bibitem{SP} J. Spielberg, \emph{Graph-based models for {K}irchberg algebras}, J. Operator Theory \textbf{57}
    (2007), 347--374.

\bibitem{W}{A. Wulfsohn}, \textit{Primitive spectrum of a tensor product}, Proceedings of American Math. Soc., {\bf 19} (1968), 1094-1096.

\bibitem{Ya} D.\ Yang, \textit{Type III von Neumann algebras associated with $\mathcal{O}_\theta$}, Preprint, University of Windsor. \newline {\tt http://arxiv.org/abs/1104.5697}.

\bibitem{Y} T.\ Yeend. \textit{Groupoid models for the C*-algebras of topological higher-rank graphs}, J. Operator Theory {\bf 57} (2007), 95--120.

\end{thebibliography}
\end{document}